\documentclass[11pt, a4paper,leqno]{amsart}
\usepackage{amsmath,amsthm,amscd,amssymb,amsfonts, amsbsy}
\usepackage{latexsym}
\usepackage{txfonts}
\usepackage{exscale}
\usepackage{mathrsfs}
\usepackage[titletoc]{appendix}
\usepackage{huncial}
\usepackage[T1]{fontenc}

\renewcommand\labelenumi{(\roman{enumi})}
\renewcommand\theenumi\labelenumi

\usepackage[colorlinks,citecolor=red,pagebackref,hypertexnames=false, breaklinks]{hyperref}
\usepackage[usenames,dvipsnames]{color}

\newcommand{\abs}[1]{{\left\lvert{#1}\right\rvert}}
\newcommand{\norm}[1]{{\left\lVert{#1}\right\rVert}}
\newcommand{\br}[1]{{\left({#1}\right)}}

\calclayout
\allowdisplaybreaks

\def \Int{\,\Int\,}
\def \iint{\int\!\!\!\int}

\def\Xint#1{\mathchoice
    {\XXint\displaystyle\textstyle{#1}}%
    {\XXint\textstyle\scriptstyle{#1}}%
    {\XXint\scriptstyle\scriptscriptstyle{#1}}%
    {\XXint\scriptscriptstyle\scriptscriptstyle{#1}}%
    \!\int}
    \def\XXint#1#2#3{{\setbox0=\hbox{$#1{#2#3}{\int}$}
    \vcenter{\hbox{$#2#3$}}\kern-.5\wd0}}
    \def\fint{\Xint-}

\def\Xint#1{\mathchoice
    {\XXint\displaystyle\textstyle{#1}}%
    {\XXint\textstyle\scriptstyle{#1}}%
    {\XXint\scriptstyle\scriptscriptstyle{#1}}%
    {\XXint\scriptscriptstyle\scriptscriptstyle{#1}}%
    \!\int}
    \def\XXint#1#2#3{{\setbox0=\hbox{$#1{#2#3}{\int}$}
    \vcenter{\hbox{$#2#3$}}\kern-.5\wd0}}
    \def\dashint{\Xint-}

\def\div{\mathop{\rm div}}

\def \N{ \mathbb{N} }
\def \R{ \mathbb{R} }

\def \Ncal { \mathcal{N} }

\def \Ocal{ \mathcal{O} }
\def \Scal{ \mathcal{S} }

\def \hh{ \mathrm{H} }
\def \pp{ \mathrm{P} }
\def \Grm{ \mathrm{G} }

\newcommand{\loc}{\text{{\rm loc}}}

\renewcommand{\Re}{{\rm Re}\,}

\def \re{ \mathbb{R} }

\def\esssup{\mathop\mathrm{\,ess\,sup\,}}
\renewcommand{\Int}{{\rm Int}\,}

\DeclareMathOperator{\supp}{supp}

\renewcommand{\chi}{{\bf 1}}

\theoremstyle{plain}
\newtheorem{theorem}[equation]{Theorem}
\newtheorem{lemma}[equation]{Lemma}
\newtheorem{corollary}[equation]{Corollary}
\newtheorem{proposition}[equation]{Proposition}

\theoremstyle{definition}
\newtheorem{definition}[equation]{Definition}

\theoremstyle{remark}
\newtheorem{remark}[equation]{Remark}

\numberwithin{equation}{section}

\numberwithin{equation}{section}

\numberwithin{equation}{section}

\begin{document}
\allowdisplaybreaks
\author{Cruz Prisuelos Arribas}
\address{Cruz Prisuelos Arribas
\\
Instituto de Ciencias Matem\'aticas CSIC-UAM-UC3M-UCM
\\
Consejo Superior de Investigaciones Cient{\'\i}ficas
\\
C/ Nicol\'as Cabrera, 13-15
\\
E-28049 Madrid, Spain} \email{cruz.prisuelos@icmat.es}

\title[Elliptic operators]{Vertical square functions and other operators associated with  an elliptic operator
\\[0.3cm]
{\small }}

\thanks{The research leading to these results has received funding from the European Research Council under the European Union's Seventh Framework Programme (FP7/2007-2013)/ ERC agreement no. 615112 HAPDEGMT. The author acknowledges financial support from the Spanish Ministry of Economy and Competitiveness, through the ``Severo Ochoa
Programme for Centres of Excellence in  R\& \!\!D'' (SEV-2015-0554).}

\date{\today}


\begin{abstract}
We study the vertical and conical square functions defined via elliptic operators in divergence form.
In general, vertical and conical square functions are  equivalent operators just in $L^2$. But when this square functions are defined through the heat or Poisson semigroup that arise from an elliptic operator, we are able to find open intervals containing $2$ where the equivalence holds. The intervals in question depend ultimately on the range where the semigroup is uniformly bounded or has off-diagonal estimates. As a consequence we obtain  new boundedness results for some square functions.
Besides, we consider a non-tangential maximal function associated with the Poisson semigroup and extend  the known range where that operator is bounded. Our methods are based on the use of extrapolation for Muckenhoupt weights and change of angle estimates. All our results are obtained in the general setting of  a degenerate elliptic operator, where the degeneracy is given by an $A_2$ weight, in weighted Lebesgue spaces. Of course they are valid in the unweighted and/or non-degenerate situations, which can be seen as special cases, and provide new results even in those particular settings.

We also consider the square root of a degenerate elliptic operator in divergence form $L_w$ and improve the lower bound of the interval where this operator is known to be bounded on $L^p(vdw)$.  Finally, we give unweighted boundedness results for the degenerate operators under consideration.
\end{abstract}

\maketitle

\tableofcontents

\bigskip
\section{Introduction}
The study of the operators that arise from an elliptic operators $L$ in divergence form has been of great interest  specially after the solution of the Kato problem in \cite{AHLMT02}. In \cite{Au07} P. Auscher
developed a complete study of  
the boundedness and off-diagonal estimates of the  heat semigroup  generated by $L$, as well as its gradient. He also proved boundedness and other norm inequalities for the square root of $L$, the Riesz transform, and two representative vertical functions.
 These results were extended to weighted Lebesgue spaces for Muckenhoupt weights
 by P. Auscher and J.M. Martell in \cite{AMIII06,AMI07,AMII07}. Recently, in \cite{CMR15}, D. Cruz-Uribe, J.M. Martell, and C. Rios
 carried on a similar study for degenerate elliptic operators which are defined by introducing a degeneracy in the elliptic operator in terms of a Muckenhoupt weight $w\in A_2$.  In \cite{CUR15}, D. Cruz-Uribe and C. Rios considered these degenerate elliptic operators and solved the Kato square root problem  under
the assumption that the associated heat kernel satisfies classic Gaussian upper bounds.
 
 Coming back to the square function associated with the operator $L$, in \cite{AHM12}, P. Auscher, S. Hofmann, and J.M. Martell studied boundedness in weighted Lebesgue spaces for vertical and conical square functions
associated  with the gradient of the heat and Poisson semigroups generated by $L$. Besides, they
 showed how the vertical and conical operators are  related in general. Specifically, they showed  that the norm on $L^p(w)$ of the conical operator controls the norm on $L^p(w)$ of the vertical one for all $0< p<2$ and $w\in RH_{\left(2/p\right)'}$, and the reverse inequality holds for all  $2<p<\infty$ and $w\in A_{p/2}$. In particular, we have  that these operators are equivalent on $L^2(\R^n)$.
 
Furthermore, in \cite{MaPAI17} J.M. Martell and the author of this paper studied weighted norm  estimates and boundedness of the conical square functions associated with the operator $L$  defined via the heat or Poisson semigroup, or their gradients. In order to obtain the boundedness of the conical square functions associated with the Poisson semigroup it was essential to compare their norms on $L^p(w)$ with the corresponding norms of the conical square functions associated with the heat semigroup.
 This work was extended for degenerate elliptic operators  in \cite{ChMPA16}. The weighted boundedness of the  non-tangential maximal functions associated with the heat and Poisson semigroup was proved in \cite{MaPAII17} in the context of Hardy spaces. 
 
 In this work our aim is to complete this theory in the most general way that has been considered so far. More precisely, given a Muckenhoupt weight $w\in A_2$ we consider a second order divergence form degenerate elliptic operator defined by
\begin{align*}
 L_w f:=-w^{-1}\div(wA\nabla f).
\end{align*}
 If the reader is interested only in the non-degenerate case, we note that all our results and proofs are valid replacing the weight $w$ with a constant equal to one. For example, when $w\equiv 1$ in the previous definition we obtain the uniformly elliptic operator $Lf=-\div(A\nabla f)$,  (see the complete definitions in Section \ref{sec:preliminaries}). 
 

 \medskip

 In the papers mentioned above the boundedness of the vertical and conical square functions were considered independently. This is natural because, for instance, when we apply the general norm comparison results between vertical and conical operators proved in \cite[Proposition 2.1, Proposition 2.3]{AHM12} to $\mathsf{s}_{2,\hh}$ and $\Scal_{2,\hh}$ (see  the definitions below), the boundedness of the conical square function implies boundedness for the vertical one just for values of $p$ less than $2$, while we know that this vertical square function could be bounded for values of $p$ up to infinity (see \cite{AHM12}), as in the case of  $L=-\Delta$.
 In Theorems \ref{thm:conical-verticalnon-gradientdegenerate} and \ref{thm:verticalconicalgegenerategradient} of this paper we prove sharper norm comparison results, in the particular case of considering vertical and conical square functions defined via the heat or Poisson semigroup generated by  $L_w$. 
 As a consequence of those results, we  obtain boundedness of the vertical square functions defined in \eqref{vertical-H} and \eqref{vertical-P} directly  from the boundedness of the conical square functions defined in \eqref{conicalH-1}-\eqref{conicalP-2}.
 To illustrate how our results improve the known comparison result proved in \cite[Proposition 2.3]{AHM12} for this class of operators, let us formulate 
 Theorems \ref{thm:conical-verticalnon-gradientdegenerate} and \ref{thm:verticalconicalgegenerategradient}  in the particular case of $w\equiv 1$ and $v\in A_{\infty}$, and for the square functions:
 $$
 \mathsf{s}_{2,\hh}f(x):=\left(\int_{0}^{\infty}|t^2Le^{-t^2L}f(y)|^2\frac{dt}{t}\right)^{\frac{1}{2}},
 \quad 
 \mathsf{g}_{0,\hh}f(x):=\left(\int_{0}^{\infty}|t\nabla e^{-t^2L}f(y)|^2\frac{dt}{t}\right)^{\frac{1}{2}},
 $$
 $$
 \Scal_{2,\hh}f(x):=\left(\iint_{\Gamma(x)}|t^2Le^{-t^2L}f(y)|^2\frac{dy\,dt}{t^{n+1}}\right)^{\frac{1}{2}},
 \quad\textrm{and}\quad
 \Grm_{0,\hh}f(x):=\left(\iint_{\Gamma(x)}|t\nabla e^{-t^2L}f(y)|^2\frac{dy\,dt}{t^{n+1}}\right)^{\frac{1}{2}},
 $$ 
 where $\Gamma(x):=\{(y,t)\in \R^{n+1}_+: \,|x-y|<t\}$, is the cone of aperture one with vertex at $x$. For  other definitions see Sections \ref{sec:weights} and \ref{sec:elliptic}.
\begin{theorem}
Given  $f\in L^2(\R^n)$ and $v\in  A_{\infty}$, we have
 \begin{list}{$(\theenumi)$}{\usecounter{enumi}\leftmargin=1cm \labelwidth=1cm\itemsep=0.2cm\topsep=.0cm \renewcommand{\theenumi}{\alph{enumi}}}
 
  \item $\|\mathsf{s}_{2,\hh}f\|_{L^p(v)}\lesssim
 \|\mathcal{S}_{2,\hh}f\|_{L^p(v)}$, for $p\in(0,p_+(L))$ and $v\in RH_{\left(\frac{p_+(L)}{p}\right)'}$.

\item $\|\mathcal{S}_{2,\hh}f\|_{L^p(v)}\lesssim
 \|\mathsf{s}_{2,\hh}f\|_{L^p(v)}$,  for $p\in (p_-(L),\infty)$ and $v\in A_{\frac{p}{p_-(L)}}$.

  \item $\|\mathsf{g}_{0,\hh}f\|_{L^p(v)}\lesssim
 \|\Grm_{0,\hh}f\|_{L^p(v)}$, for $p\in (0,q_+(L))$
 and $v\in RH_{\left(\frac{q_+(L)}{p}\right)'}$.

  \item $\|\Grm_{0,\hh}f\|_{L^p(v)}\lesssim
 \|\mathsf{g}_{0,\hh}f\|_{L^p(v)}$, for $p\in (q_-(L),\infty)$ and $v\in A_{\frac{p}{q_-(L)}}$.
 \end{list}
In particular, for all $p\in (p_-(L),p_+(L))$ and $v\in A_{\frac{p}{p_-(L)}}\cap RH_{\left(\frac{p_+(L)}{p}\right)'}$
$$
\|\mathsf{s}_{2,\hh}f\|_{L^p(v)}\approx
 \|\mathcal{S}_{2,\hh}f\|_{L^p(v)},$$ 
and, for all $p\in (q_-(L),q_+(L))$  and $v\in A_{\frac{p}{q_-(L)}}\cap RH_{\left(\frac{q_+(L)}{p}\right)'}$
 $$\|\mathsf{g}_{0,\hh}f\|_{L^p(v)}\approx
 \|\Grm_{0,\hh}f\|_{L^p(v)}.
$$

\end{theorem}
If we had applied the results in \cite[Proposition 2.1]{AHM12}, we would have obtained $(b)$ and $(d)$ for $p\in (2,\infty)$ and $v\in A_{p/2}$; and $(a)$ and $(c)$  for $p\in (0,2)$ and $v\in RH_{(2/p)'}$, which are clearly smaller intervals (see Section \ref{sec:weights}). We recall that the above norm comparison results are between vertical and conical operators. Norm comparison results between vertical square function are easy consequences of the boundedness or off-diagonal estimates of the heat or Poisson semigroup, the Riesz transform $\nabla L^{-\frac{1}{2}}$, and the subordination formula. They have been observed and used, as needed, in some papers    such as \cite{AHM12, HMay09}. As for norm comparison result between conical square functions see \cite{MaPAI17,ChMPA16}.

Coming back to Theorems \ref{thm:conical-verticalnon-gradientdegenerate} and \ref{thm:verticalconicalgegenerategradient}, their proofs are obtained from Proposition \ref{prop:comparison-general}. This is a general norm comparison result between a particular class of vertical and conical operators. This proposition follows from off-diagonal and change of angle estimates, and  extrapolation for Muckenhoupt weights.
Theorem \ref{thm:conical-verticalnon-gradientdegenerate}
implies that  the vertical and conical square function associated with the heat or Poisson semigroup are equivalent operators in $L^p(vdw)$ for the values of $p$ for which the heat semigroup is uniformly bounded.  Analogously,  when $w\equiv 1$, by Theorem \ref{thm:verticalconicalgegenerategradient} 
we obtain these equivalences for the vertical and conical square functions associated with the gradient of the heat or Poisson semigroup, for the values of $p$ for which the gradient of the heat semigroup is uniformly bounded. In the case that $w\in A_2$ the use of Poincar\'e inequality narrows the interval and the class of weights where the equivalences hold. Anyhow, in Remark \ref{remark:moreconvenient2} we see that the vertical square functions considered in Theorem \ref{thm:verticalconicalgegenerategradient} can be controlled in norm by conical square functions associated with the heat semigroup (without gradient), in the range where the gradient of the heat semigroup is uniformly bounded. Similarly, we could prove norm comparison results, between  the conical square functions considered in Theorem \ref{thm:verticalconicalgegenerategradient} and vertical square functions associated with the heat semigroup, in bigger intervals than those considered in Theorem \ref{thm:verticalconicalgegenerategradient}, parts $(b)$ and $(d)$. Although the proofs are much more long and intricate, and without an application in sight their small contribution to this work does not seem to be worth the effort, specially for the reader.

 Moreover, in Theorems \ref{thm:boundednessverticalheat} and \ref{thm:boundednessverticalpoisson} we infer the boundedness of the vertical square functions from the boundedness of the conical ones. We recall that the boundedness of representative vertical and conical square functions has been studied in several papers already mentioned: \cite{Au07,AMIII06,AHM12,MaPAI17,ChMPA16,CMR15}. We also observe that (in addition to even powers) we allow odd powers of the square root of the operator $L_w$ in the definitions of the square functions that we deal with (see \eqref{vertical-H}-\eqref{conicalP-2}), a possibility that has not always been considered in the aforementioned papers.

\medskip

We now turn towards the non-tangential maximal functions associated with the operator $L_w$. In Theorem \ref{thm:boundednessnon-tangential} we study boundedness  of $\Ncal_{\pp}^w$ (see the definition below in \eqref{nontangential}). In \cite{ChMPA18} this has been recently considered; the proof there follows the lines of \cite{MaPAII17} in the weighted non-degenerate case. Here we modify that proof to improve the range of boundedness, even in the  unweighted non-degenerate case (i.e., $w\equiv 1\equiv v$), see
\cite[(6.49)]{HMay09} and \cite{May10}. 
More specifically, this improvement follows from the comparison result in Theorem \ref{thm:improvementnontangentialpoisson}. To explain this better let us consider, for instance, that $w\equiv 1$ and $v\in A_{\infty}$. So far we knew that $\Ncal_{\pp}$ can be extended to a bounded operator on $L^p(v)$ for all $p\in (p_-(L),p_+(L))$ and $v\in A_{\frac{p}{p_-(L)}}\cap RH_{\left(\frac{p_+(L)}{p}\right)'}$. However, in  Theorem \ref{thm:improvementnontangentialpoisson} we obtain the inequality (see the definition of $\Ncal_{\hh}$ in \eqref{nontangential}):
\begin{align}\label{comparisonpoissonnondegenerate}
\|\mathcal{N}_{\pp}f\|_{L^p(v)}\lesssim \|\Ncal_{\hh}f\|_{L^p(v)}+\|\Scal_{2,\hh}f\|_{L^p(v)},\quad
\forall\,p\in (p_-(L),p_+(L)^{*})\,\, \textrm{and}\,\, v\in A_{\frac{p}{p_-(L)}}\cap RH_{\left(\frac{p_+(L)^*}{p}\right)'}.
\end{align}
 From the boundedness of $\Ncal_{\hh}$ and $\Scal_{2,\hh}$ (see \cite{MaPAI17,MaPAII17}), this implies that
$\Ncal_{\pp}$ can be extended to a bounded operator on $L^p(v)$ for all $p\in (p_-(L),p_+(L)^*)$ and $v\in A_{\frac{p}{p_-(L)}}\cap RH_{\left(\frac{p_+(L)^*}{p}\right)'}$ (see definition \eqref{p_w^*}). In turn, if $v\equiv 1$ this implies that $\Ncal_{\pp}$ can be extended to a bounded operator on $L^p(\R^n)$ for all $p\in (p_-(L),p_+(L)^*)$.  
The gist of the proof is in 
obtaining  in the second term of the  sum in \eqref{comparisonpoissonnondegenerate} the $L^p(v)$ norm of a conical square function, while so far we only knew the above inequality  replacing that norm with the $L^p(v)$ norm of a vertical square function which ultimately has worse boundedness properties than the conical one. We are able to achieve this improvement by the use of extrapolation for Muckenhoupt weights.

\medskip

Finally we consider the square root of the operator $L_w$. The properties of this operator as been widely studied, specially after the resolution of the Kato conjecture in 
any dimension in \cite{AHLMT02}, when $w\equiv 1$. That is,
\begin{align}\label{Kato}
\|\sqrt{L}f\|_{L^2(\R^n)}\approx\|\nabla f\|_{L^2(\R^n)}.
\end{align}
In the case that $w\in A_2$, this was solved in \cite{CUR15}. The extension to $L^p$, for a general $p$,
was done in \cite{Au07} in the unweighted  non-degenerate situation, and in \cite{AMIII06} in the weighted non-degenerate case. Besides, in \cite{CMR15} the authors proved the unweighted degenerate version of \eqref{Kato} for a general $p$. That is, for $w\in A_2$, \begin{align*}
\|\sqrt{L_w}f\|_{L^p(w)}\approx\|\nabla f\|_{L^p(w)}.
\end{align*}
  In that paper,
  the weighted degenerate case (that is, $w\in A_2$ and $L^p(vdw)$ with $v\in A_{\infty}(w)$) was also considered, but in view of the previous results in the unweighted or weighted non-degenerate case, we expect that the range of boundedness  obtained in  \cite[Proposition 6.1]{CMR15} regarding the inequality
 \begin{align*}
\|\sqrt{L_w}f\|_{L^p(vdw)}\lesssim \|\nabla f\|_{L^p(vdw)},
\end{align*} 
 can be improved. 
  We do so in Theorem \ref{thm:boundednesssaquareroot}, by seeing the product weight $v\cdot w$ as a weight in $A_{\infty}$ (see Remark \ref{remark:product-weight}), and then applying the Calder\'on-Zygmund  decomposition  in Lemma \ref{lem:CZweighted}
and  interpolation in weighted Lebesgue spaces and in Sobolev spaces (see \cite{Ba09}).

\medskip

The organization of this paper is as follows. In
Section \ref{sec:preliminaries} we list some properties and useful results for the present work about Muckenhoupt weights and elliptic operators. In 
 Section \ref{sec:main} we formulate our comparison and boundedness results. In Section \ref{sec:auxiliary} 
 we obtain some auxiliary results. In Section \ref{sec:proofs}
  we prove the theorems established in Section \ref{sec:main}. Finally in Section \ref{sec:unweighted}, we provide unweighted estimates for degenerate operators. That is, we consider $w\in A_2$  and $v\equiv 1$.

\section{Preliminaries}\label{sec:preliminaries}
%
First of all we note that when we say unweighted degenerate case we mean that we consider $w\in A_2$ and $v\equiv 1$, weighted degenerate case stands for $w\in A_2$ and $v\in A_{\infty}(w)$,  unweighted non-degenerate case indicates  $w\equiv 1\equiv v$, and weighted non-degenerate case refers to  $w\equiv 1$ and $v\in A_{\infty}$.

Next, we specify our notation. We denote by $n$ the dimension of the underlying space $\R^n$ and we always assume that $n\geq 2$. Let $\mu$ be a measure in $\R^n$, given a set $E\subset \R^n$,  we write 
$$
\mu(E):=\int_E\,d\mu.
$$
Moreover, for $1\leq p<\infty$, we denote the Lebesgue space $L^p(\R^n,d\mu)$ by $L^p(\mu)$.
Throughout the paper the measure $\mu$ will be given by a weight $w\in A_{\infty}$ or a product of weights $vw$, where $w\in A_{\infty}$ and $v\in  A_{\infty}(w)$, see \eqref{doublingcondition} and Remark \ref{remark:product-weight}, and the definitions below. In the latter case, we shall use independently the notation $vdw$ or $d(vw)$ depending on whether we want to emphasize the fact that $w\in A_{\infty}$ and $v\in  A_{\infty}(w)$ or  to see $vw$ as a weight in $A_\infty$. In the same line we can write  $L^p(vdw)$  or $L^p(vw)$.

Besides, for every ball $B\subset \R^n$, we define the annuli of $B$ as
\begin{align}\label{anulli}
C_1(B):=4B,\qquad C_j(B):=2^{j+1}B\setminus 2^jB, \quad\textrm{for}\quad j\geq 2,
\end{align}
where for any $\lambda>0$ we denote $\lambda B$ as the ball with the same center as $B$ and radius $\lambda$ times the radius of $B$.
Furthermore, abusing notation
\begin{align}\label{average}
\dashint_{4B}f\, d\mu :=\frac{1}{\mu(4B)}\int_{4B}f\,d\mu,\qquad
\dashint_{C_j(B)}f\, d\mu :=\frac{1}{\mu(2^{j+1}B)}\int_{C_j(B)}f\,d\mu\quad\textrm{for}\quad j\geq 2.
\end{align}

Additionally, we denote by $C$, $c$, or $\theta$ any positive constant that may depend on several parameters, but without altering the result of the main computation. In some cases, we indicate such a dependence by adding a subindex.

Finally, we 
define $\R^{n+1}_+:=\{(x,t): x\in \R^n,\,t>0\}$ the upper-half space, and $\N_0:=\N\cup\{0\}.$
\subsection{Muckenhoupt weights}\label{sec:weights}
In this section we present some properties of Muckenhoupt weights,  for further details
see~\cite{Du01, GCRF85, GrafakosI}.

A Muckenhoupt weight $w$ is a non-negative, locally integrable function.
We say that $w$ belongs to an  $A_p$ class, denoted by $w\in A_p$  if, for $1<p<\infty$,
$$
[w]_{A_p} := \sup_B \left(\dashint_B w(x)\,dx\right) \left(\dashint_B
  w(x)^{1-p'}\,dx\right)^{p-1} < \infty,
  $$
and, $w\in A_1$ if
$$
[w]_{A_1} := \sup_B \left(\dashint_B w(x)\,dx\right)  \left(\esssup_{x\in B} w(x)^{-
1}\right)<
\infty. 
$$
Here and below the suprema run over the collection of balls $B\subset\R^n$.

The weight $w$ can also belong to a Reverse H\"older class denoted by $w\in RH_s$. 
We say that $w\in RH_s$  if, for $1<s<\infty$,
$$
[w]_{RH_s} := \sup_B \left(\dashint_B w(x)\,dx\right )^{-1}
\left(\dashint_B w(x)^s\,dx\right )^{1/s} < \infty,
$$
and $w\in RH_{\infty}$ if
$$
[w]_{RH_\infty} := \sup_B\left(\dashint_B w(x)\,dx\right)^{-1} \left(\esssup_{x\in B} w(x)\right)  <
\infty. 
$$
We denote by $A_{\infty}$ the collection of all the weights
$$
 A_\infty := \bigcup_{1\leq p <\infty} A_p  = \bigcup_{1<s\le \infty}
RH_s.  
$$

An important property is that if $w\in RH_s$, $1<s\le\infty$, then 
\begin{align}\label{pesosineqw:RHq}
\frac{w(E)}{w(B)}\leq  [w]_{RH_{s}}\br{\frac{|E|}{|B|}}^{\frac{1}{s'}}, \quad \forall\,E\subset B,
\end{align}
where $B$ is any ball in $\re^n$. 
Analogously, if
$w\in A_p$,$1\leq p<\infty$, then
\begin{align}\label{pesosineqw:Ap}
 \br{\frac{|E|}{|B|}}^{p}\le [w]_{A_{p}}\frac{w(E)}{w(B)}, \quad \forall\,E\subset B.
\end{align}
This implies in particular that $w$ is a doubling measure:
\begin{align}\label{doublingcondition}
w(\lambda B)
\le
[w]_{A_r}\,\lambda^{n\,r}w(B),
\qquad \forall\,B,\ \forall\,\lambda>1.
\end{align}

As a consequence of this doubling property, we have that with the ordinary Euclidean distance
$|\cdot|$, $(\re,dw,|\cdot|)$ is a space of homogeneous type.
In this setting we can define new classes of weights  $A_p(w)$
and $RH_s(w)$ by replacing the Lebesgue measure in the definitions above with
$dw$: e.g., $v\in A_p(w)$ for $1<p<\infty$ if
\begin{align}\label{apclassv}
 [v]_{A_p(w)} = \sup_B \left(\dashint_B v(x)\,dw\right) \left(\dashint_B
  v(x)^{1-p'}\,dw\right)^{p-1} < \infty,
\end{align}
and $v\in RH_{s}(w)$ for $1<s<\infty$ if
\begin{align}\label{rhclassv}
 [v]_{RH_s(w)} = \sup_B \left(\dashint_B v(x)\,dw\right)^{-1} \left(\dashint_B
  v(x)^{s}\,dw\right)^{\frac{1}{s}} < \infty.
\end{align}
From these definitions, it follows at once  that there is a
``duality'' relationship between the weighted and unweighted
$A_p$ and $RH_s$ conditions: 
\begin{align}\label{dualityapclassesRHclasses}
w^{-1} \in A_p(w)\,\textrm{ if and only if }
\,
w\in RH_{p'},
\quad \textrm{and}\quad w^{-1} \in
RH_{s'}(w)\,
\textrm{ if and only if }\, w\in A_{s}.
\end{align}
\medskip

It is also important to observe that the
weights in the $A_p$ and $RH_{s'}$ classes have a self-improving
property: if $w\in A_p$, there exists $\epsilon>0$ such that $w\in
A_{p-\epsilon}$, and similarly if $w\in RH_{s'}$, then $w\in
RH_{(s-\delta)'}$ for some $\delta>0$. Thus, 
given $w\in A_\infty$, we define the infimum of those values by
\begin{equation}
r_w:=\inf\big\{p:\ w\in A_p\big\}, \qquad s_w:=\inf\big\{q:\ w\in RH_{q'}\big\}.
\label{eq:defi:rw}
\end{equation}
It should be noted that some authors prefer to define $s_w$ as the conjugate exponent of ours. That is, by $\sup\big\{q:\ w\in RH_{q}\big\}$, see for instance
\cite[Lemma 4.1]{AMI07} or \cite{CMR15}.

Related to $r_w$ and $s_w$ we define the following intervals.
Given $0\le p_0<q_0\le \infty$ and $w\in A_{\infty}$,  \cite[Lemma 4.1]{AMI07} implies that
\begin{align}\label{intervalrs}
\mathcal{W}_w(p_0,q_0):=\left\{p\in (p_0, q_0): \ w\in A_{\frac{p}{p_0}}\cap RH_{\left(\frac{q_0}{p}\right)'}\right\}
=
\left(p_0r_w,\frac{q_0}{s_w}\right).
\end{align}
If $p_0=0$ and $q_0<\infty$ it is understood that the only condition that stays is $w\in RH_{\left(\frac{q_0}{p}\right)'}$. Analogously, 
if $0<p_0$ and $q_0=\infty$ the only assumption is $w\in A_{\frac{p}{p_0}}$. Finally $\mathcal{W}_w(0,\infty)=(0,\infty)$.

In the same way, for a weight $v\in A_{\infty}(w)$, with $w\in A_\infty$ we set
\begin{align}\label{def:rvw}
\mathfrak{r}_v(w):=\inf\big\{r:\ v\in A_{r}(w)\big\}\quad \textrm{and}\quad
\mathfrak{s}_v(w):=\inf\big\{s:\ v\in RH_{s'}(w)\big\}.
\end{align}
For $0\le p_0<q_0\le \infty$ and $v\in A_{\infty}(w)$, by a similar argument to that of \cite[Lemma 4.1]{AMI07}, we have
\begin{align}\label{intervalrsw}
\mathcal{W}_v^w(p_0,q_0):=\left\{p\in(p_0,q_0):\ v\in A_{\frac{p}{p_0}}(w)\cap RH_{\left(\frac{q_0}{p}\right)'}(w)\right\}
=
\left(p_0\mathfrak{r}_v(w),\frac{q_0}{\mathfrak{s}_v(w)}\right).
\end{align}
If $p_0=0$ and $q_0<\infty$, as before, it is understood that the only condition that stays is $v\in RH_{\left(\frac{q_0}{p}\right)'}(w)$. Analogously, if $0<p_0$ and $q_0=\infty$ the only assumption is $v\in A_{\frac{p}{p_0}}(w)$. Finally $\mathcal{W}_v^w(0,\infty)=(0,\infty)$.

\medskip

Additionally, note that
for every $w\in A_p$, $v\in A_q(w)$, $1\le p,q<\infty$, it follows that
\begin{align}\label{pesosineq:Ap}
\left(\frac{|E|}{|B|}\right)^{p\,q}
\le
[w]_{A_{p}}^q\left(\frac{w(E)}{w(B)}\right)^{q}
\le
[w]_{A_{p}}^q[v]_{A_{q}(w)}
\frac{vw(E)}{vw(B)},\quad \forall\,E\subset B.
\end{align}
Analogously, if $w\in RH_{p}$ and $v\in RH_{q}(w)$ $1< p,q\le\infty$, one has 
\begin{align}\label{pesosineq:RHq}
\frac{vw(E)}{vw(B)}
\leq
[v]_{RH_{q}(w)}\left(\frac{w(E)}{w(B)}\right)^{\frac{1}{q'}}\leq [v]_{RH_{q}(w)}[w]_{RH_{p}}^{\frac1{q'}}\left(\frac{|E|}{|B|}\right)^{\frac{1}{p'\,q'}},\quad \forall\,E\subset B.
\end{align}
From these inequalities we can guess that given a weight $w\in A_{\infty}$ and $v\in A_{\infty}(w)$ the product of  $v$ and $w$ may belong to $A_{\infty}$. In fact, we obtain the following:

\begin{remark}\label{remark:product-weight}
Given $w\in A_{\infty}$ and $v\in A_{\infty}(w)$, we have that $r_{vw}\leq r_w\mathfrak{r}_v(w)$. The  equality holds when $r_w=1=\mathfrak{r}_v(w)$, and  the converse inequality is false in general. For instance, consider $w(x):=|x|^{n}$ and $v:=w^{-1}$. We have that $r_w\mathfrak{r}_v(w)=2$ but $r_{vw}=1$, since $vw\equiv 1$.

In view of \eqref{eq:defi:rw} and  \eqref{def:rvw}, it is immediate to see that the inequality $r_{vw}\leq r_w\mathfrak{r}_v(w)$ follows from the fact that,
for
$1\leq p,q<\infty$, $w\in A_p$, and $v\in A_q(w)$, the product of the weights $v$ and $w$ belongs to the Muckenhoupt class $p$ times $q$. That is $vw\in A_{pq}$.  
This is an  easy consequence of  H\"older's inequality.
 Indeed, assuming that $p\neq 1$ and $q\neq 1$ (the cases $p=1$ and/or $q=1$ follow similarly), since $pq-1>q-1$ and $\left(\frac{pq-1}{q-1}\right)'=\frac{pq-1}{q(p-1)}$,
\begin{multline*}
\left(\frac{1}{|B|}\int_{B}(vw)^{1-(pq)'}\right)^{pq-1}
=
\left(\frac{1}{|B|}\int_{B}v^{-\frac{1}{pq-1}}w^{\frac{q-1}{pq-1}}w^{-\frac{q}{pq-1}}\right)^{pq-1}
\\
\lesssim
\left(\frac{1}{|B|}\int_{B}w\right)^{q-1}
\left(\frac{1}{w(B)}\int_{B}v^{-\frac{1}{q-1}}w\right)^{q-1}
\left(\frac{1}{|B|}\int_{B}w^{-\frac{1}{p-1}}\right)^{q(p-1)}
\\
\lesssim
\left(\frac{1}{|B|}\int_{B}w\right)^{-1}
\left(\frac{1}{w(B)}\int_{B}vw\right)^{-1}
=
\left(\frac{1}{|B|}\int_{B}vw\right)^{-1}.
\end{multline*}

\end{remark}

Next, we observe that, under  particular assumptions, we can compare the average of a function with respect to the measure given by a weight $w\in A_{\infty}$, with that with respect to a product of weights $w\in A_{\infty}$ and $v\in A_{\infty}(w)$. More specifically:
\begin{remark}
Given $0<\widetilde{p}\leq\widetilde{q}<\infty$, note that, if $v\in A_{\frac{\widetilde{q}}{\widetilde{p}}}(w)$ then, 
\begin{align}\label{Asinpesoconpeso}
\left(\dashint_{C_j(B)}|f(x)|^{\widetilde{p}}dw(x)\right)^{\frac{1}{\widetilde{p}}}
\lesssim
\left(\dashint_{C_j(B)}|f(x)|^{\widetilde{q}}d(vw)(x)\right)^{\frac{1}{\widetilde{q}}}, \quad \forall\, B\subset\R^n,\quad j\geq 1.
\end{align}
On the other hand, if $v\in RH_{\left(\frac{\widetilde{q}}{\widetilde{p}}\right)'}(w)$ then
\begin{align}\label{RHsinpesoconpeso}
\left(\dashint_{C_j(B)}|f(x)|^{\widetilde{p}}d(vw)(x)\right)^{\frac{1}{\widetilde{p}}}
\lesssim
\left(\dashint_{C_j(B)}|f(x)|^{\widetilde{q}}dw(x)\right)^{\frac{1}{\widetilde{q}}}, \quad \forall\, B\subset\R^n,\quad j\geq 1.
\end{align}
We detail the case $\widetilde{p}<\widetilde{q}$. The case $\widetilde{p}=\widetilde{q}$ follows similarly.

We obtain \eqref{Asinpesoconpeso} applying H\"older's inequality and \eqref{apclassv}
\begin{multline*}
\left(\dashint_{C_j(B)}|f(x)|^{\widetilde{p}}dw(x)\right)^{\frac{1}{\widetilde{p}}}
=
\left(\dashint_{C_j(B)}|f(x)|^{\widetilde{p}}v(x)^{\frac{\widetilde{p}}{\widetilde{q}}}v(x)^{-\frac{\widetilde{p}}{\widetilde{q}}}dw(x)\right)^{\frac{1}{\widetilde{p}}}
\\
\lesssim
\left(\dashint_{C_j(B)}|f(x)|^{\widetilde{q}}v(x)dw(x)\right)^{\frac{1}{\widetilde{q}}}
\left(\dashint_{2^{j+1}B}v(x)^{1-\left(\frac{\widetilde{q}}{\widetilde{p}}\right)'}dw(x)\right)^{\frac{1}{\widetilde{q}}\left(\frac{\widetilde{q}}{\widetilde{p}}-1\right)}
\\
\lesssim
\left(\dashint_{C_j(B)}|f(x)|^{\widetilde{q}}v(x)dw(x)\right)^{\frac{1}{\widetilde{q}}}
\left(\dashint_{2^{j+1}B}v(x)dw(x)\right)^{-\frac{1}{\widetilde{q}}}
=
\left(\dashint_{C_j(B)}|f(x)|^{\widetilde{q}}d(vw)(x)\right)^{\frac{1}{\widetilde{q}}}.
\end{multline*}
Similarly, we obtain \eqref{RHsinpesoconpeso} applying H\"older's inequality and \eqref{rhclassv}
\begin{multline*}
\left(\dashint_{C_j(B)}|f(x)|^{\widetilde{p}}d(vw)(x)\right)^{\frac{1}{\widetilde{p}}}
=
\left(\frac{w(2^{j+1}B)}{vw(2^{j+1}B)}\right)^{\frac{1}{\widetilde{p}}}\left(\dashint_{C_j(B)}|f(x)|^{\widetilde{p}}v(x)dw(x)\right)^{\frac{1}{\widetilde{p}}}
\\
\lesssim\left(\dashint_{2^{j+1}B}v(x)dw(x)\right)^{-\frac{1}{\widetilde{p}}}\left(\dashint_{C_j(B)}|f(x)|^{\widetilde{q}}dw(x)\right)^{\frac{1}{\widetilde{q}}}
\left(\dashint_{2^{j+1}B}v(x)^{\left(\frac{\widetilde{q}}{\widetilde{p}}\right)'}dw(x)\right)^{\frac{1}{\widetilde{p}}\frac{1}{\left(\frac{\widetilde{q}}{\widetilde{p}}\right)'}}
\\
\lesssim
\left(\dashint_{C_j(B)}|f(x)|^{\widetilde{q}}dw(x)\right)^{\frac{1}{\widetilde{q}}}.
\end{multline*}
\end{remark}
We also introduce the Hardy-Littlewood maximal function
$$
\mathcal{M} f(x):=\sup_{B\ni x} \dashint_B|f(y)|\,dy.
$$
By the classical theory of weights, $w\in A_p$, $1<p<\infty$, if and only if $\mathcal{M}$ is bounded on $L^p(w)$.

Likewise, given $w\in A_\infty$, we can introduce the weighted maximal operator $\mathcal{M}^w$:
\begin{align}\label{weightedHLM}
\mathcal{M}^wf(x):=\sup_{B\ni x}\dashint_B |f(y)|\,dw(y).
\end{align} 
Since $w$ is a doubling measure, one can also show that $v\in A_p(w)$, $1<p<\infty$, if and only if $\mathcal{M}^w$ is bounded on $L^p(v dw)$.

Furthermore, for  any $p\in (0,\infty)$ and a weight $w\in A_{\infty}$, we define 
\begin{align}\label{p_{w,*}}
(p)_{w,*}:=\frac{nr_wp}{nr_w+p},
\end{align}
and, for $k\in \N$,
\begin{align}\label{p_w^*}
p_{w}^{k,*}:=\begin{cases}
\frac{nr_wp}{nr_w-kp}& \textrm{ if }nr_w>kp,
\\
\infty & \textrm{ otherwise.}
\end{cases}
\end{align}
 If we consider $w\equiv 1$, since $r_w=1$, we have that
\begin{align*}
(p)_{*}:=\frac{np}{n+p},
\quad\textrm{and}\quad
p^{k,*}:=\begin{cases}
\frac{np}{n-kp}& \textrm{ if }n>kp,
\\
\infty & \textrm{ otherwise.}
\end{cases}
\end{align*}
We write $p_w^*:=p_w^{1,*}$, or  $p^*:=p^{1,*}$,  when $w\equiv 1$.

\medskip

\subsection{Elliptic operators}\label{sec:elliptic}
Consider $A$ an $n\times n$ matrix of complex and $L^\infty$-valued coefficients defined on $\R^n$, satisfying the following uniform ellipticity (or \lq\lq
accretivity\rq\rq) condition: there exist $0<\lambda\le\Lambda<\infty$ such that
\begin{equation}\label{elepliptitiA}
\lambda\,|\xi|^2
\le
\Re A(x)\,\xi\cdot\bar{\xi}
\quad\qquad\mbox{and}\qquad\quad
|A(x)\,\xi\cdot \bar{\zeta}|
\le
\Lambda\,|\xi|\,|\zeta|,
\end{equation}
for all $\xi,\zeta\in\mathbb{C}^n$ and almost every $x\in \R^n$. We have used the notation
$\xi\cdot\zeta=\xi_1\,\zeta_1+\cdots+\xi_n\,\zeta_n$ and therefore
$\xi\cdot\bar{\zeta}$ is the usual inner product in $\mathbb{C}^n$. 
Given a Muckenhoupt weight $w\in A_2$ we define a second order divergence form degenerate elliptic operator  by
\begin{align}\label{def:L_w}
 L_w f:=-w^{-1}\div(wA\nabla f).
\end{align}
In the non-degenerate case we define $L f:=-\div(A\nabla f)$, and replace  $L_w$ with $L$ everywhere below.

  The operator $-L_w$ generates a $C_0-$semigroup of contractions on $L^2(w)$ which is called the heat semigroup $\{e^{-tL_w}\}_{t>0}$. We also consider the Poisson semigroup  $\{e^{-t\sqrt{L_w}}\}_{t>0}$ defined via the subordination formula:
\begin{align}\label{subordinationformula}
e^{-t\sqrt{L_w}}f(y)=\frac{1}{\sqrt{\pi}}\int_0^{\infty}u^{\frac{1}{2}}e^{-u}e^{-\frac{t^2}{4u}L_w}f(y)\frac{du}{u}.
\end{align}
We denote by $(p_-(L_w),p_+(L_w))$ the maximal open interval on which the heat semigroup $\{e^{-tL_w}\}_{t>0}$ is uniformly bounded on $L^p(w)$:
\begin{align*}
p_-(L_w):=\inf\Big\{p\in (0,1):\sup_{t>0}\|e^{-t^2L_w}\|_{L^p(w)\rightarrow L^p(w)}<\infty\Big\},
\\
p_+(L_w):=\sup\Big\{p\in (0,1):\sup_{t>0}\|e^{-t^2L_w}\|_{L^p(w)\rightarrow L^p(w)}<\infty\Big\}.
\end{align*}
Note that in place of the semigroup $\{e^{-tL_w}\}_{t>0}$ we are using its rescaling $\{e^{-t^2L_w}\}_{t>0}$. We do so because all the ``heat'' square functions, defined below, are written using the latter and because in the context of off-diagonal estimates, discussed in the next section, this will simplify some computations.

According to \cite{CMR15} (see also \cite{Au07}), 
\begin{equation}\label{p-p+} 
p_-(L_w) \leq (2^*_w)'<2<2^*_w\le p_+(L_w).
\end{equation}

For all $m\in \N$ and $K\in \N_0$, we define the vertical square functions associated with the heat semigroup by
\begin{align}\label{vertical-H}
\mathsf{s}_{m,\hh}^wf(y)\!=\!\!\left(\int_{0}^{\infty}|(t\sqrt{L_w})^me^{-t^2L_w}f(y)|^2\frac{dt}{t}\right)^{\!\!\frac{1}{2}}
\hspace*{0.2cm}\textrm{and}\hspace*{0.2cm}
\mathsf{g}_{K,\hh}^wf(y)\!=\!\!\left(\int_{0}^{\infty}|t\nabla (t\sqrt{L_w})^{K}e^{-t^2L_w}f(y)|^2\frac{dt}{t}\right)^{\!\!\frac{1}{2}};
\end{align}
and with the Poisson semigroup 
\begin{align}\label{vertical-P}
\mathsf{s}_{m
,\pp}^wf(y)\!=\!\!\left(\int_{0}^{\infty}| (t\sqrt{L_w})^{m}e^{-t\sqrt{L_w}}f(y)|^2\frac{dt}{t}\right)^{\!\!\frac{1}{2}}
\hspace*{0.2cm} \textrm{and}\hspace*{0.2cm}
\mathsf{g}_{K,\pp}^wf(y)\!=\!\!\left(\int_{0}^{\infty}|t\nabla (t\sqrt{L_w})^{K}e^{-t\sqrt{L_w}}f(y)|^2\frac{dt}{t}\right)^{\!\!\frac{1}{2}}.
\end{align}
The conical square functions are defined by
\begin{align}\label{conicalH-1}
\Scal_{m,\hh}^wf(y)=\left(\iint_{\Gamma(x)}|(t\sqrt{L_w})^me^{-t^2L_w}f(y)|^2\frac{dw(y)dt}{tw(B(y,t))}\right)^{\frac{1}{2}},
\end{align}
\begin{align}\label{conicalH-2}
\Grm_{K,\hh}^wf(y)=\left(\iint_{\Gamma(x)}|t\nabla(t\sqrt{L_w})^Ke^{-t^2L_w}f(y)|^2\frac{dw(y)dt}{tw(B(y,t))}\right)^{\frac{1}{2}},
\end{align}
\begin{align}\label{conicalP-1}
\Scal_{m,\pp}^wf(y)=\left(\iint_{\Gamma(x)}|(t\sqrt{L_w})^me^{-t\sqrt{L_w}}f(y)|^2\frac{dw(y)dt}{tw(B(y,t))}\right)^{\frac{1}{2}},
\end{align}
\begin{align}\label{conicalP-2}
\Grm_{K,\pp}^wf(y)=\left(\iint_{\Gamma(x)}|t\nabla(t\sqrt{L_w})^Ke^{-t\sqrt{L_w}}f(y)|^2\frac{dw(y)dt}{tw(B(y,t))}\right)^{\frac{1}{2}}.
\end{align}
Note that we allow that the square root of the operator $L_w$ is raised to odd powers,  in contrast to \cite{MaPAI17,ChMPA16}, where only even powers were considered.

Besides, the non-tangential maximal functions are defined by
\begin{align}\label{nontangential}
\Ncal_{\hh}^wf(x)=\left(\sup_{t>0}\dashint_{B(x,t)}|e^{-t^2L_w}f(y)|^2dw(y)\right)^{\frac{1}{2}},\,
\textrm{ and }\,
\Ncal_{\pp}^wf(x)=\left(\sup_{t>0}\dashint_{B(x,t)}|e^{-t\sqrt{L_w}}f(y)|^2dw(y)\right)^{\frac{1}{2}}.
\end{align}
When we write the operators defined above without expressing explicitly the dependence on $w$ (i.e., $L$, $\mathsf{s}_{m,\hh}$, $\Scal_{m,\hh}$, $\Ncal_{\hh}$, etc.) we mean that we consider the constant weight $w$ equal to one.

We also consider the following representation of the square root of the operator $L_w$ (see \cite{Au07,CMR15}):
\begin{align}\label{squareroot}
\sqrt{L_w}f(y)=\frac{1}{\sqrt{\pi}}\int_0^{\infty}sL_we^{-s^2L_w}f(y)\frac{ds}{s}.
\end{align}

\medskip

Finally, we note the following result obtained in \cite{ChMPA16} (see also \cite{AHM12,MaPAI17}) where  the authors  proved boundedness for the conical square functions defined in \eqref{conicalH-1}-\eqref{conicalP-2} considering even  powers of the square root of the operator ${L_w}$.
\begin{theorem}\label{thm:boundednessconicaleven}
Given $w\in A_2$ and $v\in A_{\infty}(w)$, for every  $m\in \N$ and $K\in \N_0$, there hold:
\begin{list}{$(\theenumi)$}{\usecounter{enumi}\leftmargin=1cm \labelwidth=1cm\itemsep=0.2cm\topsep=.2cm \renewcommand{\theenumi}{\alph{enumi}}}

\item The conical square functions $\Scal_{2m,\hh}^w$ and $\Grm_{2K,\hh}^w$ can be extended to bounded operators on $L^p(vdw)$, for all
$p\in \mathcal{W}_v^w(p_-(L_w),\infty)$.

\item The conical square functions  $\Scal_{2m,\pp}^w$ and $\Grm_{2K,\pp}^w$ can be extended to bounded operators on $L^p(vdw)$,  for all $p\in \mathcal{W}_v^w(p_-(L_w),p_+(L_w)^{2m+1,*}_w)$  and for all $p\in \mathcal{W}_v^w(p_-(L_w),p_+(L_w)^{2K+1,*}_w)$, respectively.
\end{list}
\end{theorem}

\subsection{Off-diagonal estimates}
\begin{definition}
Let $\{T_t\}_{t>0}$ be a family of sublinear operators and let $1< p< \infty$. Given a doubling measure $\mu$ we say that $\{T_t\}_{t>0}$ satisfies 
$L^p(\mu)-L^p(\mu)$ full off-diagonal estimates, denoted by $T_t\in\mathcal F(L^p(\mu)-L^p(\mu))$, if there exist constants $C,c>0$ such that 
for all closed sets $E$ and $F$, all $f\in L^p(\R^n)$, and all $t>0$ we have 
\begin{equation}\label{off:full}
\left(\int_F |T_t(\chi_E f)(x)|^p d\mu\right)^{\frac{1}{p}} 
\leq C\, 
 e^{-\frac{cd(E,F)^2}{t}} \left(\int_E |f(x)|^p d\mu\right)^{\frac{1}{p}},
\end{equation}
where $d(E,F)=\inf\{|x-y|:x\in E, y\in F\}$.
\end{definition}

Besides, set $\Upsilon(s)=\max\{s,s^{-1}\}$ for $s>0$ and recall the notation in \eqref{anulli} and \eqref{average}.

\begin{definition}
Given $1\leq p\leq q\leq \infty$ and any doubling measure $\mu$,  we say that a family of sublinear operators $\{T_t\}_{t>0}$ satisfies 
$L^p(\mu)-L^q(\mu)$ off-diagonal estimates on balls, denoted by $T_t\in\mathcal O(L^p(\mu)-L^q(\mu))$, if there exist $\theta_1,\theta_2>0$ and $c>0$ such that for all $t>0$ and for all balls $B$ with radius $r_B$, 
\begin{equation}\label{off:BtoB}
\br{\fint_B \abs{T_t(f \chi_B)}^q d\mu}^{{\frac{1}{q}}} \lesssim \Upsilon\br{\frac{r_B}{\sqrt t}}^{\theta_2} \br{\fint_B |f|^p d\mu}^{{\frac{1}{p}}},
\end{equation}
and for $j\geq 2$,
\begin{equation}\label{off:CtoB}
\br{\fint_B \abs{T_t(f \chi_{C_j(B)})}^q d\mu}^{{\frac{1}{q}}} 
\lesssim 2^{j\theta_1} \Upsilon\br{\frac{2^j r_B}{\sqrt t}}^{\theta_2} e^{-\frac{c4^j r_B^2}{t}} \br{\fint_{C_j(B)} |f|^p d\mu}^{{\frac{1}{p}}},
\end{equation}
and
\begin{equation}\label{off:BtoC}
\br{\fint_{C_j(B)} \abs{T_t(f \chi_B)}^q d\mu}^{{\frac{1}{q}}} 
\lesssim 2^{j\theta_1} \Upsilon\br{\frac{2^j r_B}{\sqrt t}}^{\theta_2} e^{-\frac{c4^j r_B^2}{t}} \br{\fint_B |f|^p d\mu}^{{\frac{1}{p}}}.
\end{equation}
\end{definition}

Throughout the paper we shall use the following  results about off-diagonal estimates on balls:
\begin{lemma}[{\cite[Section 2]{AMII07},\,\cite[Sections 3 and 7]{CMR15}}]\label{lem:ODweighted}
Let $L_w$ be a degenerate elliptic operator as in \eqref{def:L_w} with $w\in A_2$. There hold:
\begin{list}{$(\theenumi)$}{\usecounter{enumi}\leftmargin=1cm \labelwidth=1cm\itemsep=0.2cm\topsep=.2cm \renewcommand{\theenumi}{\alph{enumi}}}

\item If $p_-(L_w)<p\leq q<p_+(L_w)$, then $e^{-t L_w}$ and $(tL_w)^m e^{-t L_w}$  belong to $\mathcal O(L^p(w)-L^q(w))$, for every $m\in\N$. 

\item Let $p_-(L_w)<p\le q<p_+(L_w)$. If $v\in A_{p/p_-(L_w)}(w) \cap RH_{(p_+(L_w)/q)'}(w)$, then $e^{-t L_w}$ and $(tL_w)^m e^{-t L_w}$ belong to $\mathcal O(L^p(vw)-L^q(vw))$, for every $m\in\N$.

\item There exists an interval $\mathcal K(L_w)$ such that if $p,q \in \mathcal K(L_w)$, $p\leq q$, we have that $\sqrt t\nabla e^{-t L_w}\in \mathcal O(L^p(w)-L^q(w))$. Moreover, denoting respectively by $q_-(L_w)$ and $q_+(L_w)$ the left and right endpoints of $\mathcal K(L_w)$, we have that $q_-(L_w)=p_-(L_w)$, and $2<q_+(L_w)\leq q_+(L_w)^*_w\leq p_+(L_w)$. 

\item Let $q_-(L_w)<p\le q<q_+(L_w)$. If $v\in A_{p/q_-(L_w)}(w) \cap RH_{(q_+(L_w)/q)'}(w)$, then $\sqrt t\nabla e^{-t L_w}\in \mathcal O(L^p(vw)-L^q(vw))$.

\item If $p=q$ and $\mu$ a doubling measure then $\mathcal{F}(L^p(\mu)-L^p(\mu))$ and $\mathcal{O}(L^p(\mu)-L^p(\mu))$ are equivalent.
\end{list}
\end{lemma}
Furthermore, in the following result, which is a weighted version of \cite[(5.12)]{MaPAII17} (see also \cite{HMay09}), we show off-diagonal estimates for the family $\{\mathcal{T}_{t,s}\}_{s,t>0}:=\{(e^{-t^2L_w}-e^{-(t^2+s^2)L_w})^M\}_{s,t>0}$.
\begin{proposition}\label{prop:lebesgueoff-dBQ}
For $0<t,s<\infty$, $M\in \N$, and  for  ${E}_1, {E}_2\subset \mathbb{R}^n$ closed, given $p\!\in \!(p_-(L_w),p_+(L_w))$, and $f\in L^{p}(w)$ such that $\textrm{supp}(f)\subset {E}_1$,  we have that $\{\mathcal{T}_{t,s}\}_{s,t>0}$ satisfies the following $L^p(w)-L^p(w)$ off-diagonal estimates: 
\begin{align}\label{AB}
\left\|\chi_{E_2}\mathcal{T}_{t,s}f\right\|_{L^{p}(w)}
\lesssim
\left(\frac{s^2}{t^2}\right)^M
e^{-c\frac{d({E}_1,{E}_2)^2}{t^2+s^2}}\|f\chi_{E_1}\|_{L^{p}(w)}.
\end{align}
\end{proposition}
\begin{proof}
Note that  we have
\begin{align*}
&\left\|\chi_{E_2}\mathcal{T}_{t,s}f\right\|_{L^{p}(w)}
=
\left\|\chi_{E_2}\left(e^{-t^2L_w}-e^{-(t^2+s^2)L_w}\right)^M f\right\|_{L^{p}(w)}
=
\left\|\chi_{E_2}\left(\int_{0}^{s^2}\partial_r e^{-(r+t^2)L_w}\,dr\right)^M f\,  \right\|_{L^{p}(w)}
\\ \nonumber
&\leq
\int_{0}^{s^2}\!\!\!\dots \int_{0}^{s^2}\left\|\chi_{E_2}
\big((r_1+\dots+r_M+M t^2) L_w \big)^M e^{-(r_1+\dots+r_M+M t^2) L_w}f\right\|_{L^{p}(w)} \ \frac{dr_1 \dots dr_M}{(r_1+\dots+r_M+M t^2)^M}
\\ \nonumber
&\lesssim
\int_{0}^{s^2}\!\!\!\dots \int_{0}^{s^2}
e^{-c\frac{d({E}_1,{E}_2)^2}{r_1+\dots+r_M+M t^2}}\frac{dr_1 \dots dr_M}{(r_1+\dots+r_M+M t^2)^M}
\|\chi_{E_1}f\|_{L^{p}(w)}
\\ \nonumber
&
\lesssim
\left(\frac{s^2}{t^2}\right)^M
e^{-c\frac{d({E}_1,{E}_2)^2}{t^2+s^2}}\|\chi_{E_1}f\|_{L^{p}(w)},
\end{align*}
where we have used that  $(t L_w)^M e^{-tL_w}\in \mathcal{F}(L^{p}(w)- L^{p}(w))$, for all $p\in (p_-(L_w),p_+(L_w))$, (see Lemma \ref{lem:ODweighted}, part $(e)$). 
\end{proof}
\begin{remark}\label{remark:boundednesstsr}
From Proposition \ref{prop:lebesgueoff-dBQ}, we immediately obtain that, for all $p\in (p_-(L_w),p_+(L_w))$ and $f\in L^{p}(w)$,
$$
\|\mathcal{T}_{t,s}f\|_{L^{p}(w)}\lesssim \left(\frac{s^2}{t^2}\right)^M\|f\|_{L^{p}(w)}, \quad t,s>0.
$$
\end{remark}

\medskip

We conclude this section by introducing the following  off-diagonal estimates on Sobolev spaces. The proof follows as in \cite{ChMPA18}, see also \cite{Au07} for the unweighted non-degenerate case.
\begin{lemma}\label{lem:Gsum}
Let $q\in (q_-(L_w),q_+(L_w))$, $p$ such that $\max\left\{r_w, (q_{-}(L_w))_{w,*}\right\}<p\le q$, and $\alpha>0$. Then, for every $(x,t)\in \R^{n+1}_+$, there exists $\theta>0$ such that
\begin{equation}\label{eq:Gsum}
\br{\,\fint_{B(x,\alpha t)} \abs{\nabla e^{-t^2L_w}f(z)}^q dw(z)}^{{\frac{1}{q}}} \lesssim \Upsilon(\alpha)^{\theta}\sum_{j\geq 1} e^{-c4^j} \br{\,\fint_{B(x,2^{j+1}\alpha t)} |\nabla S_t f(z)|^{p} dw(z)}^{\frac{1}{p}},
\end{equation}
where $S_t$ can be equal to $e^{-\frac
{t^2}{2}L_w}$ or the identity, for all $t>0$.  
%
\end{lemma}

\medskip

\subsection{Extrapolation and change of angle}
In our proofs the use of extrapolation and change of angle formulas
will be essential. In this section we formulate these results.

The following extrapolation result  for $w\equiv 1$ can be found in \cite[Chapter 2]{CruzMartellPerez}, \cite{AMI07}, and see also \cite[Lemma 3.3]{MaPAI17}. The proof
can be easily obtained by adapting the arguments there replacing everywhere the Lebesgue measure with the measure given by $w$ and the Hardy-Littlewood maximal function with its weighted version $\mathcal{M}^w$ introduced in \eqref{weightedHLM}. Further details are left to the interested reader.
\begin{theorem}{\cite[Theorem A.1]{ChMPA16}}\label{theor:extrapol}
Let $\mathcal{F}$ be a given family of pairs $(f,g)$ of non-negative and not identically zero measurable functions.
\begin{list}{$(\theenumi)$}{\usecounter{enumi}\leftmargin=1cm
\labelwidth=1cm\itemsep=0.2cm\topsep=.2cm
\renewcommand{\theenumi}{\alph{enumi}}}

\item Suppose that for some fixed exponent $p_0$, $1\le p_0<\infty$, and every weight $v\in A_{p_0}(w)$,
\begin{equation*}
\int_{\mathbb{R}^{n}}f(x)^{p_0}\,v(x)dw(x)
\leq
 C_{w,v,p_0}
\int_{\mathbb{R}^{n}}g(x)^{p_0}\,v(x)dw(x),
\qquad\forall\,(f,g)\in{\mathcal{F}}.
\end{equation*}
Then, for all $1<p<\infty$, and for all $v\in A_p(w)$,
\begin{equation*}
\int_{\mathbb{R}^{n}}f(x)^{p}\,v(x)dw(x)
\leq
C_{w,v,p}
\int_{\mathbb{R}^{n}}g(x)^{p}\,v(x)dw(x),
\qquad\forall\,(f,g)\in{\mathcal{F}}.
\end{equation*}

\item Suppose that for some fixed exponent $q_0$, $1\le q_0<\infty$, and every weight $v\in RH_{q_0'}(w)$,
\begin{equation*}
\int_{\mathbb{R}^{n}}f(x)^{\frac1{q_0}}\,v(x)dw(x)
\leq
 C_{w,v,q_0}
\int_{\mathbb{R}^{n}}g(x)^{\frac1{q_0}}\,v(x)dw(x),
\qquad\forall\,(f,g)\in{\mathcal{F}}.
\end{equation*}
Then, for all $1<q<\infty$ and for all $v\in RH_{q'}(w)$,
\begin{equation*}
\int_{\mathbb{R}^{n}}f(x)^{\frac1q}\,v(x)dw(x)
\leq
C_{w,v,q}
\int_{\mathbb{R}^{n}}g(x)^{\frac1q}\,v(x)dw(x),
\qquad\forall\,(f,g)\in{\mathcal{F}}.
\end{equation*}

\item Suppose that for some fixed exponent $r_0$, $0< r_0<\infty$, and every weight $v\in A_\infty(w)$,
\begin{equation*}
\int_{\mathbb{R}^{n}}f(x)^{r_0}\,v(x)dw(x)
\leq
C_{w,v,r_0}
\int_{\mathbb{R}^{n}}g(x)^{r_0}\,v(x)dw(x),
\qquad\forall\,(f,g)\in{\mathcal{F}}.
\end{equation*}
Then, for all $0<r<\infty$ and for all $v\in A_\infty(w)$,
\begin{equation*}
\int_{\mathbb{R}^{n}}f(x)^{r}\,v(x)dw(x)
\leq
C_{w,v,r}
\int_{\mathbb{R}^{n}}g(x)^{r}\,v(x)dw(x),
\qquad\forall\,(f,g)\in{\mathcal{F}}.
\end{equation*}
\item Suppose that for some fixed exponents $p_0$ and $q_0$, given $0< p_0<p<q_0<\infty$, and every weight $v\in A_{\frac{p}{p_0}}(w)\cap RH_{\left(\frac{q_0}{p}\right)'}(w)$,
\begin{equation*}
\int_{\mathbb{R}^{n}}f(x)^{p}\,v(x)dw(x)
\leq
 C_{w,v,p}
\int_{\mathbb{R}^{n}}g(x)^{p}\,v(x)dw(x),
\qquad\forall\,(f,g)\in{\mathcal{F}}.
\end{equation*}
Then, for all $p_0<q<q_0$, and for all $v\in A_{\frac{q}{p_0}}(w)\cap RH_{\left(\frac{q_0}{q}\right)'}(w)$,
\begin{equation*}
\int_{\mathbb{R}^{n}}f(x)^{q}\,v(x)dw(x)
\leq
C_{w,v,q}
\int_{\mathbb{R}^{n}}g(x)^{q}\,v(x)dw(x),
\qquad\forall\,(f,g)\in{\mathcal{F}}.
\end{equation*}
\end{list}
\end{theorem}

\medskip

In our proofs we will use these extrapolation results to prove inequalities of the type
$$
\|\mathcal{Q}_0f\|_{L^p(vdw)}\lesssim \|\mathcal{Q}_1f\|_{L^p(vdw)},\quad \forall p\in \mathcal{W}_v^w(\tilde{p},\tilde{q}_w^{\widetilde{k},*}),
$$
where $w\in A_2$, $v\in A_{\infty}(w)$, $0\leq\tilde{p}<q_1<\max\{2,q_1\}<\tilde{q}\leq\infty$, $\widetilde{k}\in \N$, $\mathcal{Q}_0$ and $\mathcal{Q}_1$ are operators acting over a function $f$, and $\tilde{q}_w^{\widetilde{k},*}$  is defined in \eqref{p_w^*}. Hence, it will be enough to show
$$
\|\mathcal{Q}_0f\|_{L^{q_1}(v_0dw)}\lesssim \|\mathcal{Q}_1f\|_{L^{q_1}(v_0dw)},\quad \forall v_0\in A_{\frac{q_1}{\tilde{p}}}(w)\cap RH_{\left(\frac{\tilde{q}_w^{\widetilde{k},*}}{q_1}\right)'}(w).
$$
The fact that $v_0\in RH_{\left(\frac{\tilde{q}_w^{\widetilde{k},*}}{q_1}\right)'}(w)$ yields an important consequence that we see in the following remark. 
\begin{remark}\label{remark:choicereverseholder}
Let $w\in A_2$, $\widetilde{k}\in \N$, $1<q_1<\tilde{q}_w^{\widetilde{k},*}$, $\max\{2,q_1\}<\tilde{q}$,  and $v_0\in RH_{\left(\frac{\tilde{q}_w^{\widetilde{k},*}}{q_1}\right)'}(w)$ then 
 we can find  $\widehat{r}$, $q_0$, and $r$ such that $r_w<\widehat{r}<2$, $2<q_0<\tilde{q}$, $\max\{q_0/2,q_1/2\}\leq r<\infty$, $v_0\in RH_{\left(2r/q_1\right)'}(w)$ so that
\begin{align}\label{positiveremark}
\widetilde{k}+\frac{n\,\widehat{r}}{2r}-\frac{n\,\widehat{r}}{q_0}
> 0.
\end{align}
Indeed, for $n\,r_w>\widetilde{k}\tilde{q}$, note that
$
\mathfrak{s}_{v_0}(w)<\frac{nr_w\tilde{q}}{q_1(nr_w-\widetilde{k}\tilde{q})}.
$
Hence, 
 we can find $2>\widehat{r}>r_w$ close enough to $r_w$, $\varepsilon_0>0$ small enough and $2<q_0<\tilde{q}$, close enough to $\tilde{q}$
so that $\max\{q_1,q_0/\mathfrak{s}_{v_0}(w)\}<\frac{n\widehat{r}q_0}{\mathfrak{s}_{v_0}(w)(1+\varepsilon_0)(n\widehat{r}-\widetilde{k}q_0)}$. Besides, 
define $r:=\frac{q_0n\widehat{r}}{2(1+\varepsilon_0)(n\widehat{r}-\widetilde{k}q_0)}$. Then, 
 $q_0/2<r<\infty$,  $v_0\in RH_{(2r/q_1)'}(w)$, and 
\begin{align*}
\widetilde{k}+\frac{n\widehat{r}}{2r}-\frac{n\widehat{r}}{q_0}>0.
\end{align*}

If now  $n\,r_w\leq \widetilde{k}\tilde{q}$, our condition on the weight $v_0$ becomes $v_0\in A_{\infty}(w)$. Then, we take $r>\mathfrak{s}_{v_0}(w)\max\{q_1/2,1\}$ and $q_0$ satisfying
$\max\left\{2,\frac{2r\tilde{q}}{\tilde{q}+2r}\right\}<q_0<\min\left\{\tilde{q},2r\right\}$ if $\tilde{q}<\infty$, or $q_0=2r$ if $\tilde{q}=\infty$. Therefore, we have that $2<q_0<\tilde{q}$, $q_0/2\leq r<\infty$, and $v_0\in RH_{(2r/q_1)'}(w)$. Besides, 
\begin{align*}
\widetilde{k}+\frac{n\,r_w}{2r}-\frac{n\,r_w}{q_0}
>
\widetilde{k}-\frac{n\,r_w}{\tilde{q}}\geq 0.
\end{align*}
Then, taking $r_w<\widehat{r}<2$ close enough to $r_w$ we get \eqref{positiveremark}.

Besides, note that if we know that there exists $0<\widetilde{p}<q_1$ so that $v_0\in A_{\frac{q_1}{\widetilde{p}}}(w)$, then we can  find $p_0$, $\widetilde{p}<p_0<q_1$, close enough to $\widetilde{p}$ so that  $v_0\in A_{\frac{q_1}{p_0}}(w)$.
\end{remark}

Finally we present the change of angle results that we shall use.  
The following proposition is a version of \cite[Proposition 3.30]{MaPAI17} in the weighted degenerate case.
\begin{proposition}{\cite[Proposition A.2]{ChMPA16}}\label{prop:Q} 
Let $w\in A_{\widehat{r}}$ and $v\in RH_{r'}(w)$ with $1\le \widehat{r},r<\infty$. For every $1\le q\le r$, $0<\beta\leq 1$, and $t>0$, there holds
\begin{align*}
\int_{\mathbb{R}^n}\!\!\left(\int_{B(x,\beta t)}|h(y,t)| \, \frac{dw(y)}{w(B(y,\beta t))} \right)^{\!\frac{1}{q}}\!\!v(x)dw(x)
\lesssim 
\beta^{n\widehat{r}\left(\frac{1}{r}-\frac{1}{q}\right)} 
\!\!\!\int_{\mathbb{R}^n}\!\!\left(\int_{B(x,t)}|h(y,t)| \, \frac{dw(y)}{w(B(y,t))} \right)^{\!\frac{1}{q}}\!\!v(x)dw(x).
\end{align*}
\end{proposition}
Consider now the following operator:
\begin{align*}
\mathcal{A}^{\alpha}_w F(x):=\left(\iint_{\Gamma^{\alpha}(x)}|F(y,t)|^2\frac{dw(y)\,dt}{tw(B(y,t))}\right)^{\frac{1}{2}},\quad \forall x\in \R^n,
\end{align*}
where $F$ is a measurable function defined in $\R^{n+1}_+$, $\alpha >0$, and $\Gamma^{\alpha}(x)$ is the cone of aperture $\alpha$ with vertex at $x$, $\Gamma^{\alpha}(x):=\{(y,t)\in \R^{n+1}_+: |x-y|<\alpha t\}$.

\begin{proposition}{\cite[Proposition 4.9]{ChMPA16}}\label{prop:alpha}
Let $0< \alpha\leq \beta<\infty$.
\begin{list}{$(\theenumi)$}{\usecounter{enumi}\leftmargin=1cm
\labelwidth=1cm\itemsep=0.2cm\topsep=.2cm
\renewcommand{\theenumi}{\roman{enumi}}}

\item  For every $w\in A_{\widetilde{r}}$ and  $v\in A_r(w)$, $1\leq r,\widetilde{r}<\infty$,
there holds
\begin{align}\label{change-alph-1}
\norm{\mathcal{A}_w^{\beta}F}_{L^p(vdw)}\
\leq 
C \br{\frac{\beta}{\alpha}}^{\frac{n\,\widetilde{r}\,r}{p}}
 \norm{\mathcal{A}^{\alpha}_w F}_{L^p(vdw)} \quad \textrm{for all} \quad 0<p\leq 2r,
\end{align}
 where   $C\ge 1$ depends on $n$, $p$, $r$, $\widetilde{r}$, $[w]_{A_{\widetilde{r}}}$, and  $[v]_{A_r(w)}$, but it is independent of  $\alpha$ and $\beta$.

\item  For every $w\in RH_{\widetilde{s}'}$ and $v\in RH_{s'}(w)$, $1\leq s,\widetilde{s}<\infty$, there holds
\begin{align}\label{change-alph-2}
\norm{\mathcal{A}^{\alpha}_w F}_{L^p(vdw)}
\leq  
C\br{\frac{\alpha}{\beta}}^{\frac{n}{s\,\widetilde{s}\,p}} \norm{\mathcal{A}^{\beta}_w F}_{L^p(vdw)}
\quad \text{for all} \quad \frac{2}{s}\leq p<\infty,
\end{align}
 where   $C\ge 1$ depends on $n$, $p$, $s$, $\widetilde{s}$, $[w]_{RH_{\widetilde{s}'}}$, and  $[v]_{RH_{s'}(w)}$, but it is independent of  $\alpha$ and $\beta$.
\end{list}
\end{proposition}
The previous proposition was proved in  the unweighted non-degenerate   case in \cite{ChanglesAuscher}
and in the weighted non-degenerate case in \cite[Proposition 3.2]{MaPAI17}.

\section{Main results}\label{sec:main}
In this section we present our main results. In particular in Section \ref{sec:verticaconical}
we establish comparison in weighted degenerate Lebesgue spaces for the vertical and conical square functions defined in \eqref{vertical-H}-\eqref{conicalP-2}. In Section \ref{sec:boundedness} we formulate boundedness results for the operators defined in \eqref{vertical-H}-\eqref{squareroot}.
\subsection{Norm comparison for vertical and conical square functions}\label{sec:verticaconical}
In this section we study the values of $p$ where the vertical and conical square functions defined in \eqref{vertical-H}-\eqref{conicalP-2} are comparable on $L^p(vdw)$. 

We first consider the vertical and conical square functions defined via the heat or the Poisson semigroup. 

\begin{theorem}\label{thm:conical-verticalnon-gradientdegenerate}
Given  $w\in A_2$, $v\in A_{\infty}(w)$, and $f\in L^2(w)$, for every $m\in \N$, we have
 \begin{list}{$(\theenumi)$}{\usecounter{enumi}\leftmargin=1cm \labelwidth=1cm\itemsep=0.2cm\topsep=.0cm \renewcommand{\theenumi}{\alph{enumi}}}
 
  \item $\|\mathsf{s}_{m,\hh}^wf\|_{L^p(vdw)}\lesssim
 \|\mathcal{S}_{m,\hh}^wf\|_{L^p(vdw)}$, for $p\in \mathcal{W}_v^w(0,p_+(L_w))$;

\item $\|\mathcal{S}_{m,\hh}^wf\|_{L^p(vdw)}\lesssim
 \|\mathsf{s}_{m,\hh}^wf\|_{L^p(vdw)}$,  for $p\in \mathcal{W}_v^w(p_-(L_w),\infty)$;

  \item $\|\mathsf{s}_{m,\pp}^wf\|_{L^p(vdw)}\lesssim
 \|\mathcal{S}_{m,\pp}^wf\|_{L^p(vdw)}$, for $p\in \mathcal{W}_v^w(p_-(L_w),p_+(L_w))$ ;

  \item $\|\mathcal{S}_{m,\pp}^wf\|_{L^p(vdw)}\lesssim
 \|\mathsf{s}_{m,\pp}^wf\|_{L^p(vdw)}$, for $p\in \mathcal{W}_v^w(p_-(L_w),p_+(L_w)^*_w)$.

 \end{list}
In particular, for all  $p\in \mathcal{W}_v^w(p_-(L_w),p_+(L_w))$,  we have
$$
\|\mathsf{s}_{m,\hh}^wf\|_{L^p(vdw)}\approx
 \|\mathcal{S}_{m,\hh}^wf\|_{L^p(vdw)},\quad \textrm{and}\quad
\|\mathsf{s}_{m,\pp}^wf\|_{L^p(vdw)}\approx
 \|\Scal_{m,\pp}^wf\|_{L^p(vdw)}.
$$

\end{theorem}
As for the vertical and conical square functions defined via the gradient of the heat or the Poisson semigroup, we have the  following result.
\begin{theorem}\label{thm:verticalconicalgegenerategradient}

Given $w\in A_2$, $v\in A_{\infty}(w)$, and $f\in L^2(w)$, for every $K\in \N_0$,  we have
 \begin{list}{$(\theenumi)$}{\usecounter{enumi}\leftmargin=1cm \labelwidth=1cm\itemsep=0.2cm\topsep=.0cm \renewcommand{\theenumi}{\alph{enumi}}}
 \item $\|\mathsf{g}_{K,\hh}^wf\|_{L^p(vdw)}\lesssim
 \|\mathrm{G}_{K,\hh}^wf\|_{L^p(vdw)}$, for $p\in \mathcal{W}_v^w(0,q_+(L_w))$;
 \item $\|\mathrm{G}_{K,\hh}^wf\|_{L^p(vdw)}\lesssim
 \|\mathsf{g}_{K,\hh}^wf\|_{L^p(vdw)}$, for $p\in \mathcal{W}_v^w(\max\{r_w,q_-(L_w)\},\infty)$;

  \item $\|\mathsf{g}^w_{K,\pp}f\|_{L^p(vdw)}\lesssim
 \|\mathrm{G}^w_{K,\pp}f\|_{L^p(vdw)}$, for $p\in \mathcal{W}_v^w(\max\{r_w,q_-(L_w)\},q_+(L_w))$;
  \item $\|\mathrm{G}^w_{K,\pp}f\|_{L^p(vdw)}\lesssim\|\mathsf{g}^w_{K,\pp}f\|_{L^p(vdw)}$, for $p\in \mathcal{W}_v^w(\max\{r_w,q_-(L_w)\},q_+(L_w))$.
 \end{list} 
In particular, for all $p\in \mathcal{W}_v^w(\max\{r_w,q_-(L_w)\},q_+(L_w))$, we have
$$
\|\mathsf{g}_{K,\hh}^wf\|_{L^p(vdw)}\approx
 \|\Grm_{K,\hh}^wf\|_{L^p(vdw)},
\quad\textrm{and}\quad
\|\mathsf{g}_{K,\pp}^wf\|_{L^p(vdw)}\approx
 \|\Grm_{K,\pp}^wf\|_{L^p(vdw)}.
$$

\end{theorem}
\begin{remark}
The additional restriction  $\max\{r_w,q_-(L_w)\}<p<\infty$ and
$v\in A_{\frac{p}{\max\{r_w,q_-(L_w)\}}}(w)$ in Theorem \ref{thm:verticalconicalgegenerategradient} (see  \eqref{intervalrsw} ), comes from the use of Poincar\'e inequality (see Lemma \ref{lem:Gsum}). Note that in the non-degenerate case (i.e., $w\equiv 1$) we have that $r_w=1$. Then, we would obtain that 
for every $K\in \N_0$, and $p\in \mathcal{W}_v(q_-(L),q_+(L))$,
$$
\|\mathsf{g}_{K,\hh}f\|_{L^p(v)}\approx
 \|\Grm_{K,\hh}f\|_{L^p(v)}
\quad
\textrm{and}\quad
\|\mathsf{g}_{K,\pp}f\|_{L^p(v)}\approx
 \|\Grm_{K,\pp}f\|_{L^p(v)}.
$$
\end{remark}
%
%
%
%
%
\subsection{Boundedness results}\label{sec:boundedness}
In our first result we study boundedness for the conical square functions defined in \eqref{conicalH-1}-\eqref{conicalP-2} allowing odd powers of the square root of the operator $L_w$. Recall that the case of even powers was considered in \cite{ChMPA16,MaPAI17} (see Theorem \ref{thm:boundednessconicaleven}).
\begin{theorem}\label{thm:boundednessconicalodd}
Given $w\in A_2$ and $v\in A_{\infty}(w)$, for every $m\in \N$, the conical square functions $\Scal_{2m-1,\hh}^w$, $\Scal_{2m-1,\pp}^w$, $\Grm_{2m-1,\pp}^w$, and $\Grm_{2m-1,\hh}^w$ can be extended to bounded operators on $L^p(vdw)$, 
for all $p\in \mathcal{W}_v^w(p_-(L_w),p_+(L_w)^{2m-1,*}_w)$.
\end{theorem}
As a consequence of the above results and Theorem \ref{thm:boundednessconicaleven}, in the following results,
 we obtain boundedness of the vertical square functions defined in \eqref{vertical-H}-\eqref{vertical-P}.
\begin{theorem}\label{thm:boundednessverticalheat}
Given $w\in A_2$ and $v\in A_{\infty}(w)$, for all $m\in \N$, $K\in \N_0$, and $p\in\mathcal{W}_v^w(p_-(L_w),p_+(L_w))$ the operators $\mathsf{s}_{m,\hh}^w$ and $\mathsf{s}_{K,\pp}^w$ can be extended to  bounded operators on $L^p(vdw)$.
\end{theorem}
\begin{theorem}\label{thm:boundednessverticalpoisson}
Given $w\in A_2$ and $v\in A_{\infty}(w)$, for all $m\in \N$, $K\in \N_0$, and $p\in\mathcal{W}_v^w(q_-(L_w),q_+(L_w))$ the operators $\mathsf{g}_{m,\hh}^w$ and $\mathsf{g}_{K,\pp}^w$ can be extended to  bounded operators on $L^p(vdw)$.
\end{theorem}
Finally, in the next results, we improve the range of values of $p$ where $\Ncal_{\pp}^w$ and $\sqrt{L_w}$ are respectively known to be bounded on $L^p(vdw)$. Besides, note that Theorem \ref{thm:boundednessnon-tangential} improves the range of values of $p$ where $\Ncal_{\pp}^w$ is bounded even in the unweighted or weighted non-degenerate cases (see \cite{HMay09,MaPAII17,May10}). The boundedness of $\sqrt{L_w}$ on $L^p(vdw)$ was also studied in \cite{CMR15}. Additionally, see \cite{AMIII06} for the weighted non-degenerate case,  and \cite{Au07} for the  unweighted non-degenerate case.
\begin{theorem}\label{thm:boundednessnon-tangential}
Given $w\in A_2$ and $v\in A_{\infty}(w)$, the non-tangential maximal function
 $\Ncal_{\pp}^w$ can be extended to a bounded operator on $L^p(vdw)$, for all $p\in \mathcal{W}_v^w(p_-(L_w),p_+(L_w)^{*}_w)$.
\end{theorem}
\begin{theorem}\label{thm:boundednesssaquareroot}
Given $w\in A_2$ and $v\in A_{\infty}(w)$, assume that $\mathcal{W}_v^w(\max\{r_w,p_-(L_w)\},p_+(L_w))\neq \emptyset$. Then, the operator $\sqrt{L_w}$ can be extended to a bounded operator from  $\dot{W}^{1,p}(vdw)$ to $L^p(vdw)$,
for all  $p\in \mathcal{W}_v^w\big(\max\left\{r_w,(p_-(L_w))_{w,*}\right\},p_+(L_w)\big)$.
\end{theorem}
The space $\dot{W}^{1,p}(vdw)$ is defined as the completion of $\big\{f\in C_0^{\infty}(\R^n): \nabla f\in L^p(vdw)\big\}$ under the semi-norm $\|f\|_{\dot{W}^{1,p}(vdw)}:=\|\nabla f\|_{L^p(vdw)}$.
%
%
\section{Auxiliary results}\label{sec:auxiliary}
In this section we obtain some results that will simplify the proofs of the theorems formulated in the previous section.

First of all, 
consider the following conical and vertical square functions:
$$
\mathcal{V}F(x):=\left(\int_0^{\infty}|T_{t^2}F(y,t)|^2\frac{dt}{t}\right)^{\frac{1}{2}}
\quad\textrm{and}\quad
\mathcal{S}F(x):=\left(\int_0^{\infty}\int_{B(x,t)}|T_{t^2}F(y,t)|^2\frac{dw(y)\,dt}{tw(B(y,t))}\right)^{\frac{1}{2}},
$$
where  $\{T_t\}_{t>0}$ is  a family of sublinear operators and $F$ is a measurable function defined in $\R^{n+1}_+$. 

Note that given $w\in A_{\infty}$ and $2<q_0<\infty$, for all $v_0\in RH_{\left(\frac{q_0}{2}\right)'}(w)$, there holds
\begin{align}\label{verticalsinRH}
\|\mathcal{V}F\|_{L^2(v_0dw)}
\lesssim
\left(\int_{\R^n}\int_{0}^{\infty}
\left(\dashint_{B(x,t)}|{T}_{t^2}F(y,t)|^{q_0}dw(y)\right)^{\frac{2}{q_0}}\frac{dt}{t}v_0(x)dw(x)\right)^{\frac{1}{2}}.
\end{align}
Indeed, by Fubini's theorem, \eqref{doublingcondition}, and \eqref{RHsinpesoconpeso}, we get
\begin{align*}
\|\mathcal{V}F\|_{L^2(v_0dw)}&=
\left(\int_0^{\infty}\int_{\R^n}\left|T_{t^2}F(y,t)\right|^{2}
\int_{B(y,t)}\frac{d(v_0w)(x)}{v_0w(B(y,t))}
v_0(y)dw(y)\frac{dt}{t}\right)^{\frac{1}{2}}
\\\nonumber
&
\lesssim
\left(\int_0^{\infty}\int_{\R^n}\dashint_{B(x,t)}\left|T_{t^2}F(y,t)\right|^{2}d(v_0w)(y)
v_0(x)dw(x)
\frac{dt}{t}\right)^{\frac{1}{2}}
\\\nonumber
&
\lesssim
\left(\int_{\R^n}\int_0^{\infty}\left(\dashint_{B(x,t)}\left|T_{t^2}F(y,t)\right|^{q_0}dw(y)\right)^{\frac{2}{q_0}}\frac{dt}{t}
v_0(x)dw(x)
\right)^{\frac{1}{2}}.
\end{align*}

From this and under some assumptions on
 the family $\{T_t\}_{t>0}$, we obtain comparison results between $\mathcal{V}$ and $\Scal$, as we see in  the following proposition.
\begin{proposition}\label{prop:comparison-general}
Given $w\in A_\infty$ and $v\in A_{\infty}(w)$. Let $\{T_t\}_{t>0}$ be a family of sublinear operators  and $0<p_0<2<q_0<\infty$. Consider $B:=B(x,t)$, for $(x,t)\in \R^{n+1}_+$,  a measurable function $F$ defined in $\R^{n+1}_+$, and the following conditions:
\begin{list}{$(\theenumi)$}{\usecounter{enumi}\leftmargin=1cm \labelwidth=1cm\itemsep=0.2cm\topsep=.0cm \renewcommand{\theenumi}{\alph{enumi}}}

\item[(i)] For any constant $c>0$ there exists a constant $C>0$ such that $F(y,ct)=C F(y,t)$;

\item[(ii)] $w(B)^{-\frac{1}{2}}\|\chi_{B}T_{t^2}F\|_{L^{2}(w)}\lesssim
\sum_{j\geq 1}e^{-c4^j}
w(2^{j+1}B)^{-\frac{1}{p_0}}\|\chi_{2^{j+1}B}T_{t^2/2}F\|_{L^{p_0}(w)}$;

\item[(iii)] $w(B)^{-\frac{1}{q_0}}\|\chi_{B}T_{t^2}F\|_{L^{q_0}(w)}\lesssim
\sum_{j\geq 1}e^{-c4^j}
w(2^{j+1}B)^{-\frac{1}{2}}\|\chi_{2^{j+1}B}T_{t^2/2}F\|_{L^{2}(w)}$.
\end{list}
Then, assuming that  $F$ satisfies condition $(i)$, there hold:
\begin{list}{$(\theenumi)$}{\usecounter{enumi}\leftmargin=1cm \labelwidth=1cm\itemsep=0.2cm\topsep=.0cm \renewcommand{\theenumi}{\alph{enumi}}}

\item If $\{T_t\}_{t>0}$ satisfies condition $(ii)$,  
\begin{align*}
 \|\Scal F\|_{L^p(vdw)}\lesssim \|\mathcal{V}F\|_{L^p(vdw)},\quad \forall p\in \mathcal{W}_v^w(p_0,\infty).
\end{align*}

\item If $\{T_t\}_{t>0}$ satisfies condition $(iii)$,
\begin{align*}
 \|\mathcal{V}F\|_{L^p(vdw)}\lesssim \|\Scal F\|_{L^p(vdw)},  \quad \forall p\in \mathcal{W}_v^w(0,q_0).
\end{align*}
\end{list}
In particular, if $F$ satisfies condition $(i)$ and $\{T_t\}_{t>0}$ satisfies conditions $(ii)$, and $(iii)$, we have
\begin{align*}
 \|\mathcal{V}F\|_{L^p(vdw)}\approx \|\Scal F\|_{L^p(vdw)}, \quad \forall p\in \mathcal{W}_v^w(p_0,q_0).
\end{align*}
\end{proposition}
\begin{proof}
We shall proceed by extrapolation to prove both part ($a$) and part ($b$). Indeed, to obtain part ($a$), in view of \eqref{intervalrsw}, and from Theorem \ref{theor:extrapol}, part ($a$),  it is enough to prove 
\begin{align}\label{conicalverticalp=2}
\|\Scal F\|_{L^2(v_0dw)}\lesssim \|\mathcal{V}F\|_{L^2(v_0dw)}, \quad\forall\,v_0\in A_{\frac{2}{p_0}}(w).
\end{align}
As for proving part ($b$), in view of \eqref{intervalrsw}, and from Theorem \ref{theor:extrapol}, part ($b$), we just need to show 
\begin{align}\label{verticalconicalp=2}
\|\mathcal{V}F\|_{L^2(v_0dw)}\lesssim \|\Scal F\|_{L^2(v_0dw)},\quad \forall\,v_0\in RH_{\left(\frac{q_0}{2}\right)'}.
\end{align}

\medskip

We first prove \eqref{conicalverticalp=2}. 
By \eqref{doublingcondition}, condition $(ii)$, and \eqref{Asinpesoconpeso} (recall that $v_0\in A_{\frac{2}{p_0}}(w)$),  we get
\begin{align*}
\Scal F(x)&\lesssim
\left(\int_0^{\infty}\dashint_{B(x,t)}\left|T_{t^2}F(y,t)\right|^2\frac{dw(y)dt}{t}\right)^{\frac{1}{2}}
\\
&\lesssim \sum_{j\geq 1} e^{-c4^j}
\left(\int_0^{\infty}\left(\dashint_{B(x,2^{j+1}t)}\left|T_{\frac{t^2}{2}}F(y,t)\right|^{p_0}dw(y)\right)^{\frac{2}{p_0}}\frac{dt}{t}\right)^{\frac{1}{2}}
\\
&\lesssim \sum_{j\geq 1} e^{-c4^j}
\left(\int_0^{\infty}\dashint_{B(x,2^{j+1}t)}\left|T_{\frac{t^2}{2}}F(y,t)\right|^{2}d(v_0w)(y)\frac{dt}{t}\right)^{\frac{1}{2}}
\\
&\lesssim \sum_{j\geq 1} e^{-c4^j}
\left(\int_0^{\infty}\int_{B(x,2^{j+1}t)}\left|T_{\frac{t^2}{2}}F(y,t)\right|^{2}\frac{d(v_0w)(y)}{v_0w(B(y,2^{j+1}t))}\frac{dt}{t}\right)^{\frac{1}{2}},
\end{align*}
where we have used again \eqref{doublingcondition} in the last inequality.

Hence, by Fubini's theorem, changing the variable   $t$ into $\sqrt{2}t$, and by condition $(i)$, we conclude that
\begin{multline*}
\|\Scal F\|_{L^2(v_0dw)}
\lesssim \sum_{j\geq 1} e^{-c4^j}
\left(\int_0^{\infty}\int_{\R^n}\left|T_{\frac{t^2}{2}}F(y,t)\right|^{2}\int_{B(y,2^{j+1}t)}\frac{d(v_0w)(x)}{v_0w(B(y,2^{j+1}t))}v_0(y)dw(y)\frac{dt}{t}\right)^{\frac{1}{2}}
\\
\lesssim \sum_{j\geq 1} e^{-c4^j}
\left(\int_0^{\infty}\int_{\R^n}\left|T_{t^2}F(y,\sqrt{2}t)\right|^{2}v_0(y)dw(y)\frac{dt}{t}\right)^{\frac{1}{2}}
\\
\lesssim 
\left(\int_{\R^n}\int_0^{\infty}\left|T_{t^2}F(y,t)\right|^{2}\frac{dt}{t}v_0(y)dw(y)\right)^{\frac{1}{2}}
=\|\mathcal{V}F\|_{L^2(v_0dw)}.
\end{multline*}

\medskip

We next prove \eqref{verticalconicalp=2}. To this end, fist apply \eqref{verticalsinRH} and  condition $(iii)$. Then, changing the variable $t$ into $\sqrt{2}t$, by condition $(i)$, \eqref{doublingcondition}, and Proposition \ref{prop:alpha}, we obtain that
\begin{align*}
\|\mathcal{V}F\|_{L^2(v_0dw)}
&
\lesssim
\left(\int_{\R^n}\int_0^{\infty}\left(\dashint_{B(x,t)}\left|T_{t^2}F(y,t)\right|^{q_0}dw(y)\right)^{\frac{2}{q_0}}\frac{dt}{t}
v_0(x)dw(x)
\right)^{\frac{1}{2}}
\\
&
\lesssim \sum_{j\geq 1} e^{-c4^j}
\left(\int_{\R^n}\int_0^{\infty}\dashint_{B(x,2^{j+1}t)}\left|T_{\frac{t^2}{2}}F(y,t)\right|^{2}dw(y)
\frac{dt}{t}v_0(x)dw(x)
\right)^{\frac{1}{2}}
\\
&
\lesssim \sum_{j\geq 1} e^{-c4^j}
\left(\int_{\R^n}\int_0^{\infty}\dashint_{B(x,2^{j+2}t)}\left|T_{t^2}F(y,\sqrt{2}t)\right|^{2}dw(y)
\frac{dt}{t}v_0(x)dw(x)
\right)^{\frac{1}{2}}
\\
&
\lesssim \sum_{j\geq 1} e^{-c4^j}
\left(\int_{\R^n}\int_0^{\infty}\int_{B(x,2^{j+2}t)}\left|T_{t^2} F(y,t)\right|^{2}\frac{dw(y)dt}{tw(B(y,t))}
v_0(x)dw(x)\right)^{\frac{1}{2}}
\\
&
\lesssim \sum_{j\geq 1} e^{-c4^j}2^{j\theta_{v_0,w}}\|\Scal F\|_{L^2(v_0dw)}
\\
&
\lesssim
\|\Scal F\|_{L^2(v_0dw)},
\end{align*}
which concludes the proof.
\end{proof}
%
%
%
%
%
In the following proposition, for every $m\in \N$, we 
compare the norms on  $L^p(vdw)$ of $\Scal_{m,\hh}^w$  and $\Scal_{m+1,\hh}^w$. This will allow us to obtain Theorem \ref{thm:boundednessconicalodd} for $\Scal_{2m-1,\hh}^w$,  from Theorem \ref{thm:boundednessconicaleven}.
 This result  is proved in \cite{ChMPA18} for $m=1$, 
 (see also \cite[Proposition 9.4]{PA17} \cite[Corollary 4.17 and (5.21)]{HMayMc11}, where  the weighted and unweighted non-degenerate cases were considered respectively).  
\begin{proposition}\label{prop:widetildeS-heatS}
Given $w\in A_{2}$,  $v\in A_{\infty}(w)$, $m\in \N$, and $f\in L^2(w)$, there hold
\begin{list}{$(\theenumi)$}{\usecounter{enumi}\leftmargin=1cm \labelwidth=1cm\itemsep=0.2cm\topsep=.2cm \renewcommand{\theenumi}{\alph{enumi}}}

\item $\|\Scal_{m+1,\hh}^wf\|_{L^p(vdw)}\lesssim
\|\Scal_{m,\hh}^wf\|_{L^p(vdw)}$,\, for all\, $p\in \mathcal{W}_v^w(0,p_+(L_w)^{m+1,*}_w)$;

\item $\|\Scal_{m,\hh}^wf\|_{L^p(vdw)}\lesssim\|\Scal_{m+1,\hh}^wf\|_{L^p(vdw)}$,\, for all\, $p\in \mathcal{W}_v^w(0,p_+(L_w)^{m,*}_w)$.
\end{list}
In particular, we have
\begin{align*}
\|\Scal_{m,\hh}^wf\|_{L^p(vdw)}\approx\|\Scal_{m+1,\hh}^wf\|_{L^p(vdw)},\, \textrm{ for all }\, p\in \mathcal{W}_v^w(0,p_+(L_w)^{m,*}_w).
\end{align*}
\end{proposition}
\begin{proof}\label{extrapol1}
We shall use extrapolation to prove both inequalities. Indeed, Theorem \ref{theor:extrapol}, part ($b$) (or Theorem \ref{theor:extrapol}, part ($c$), if $p_+(L_w)^{m+1,*}_w=\infty$ or $p_+(L_w)^{m,*}_w=\infty$) allows us to obtain  $(a)$ from
\begin{align}\label{extrapol-1}
\|\Scal_{m+1,\hh}^wf\|_{L^2(v_0dw)}\lesssim\|{\Scal^w_{m,\hh}}f\|_{L^2(v_0dw)},\quad \forall v_0\in  RH_{\left(\frac{p_+(L_w)^{m+1,*}_w}{2}\right)'}(w)
\end{align}
and $(b)$ from 
\begin{align}\label{extrapol2}
\|\Scal_{m,\hh}^wf\|_{L^2(v_0dw)}\lesssim\|\Scal_{m+1,\hh}^wf\|_{L^2(v_0dw)}\quad \forall v_0\in  RH_{\left(\frac{p_+(L_w)^{m,*}_w}{2}\right)'}(w).
\end{align}

Note that Remark \ref{remark:choicereverseholder} with $q_1=2$, $\tilde{q}=p_+(L_w)$, and $\widetilde{k}$ equal to $m+1$ or $m$ implies that we can find $\widehat{r}$, $q_0$, and $r$ such that $r_w<\widehat{r}<2$, $2<q_0<p_+(L_w)$, and $q_0/2\leq r<\infty$ so that $v_0\in RH_{r'}(w)$ and
\begin{align*}
\widetilde{k}+\frac{n\,\widehat{r}}{2r}-\frac{n\,\widehat{r}}{q_0}
> 0.
\end{align*}
We also observe that for $M\in \N$ large enough the above inequality gives
\begin{align}\label{positivetilde}
\widetilde{k}+\frac{n\,\widehat{r}}{2r}-\frac{n\,\widehat{r}}{q_0}
-\frac{1}{M}> 0.
\end{align}
After these observations, we prove \eqref{extrapol-1} and \eqref{extrapol2}.

\medskip

We first prove \eqref{extrapol-1}.  
By \eqref{squareroot}, Minkowski's integral inequality, and  \eqref{doublingcondition} (note that $B(x,t)\subset B(y,2t)$, for all $y\in B(x,t)$), we obtain
\begin{multline*}
\Scal_{m+1,\hh}^wf(x)
\lesssim\!\!
\left(\int_0^{\infty}\left(\int_{0}^{t}\left(\int_{B(x,t)}|tsL_we^{-s^2L_w}(t\sqrt{L_w})^me^{-t^2L_w}f(y)|^2dw(y)\right)^{\frac{1}{2}}\frac{ds}{s}\right)^2\!\!\frac{dt}{tw(B(x,t))}\right)^{\!\!\frac{1}{2}}
\\
+
\left(\int_0^{\infty}\left(\int_{t}^{\infty}\left(\int_{B(x,t)}\left|tsL_we^{-s^2L_w}(t\sqrt{L_w})^me^{-t^2L_w}f(y)\right|^2dw(y)\right)^{\frac{1}{2}}\frac{ds}{s}\right)^2\!\!\frac{dt}{tw(B(x,t))}\right)^{\!\!\frac{1}{2}}
=:I+II.
\end{multline*}
In the case that $s<t$, for $F(y,t):=(t\sqrt{L_w})^me^{-\frac{t^2}{2}L_w}f(y)$, we use the fact that $\tau L_we^{-\tau L_w}\in \mathcal{F}(L^2(w)-L^2(w))$
\begin{align*}
I&
=
\left(\int_0^{\infty}\!\!\left(\int_{0}^{t}\!\frac{st}{t^2/2+s^2}\left(\int_{B(x,t)}\!\!\!\!\!\!\big|(s^2+{t^2}/{2})L_we^{-\left(s^2+\frac{t^2}{2}\right)L_w}F(y,t)\big|^2dw(y)\right)^{\frac{1}{2}}\!\!\frac{ds}{s}\right)^{\!\!2}\!\!\frac{dt}{tw(B(x,t))}\right)^{\!\!\frac{1}{2}}
\\
&
\lesssim
\sum_{j\geq 1}e^{-c4^{j}}\left(
\int_0^{\infty}\left(\int_{0}^{t}\frac{s}{t}\frac{ds}{s}\right)^{2}\int_{B(x,2^{j+1}t)}|F(y,t)|^2\frac{dw(y)\,dt}{tw(B(x,t))}\right)^{\frac{1}{2}}
\\
&
\lesssim
\sum_{j\geq 1}e^{-c4^{j}}\left(
\int_0^{\infty}\int_{B(x,2^{j+2}t)}|F(y,\sqrt{2}t)|^2\frac{dw(y)\,dt}{tw(B(y,t))}\right)^{\frac{1}{2}},
\end{align*}
where in the last inequality we have changed the variable $t$ into $\sqrt{2}t$ and used \eqref{doublingcondition}.
Then, applying change of angle (Proposition \ref{prop:alpha}), we conclude that
\begin{align*}
\|I\|_{L^2(v_0dw)}
\lesssim
\sum_{j\geq 1}e^{-c4^j}2^{j\theta_{v_0,w}}
\|\Scal_{m,\hh}^wf\|_{L^2(v_0dw)}
\lesssim\|\Scal_{m,\hh}^wf\|_{L^2(v_0dw)}.
\end{align*}
As for the estimate of $II$, consider $\widetilde{F}(y,s):=(s\sqrt{L_w})^{m+2}e^{-s^2L_w}f(y)$,  apply Cauchy-Schwartz's inequality in the integral in $s$, the fact that $e^{-\tau L_w}\in \mathcal{F}(L^2(w)- L^2(w))$,\eqref{doublingcondition},     Jensen's inequality in the integral in $y$ ($q_0>2$), Fubini's theorem, and \eqref{doublingcondition} to obtain
\begin{align*}
II
&
=
\left(\int_0^{\infty}\left(\int_{t}^{\infty}\left(\frac{t}{s}\right)^{\frac{1}{M}}
\left(\frac{t}{s}\right)^{m+1-\frac{1}{M}}\left(\int_{B(x,t)}\big|e^{-t^2L_w}\widetilde{F}(y,s)\big|^2\frac{dw(y)}{w(B(x,t))}\right)^{\frac{1}{2}}\frac{ds}{s}\right)^2\frac{dt}{t}\right)^{\frac{1}{2}}
\\
&
\lesssim
\left(\int_0^{\infty}\left(\int_{t}^{\infty}\left(\frac{t}{s}\right)^{\frac{2}{M}}\frac{ds}{s}\right)\left(\int_{t}^{\infty}
\left(\frac{t}{s}\right)^{2(m+1)-\frac{2}{M}}\dashint_{B(x,t)}\big|e^{-t^2L_w}\widetilde{F}(y,s)\big|^2dw(y)\frac{ds}{s}\right)\frac{dt}{t}\right)^{\frac{1}{2}}
\\
&
\lesssim
\sum_{j\geq 1}e^{-c4^j}
\left(\int_0^{\infty}\int_{t}^{\infty}\left(\frac{t}{s}\right)^{2(m+1)-\frac{2}{M}}\dashint_{B(x,2^{j+1}t)}\big|\widetilde{F}(y,s)\big|^2dw(y)\frac{ds}{s}\frac{dt}{t}\right)^{\frac{1}{2}}
\\&
\lesssim
\sum_{j\geq 1}e^{-c4^j}\left(\int_0^{\infty}\int_{t}^{\infty}\left(\frac{t}{s}\right)^{2(m+1)-\frac{2}{M}}\left(\dashint_{B(x,2^{j+1}t)}\big|\widetilde{F}(y,s)\big|^{q_0}dw(y)\right)^{\frac{2}{q_0}}\frac{ds}{s}\frac{dt}{t}\right)^{\frac{1}{2}}
\\&
\lesssim
\sum_{j\geq 1}e^{-c4^j}\left(\int_0^{\infty}\int_{0}^{s}\left(\frac{t}{s}\right)^{2(m+1)-\frac{2}{M}}\left(\int_{B(x,2^{j+1}st/s)}\big|\widetilde{F}(y,s)\big|^{q_0}\frac{dw(y)}{w(B(y,2^{j+1}st/s))}\right)^{\frac{2}{q_0}}\frac{dt}{t}\frac{ds}{s}\right)^{\frac{1}{2}}.
\end{align*}
Note now that  applying Propositions \ref{prop:Q} with $\beta=t/s<1$ and $q=q_0/2$, and \eqref{doublingcondition},
\begin{multline*}
\int_{\R^n}\left(\int_{B(x,2^{j+1}st/s)}\big|\widetilde{F}(y,s)\big|^{q_0}\frac{dw(y)}{w(B(y,2^{j+1}st/s))}\right)^{\frac{2}{q_0}}v_0(x)dw(x)
\\
\lesssim \left(\frac{t}{s}\right)^{n\widehat{r}\left(\frac{1}{r}-\frac{2}{q_0}\right)}
\int_{\R^n}\left(\dashint_{B(x,2^{j+1}s)}\big|\widetilde{F}(y,s)\big|^{q_0}dw(y)\right)^{\frac{2}{q_0}}v_0(x)dw(x),
\end{multline*}
recall the choices of $r$, $\widehat{r}$, and $q_0$ at the beginning of the proof.
Besides, since   $\tau L_we^{-\tau L_w}\in \Ocal(L^{2}(w)-L^{q_0}(w))$, for $\widehat{F}(y,s):=(s\sqrt{L_w})^me^{-\frac{s^2}{2}L_w}$
\begin{multline*}
\left(\dashint_{B(x,2^{j+1}s)}\big|\widetilde{F}(y,s)\big|^{q_0}dw(y)\right)^{\frac{1}{q_0}}
\lesssim
\left(\dashint_{B(x,2^{j+1}s)}\left|\frac{s^2}{2}L_we^{-\frac{s^2}{2}L_w}\widehat{F}(y,s)\right|^{q_0}dw(y)\right)^{\frac{1}{q_0}}
\\
\lesssim
\sum_{l\geq 1}e^{-c4^l}2^{j\theta_2}
\left(\dashint_{B(x,2^{l+j+2}s)}\big|\widehat{F}(y,s)\big|^{2}dw(y)\right)^{\frac{1}{2}}.
\end{multline*}
Consequently, applying Fubini's theorem, \eqref{positivetilde} with $\widetilde{k}=m+1$, changing the variable $s$ into $\sqrt{2}s$, and by \eqref{doublingcondition} and change of angle (Proposition \ref{prop:alpha}), we get
\begin{align*}
&\|II\|_{L^2(v_0dw)}
\\
&
\lesssim\!\!
\sum_{j\geq 1}e^{-c4^j}\!\sum_{l\geq 1}e^{-c4^l}\!\!\left(\!\int_{\R^n}\!\int_0^{\infty}\!\!\int_{0}^{s}\!\left(\frac{t}{s}\right)^{\!\!2(m+1)+n\widehat{r}\left(\frac{1}{r}-\frac{2}{q_0}\right)-\frac{2}{M}}\!\frac{dt}{t}
\dashint_{\!B(x,2^{l+j+2}s)}\!\big|\widehat{F}(y,s)\big|^{\!2}dw(y)\frac{ds}{s}v_0(x)dw(x)\!\!\right)^{\!\!\!\frac{1}{2}}
\\&
\lesssim
\sum_{j\geq 1}e^{-c4^j}
\sum_{l\geq 1}e^{-c4^l}\left(\int_{\R^n}\int_0^{\infty}
\int_{B(x,2^{l+j+3}s)}\big|\widehat{F}(y,\sqrt{2}s)\big|^{2}\frac{dw(y)ds}{sw(B(y,s))}v_0(x)dw(x)\right)^{\frac{1}{2}}
\\&
\lesssim
\sum_{j\geq 1}e^{-c4^j}\sum_{l\geq 1}e^{-c4^l}2^{(l+j)\theta_{v_0,w}}\|\Scal_{m,\hh}^wf\|_{L^2(v_0dw)}
\\&
\lesssim
\|\Scal_{m,\hh}^wf\|_{L^2(v_0dw)}.
\end{align*}

\medskip

As for proving \eqref{extrapol2}, by \eqref{doublingcondition}, \eqref{squareroot}, and Minkowski's integral inequality, we obtain
\begin{multline*}
\Scal_{m,\hh}^wf(x)\lesssim\!\!
\left(\int_0^{\infty}\left(\int_0^{t}\left(\int_{B(x,t)}|tsL_we^{-s^2L_w}(t\sqrt{L_w})^{m-1}e^{-t^2L_w}f(y)|^2dw(y)\right)^{\frac{1}{2}}\frac{ds}{s}\right)^{{2}}\!\!\frac{dt}{tw(B(x,t))}\right)^{\!\!\frac{1}{2}}
\\
+
\left(\int_0^{\infty}\left(\int_{t}^{\infty}\left(\int_{B(x,t)}|tsL_we^{-s^2L_w}(t\sqrt{L_w})^{m-1}e^{-t^2L_w}f(y)|^2dw(y)\right)^{\frac{1}{2}}\frac{ds}{s}\right)^{{2}}\!\!\frac{dt}{tw(B(x,t))}\right)^{\!\!\frac{1}{2}}
=:\widetilde{I}+\widetilde{II}.
\end{multline*}
We first estimate $\widetilde{I}$. Using that $s<t$ and applying the fact that $e^{-\tau L_w}\in \mathcal{F}(L^2(w)-L^2(w))$, and \eqref{doublingcondition}, we have
\begin{align*}
\widetilde{I} &=
\left(\int_0^{\infty}\left(\int_0^{t}\frac{s}{t}\left(\int_{B(x,t)}|e^{-s^2L_w}(t\sqrt{L_w})^{m+1}e^{-t^2L_w}f(y)|^2dw(y)\right)^{\frac{1}{2}}\frac{ds}{s}\right)^{{2}}\frac{dt}{tw(B(x,t))}\right)^{\frac{1}{2}}
\\
&\lesssim
\sum_{j\geq 1}e^{-c4^j}\left(\int_0^{\infty}\left(\int_0^{t}\frac{s}{t}\frac{ds}{s}\right)^2\int_{B(x,2^{j+1}t)}|(t\sqrt{L_w})^{m+1}e^{-t^2L_w}f(y)|^2\frac{dw(y)\,dt}{tw(B(y,t))}\right)^{\frac{1}{2}}
\\
&\lesssim
\sum_{j\geq 1}e^{-c4^j}
\left(\int_0^{\infty}\int_{B(x,2^{j+1}t)}|(t\sqrt{L_w})^{m+1}e^{-t^2L_w}f(y)|^2\frac{dw(y)\,dt}{tw(B(y,t))}\right)^{\frac{1}{2}}.
\end{align*}
Therefore, applying change of angle (Proposition \ref{prop:alpha}), we get
$$
\|\widetilde{I}\|_{L^2(v_0dw)}
\lesssim
\sum_{j\geq 1}e^{-c4^j}2^{j\theta_{v_0,w}}\|\Scal_{m+1,\hh}^wf\|_{L^2(v_0dw)}\lesssim \|\Scal_{m+1,\hh}^wf\|_{L^2(v_0dw)}.
$$
The estimate of $\widetilde{II}$ is very similar to that of $II$ (in the proof of \eqref{extrapol-1}), so we skip some details. We  apply again the fact that $e^{-\tau L_w}\in \mathcal{F}(L^2(w)-L^2(w))$, Cauchy-Schwartz's inequality in the integral in $s$,  Jensen's inequality in the integral in $y$,  Fubini's theorem, and
\eqref{doublingcondition} in order to obtain
\begin{align*}
\widetilde{II}
&\lesssim\!\!
\sum_{j\geq 1}e^{-c4^j}\!\!
\left(\int_0^{\infty}\left(\int_{t}^{\infty}\left(\frac{t}{s}\right)^{\frac{1}{M}}\left(\frac{t}{s}\right)^{m-\frac{1}{M}}\left(
\dashint_{B(x,2^{j+1}t)}|(s\sqrt{L_w})^{m+1}e^{-s^2L_w}f(y)|^{2}dw(y)\right)^{\!\!\frac{1}{2}}\frac{ds}{s}\right)^{\!\!2}\frac{dt}{t}\right)^{\!\!\frac{1}{2}}
\\
&\lesssim\!\!
\sum_{j\geq 1}e^{-c4^j}\!\!
\left(\int_0^{\infty}\int_{t}^{\infty}\left(\frac{t}{s}\right)^{2m-\frac{2}{M}}\left(
\dashint_{B(x,2^{j+1}t)}|(s\sqrt{L_w})^{m+1}e^{-s^2L_w}f(y)|^{q_0}dw(y)\right)^{\frac{2}{q_0}}\frac{ds}{s}\frac{dt}{t}\right)^{\!\!\frac{1}{2}}
\\
&\lesssim\!\!
\sum_{j\geq 1}e^{-c4^j}\!\!
\left(\int_0^{\infty}\!\!\int_{0}^{s}\!\!\left(\frac{t}{s}\right)^{\!\!2m-\frac{2}{M}}\!\!\left(
\int_{B(x,2^{j+1}st/s)}|(s\sqrt{L_w})^{m+1}e^{-s^2L_w}f(y)|^{q_0}\frac{dw(y)}{w(B(y,2^{j+1}st/s))}\right)^{\!\!\frac{2}{q_0}}\!\!\frac{dt}{t}\frac{ds}{s}\right)^{\!\!\frac{1}{2}}.
\end{align*}
Note that,    Proposition \ref{prop:Q} with $\beta=t/s<1$ and $q=q_0/2$, and  \eqref{doublingcondition} imply
\begin{multline*}
\int_{\R^n}
\left(\int_{B(x,2^{j+1}st/s)}|(s\sqrt{L_w})^{m+1}e^{-s^2L_w}f(y)|^{q_0}\frac{dw(y)}{B(y,2^{j+1}st/s)}\right)^{\frac{2}{q_0}}v_0(x)dw(x)
\\
\lesssim
\left(\frac{t}{s}\right)^{n\,\widehat{r}\left(\frac{1}{r}-\frac{2}{q_0}\right)}\int_{\R^n}
\left(\dashint_{B(x,2^{j+1}s)}|(s\sqrt{L_w})^{m+1}e^{-s^2L_w}f(y)|^{q_0}dw(y)\right)^{\frac{2}{q_0}}v_0(x)dw(x).
\end{multline*}
Besides, since $e^{-\tau L_w}\in \mathcal{O}(L^{2}(w)\rightarrow L^{q_0}(w))$
\begin{multline*}
\left(\dashint_{B(x,2^{j+1}s)}\!\!|(s\sqrt{L_w})^{m+1}e^{-s^2L_w}f(y)|^{q_0}dw(y)\right)^{\!\!\frac{1}{q_0}}
\\
\lesssim
\sum_{l\geq 1}e^{-c4^l}2^{j\theta_2}\left(\dashint_{B(x,2^{j+l+2}s)}\!\!|(s\sqrt{L_w})^{m+1}e^{-\frac{s^2}{2}L_w}f(y)|^{2}dw(y)\right)^{\!\!\frac{1}{2}}.
\end{multline*}
Hence, applying Fubini's theorem, \eqref{positivetilde} with $k=m$, changing the variable $s$ into $\sqrt{2}s$, and by \eqref{doublingcondition} and  Proposition \ref{prop:alpha} we have
\begin{align*}
\|\widetilde{II}\|_{L^2(v_0dw)}&
\lesssim\!\!
\sum_{j\geq 1}e^{-c4^j} \sum_{l\geq 1}e^{-c4^l}
\!\left(\int_{\R^n}\!\!\int_0^{\infty}\!\!\int_{B(x,2^{j+l+3}s)}\!\!|(s\sqrt{L_w})^{m+1}e^{-s^2L_w}f(y)|^2\frac{
dw(y)\,ds}{sw(B(y,s))}v_0(x)dw(x)\!\right)^{\!\!\frac{1}{2}}
\\&
\lesssim\!\!
\sum_{j\geq 1}e^{-c4^j} \sum_{l\geq 1}e^{-c4^l}2^{(l+j)\theta_{w,v_0}}\|\Scal_{m+1,\hh}^wf\|_{L^2(v_0dw)}
\lesssim\|\Scal_{m+1,\hh}^wf\|_{L^2(v_0dw)}.
\end{align*} 
This and the estimate obtained for $\|\widetilde{I}\|_{L^2(v_0dw)}$ give us \eqref{extrapol2}.
\end{proof}
Our next result  will be useful  in the proof of Theorem \ref{thm:conical-verticalnon-gradientdegenerate} parts $(c)$ and $(d)$.
Given  a measurable function $F$ defined in $\R^{n+1}_+$, consider the following vertical and conical operators:
\begin{align*}
\widetilde{\mathcal{V}}F(x):=\left(\int_{\R^n}\int_0^{\infty}|F(y,t)|^{2}\frac{dt}{t}v_0(x)dw(x)\right)^{\frac{1}{2}}
\end{align*}
and
\begin{align*}
\widetilde{\Scal}F(x):=\left(\int_{\R^n}\!\!\int_0^{\infty}\!\!\!\int_{B(x,t)}\!\!|F(y,t)|^{2}\frac{dw(y)\,dt}{tw(B(y,t))}v_0(x)dw(x)\right)^{\!\!\frac{1}{2}}.
\end{align*}

\begin{lemma}\label{lemma:poissonmaximal}
Given $w\in A_2$, $0<p_0<2<q_0<\infty$, $r\geq q_0/2$, $v_0\in A_{\frac{2}{p_0}}(w)\cap RH_{r'}(w)$, $\alpha\geq 1$, $0<u<1/4$,  and $F$ a measurable function defined in $\R^{n+1}_+$,  let $\widehat{r}>r_w$ and let $\{\mathcal{T}_{\tau}\}_{\tau>0}$ be a family of sublinear operators such that $\mathcal{T}_{\tau}\in \mathcal{O}(L^{p_0}(w)-L^{q_0}(w))$. Then,
\begin{multline}\label{poissonmaximalop}
\left(\int_{\R^n}\int_0^{\infty}\int_{B(x,\alpha t)}\big|\mathcal{T}_{\frac{t^2}{4u}}F(y,t)\big|^2\frac{dw(y)\,dt}{tw(B(y, \alpha t))}v_0(x)dw(x)\right)^{\frac{1}{2}}
\\
\lesssim\left(\int_0^{\infty}\int_{\R^n}\left(
\int_{B(x,\alpha t)}\big|\mathcal{T}_{\frac{t^2}{4u}}F(y,t)\big|^{q_0}\frac{dw(y)}{w(B(y,\alpha t))}\right)^{\frac{2}{q_0}}v_0(x)dw(x)\frac{dt}{t}\right)^{\frac{1}{2}}
\\
\lesssim
u^{n\widehat{r}\left(\frac{1}{4r}-\frac{1}{2q_0}\right)}
\min\left\{\|\widetilde{\Scal}F\|_{L^2(v_0dw)},\|\widetilde{\mathcal{V}}F\|_{L^2(v_0dw)}\right\}.
\end{multline}
\end{lemma}
\begin{proof}
We fix $w$, $v_0$, $q_0$, $r$, $\alpha$, $\widehat{r}$, and $u$ as in the statement of the lemma and denote:
$$
I:=\left(\int_{\R^n}\int_0^{\infty}\int_{B(x,\alpha t)}\big|\mathcal{T}_{\frac{t^2}{4u}}F(y,t)\big|^2\frac{dw(y)\,dt}{tw(B(y, \alpha t))}v_0(x)dw(x)\right)^{\frac{1}{2}}.
$$
Then, by \eqref{doublingcondition},  Jenssen's inequality ($q_0>2$), and Fubini's theorem 
\begin{align}\label{poissonmaximalop1}
I&\lesssim\left(\int_0^{\infty}\int_{\R^n}\left(
\int_{B(x,\alpha t)}\big|\mathcal{T}_{\frac{t^2}{4u}}F(y,t)\big|^{q_0}\frac{dw(y)}{w(B(y,\alpha t))}\right)^{\frac{2}{q_0}}v_0(x)dw(x)\frac{dt}{t}\right)^{\frac{1}{2}}=:\widetilde{I},
\end{align}
which proves the first inequality in \eqref{poissonmaximalop}.

Thus, by Proposition \ref{prop:Q} with $\beta=2\sqrt{u}<1$ and $q=q_0/2$, and \eqref{doublingcondition}, we have 
\begin{align}\label{poissonmaximalop2}
\widetilde{I}&=
\left(\int_0^{\infty}\int_{\R^n}\left(
\int_{B(x,2\sqrt{u}\alpha t/2\sqrt{u})}\big|\mathcal{T}_{\frac{t^2}{4u}}F(y,t)\big|^{q_0}\frac{dw(y)}{w(B(y,2\sqrt{u}\alpha t/2\sqrt{u}))}\right)^{\frac{2}{q_0}}v_0(x)dw(x)\frac{dt}{t}\right)^{\frac{1}{2}}
\\\nonumber&\lesssim
u^{n\widehat{r}\left(\frac{1}{4r}-\frac{1}{2q_0}\right)}
\left(\int_0^{\infty}\int_{\R^n}\left(
\dashint_{B(x,\alpha t/2\sqrt{u})}\big|\mathcal{T}_{\frac{t^2}{4u}}F(y,t)\big|^{q_0}dw(y)\right)^{\frac{2}{q_0}}v_0(x)dw(x)\frac{dt}{t}\right)^{\frac{1}{2}}
\\\nonumber&
\leq C_{\alpha}
u^{n\widehat{r}\left(\frac{1}{4r}-\frac{1}{2q_0}\right)}\sum_{i\geq 1}e^{-c4^i}
\left(\int_0^{\infty}\int_{\R^n}\left(
\dashint_{B(x,2^{i+1}\alpha t/2\sqrt{u})}|F(y,t)|^{p_0}dw(y)\right)^{\frac{2}{p_0}}v_0(x)dw(x)\frac{dt}{t}\right)^{\frac{1}{2}},
\end{align}
where in the last inequality we have also used that  $\mathcal{T}_{\tau}\in \mathcal{O}(L^{p_0}(w)-L^{q_0}(w))$.
Now note that,  on the one hand, we have
\begin{align*}
\left(
\dashint_{B(x,2^{i+1}\alpha t/2\sqrt{u})}|F(y,t)|^{p_0}dw(y)\right)^{\frac{1}{p_0}}
\lesssim
\mathcal{M}_{p_0}^w
\left(F(\cdot,t)\right)(x), \quad (x,t)\in \R^{n+1}_+.
\end{align*}
On the other hand, by Fubini's theorem and since $\frac{2^{i+1}\alpha}{2\sqrt{u}}>1$  we also have
\begin{align*}
\left(
\dashint_{B(x,2^{i+1}\alpha t/2\sqrt{u})}|F(y,t)|^{p_0}dw(y)\right)^{\frac{1}{p_0}}
&=
\left(
\dashint_{B(x,2^{i+1}\alpha t/2\sqrt{u})}\dashint_{B(y, t)}dw(\xi)|F(y,t)|^{p_0}dw(y)\right)^{\frac{1}{p_0}}
\\& \lesssim
\left(
\dashint_{B(x,2^{i+1}\alpha t/\sqrt{u})}\dashint_{B(\xi, t)}|F(y,t)|^{p_0}dw(y)dw(\xi)\right)^{\frac{1}{p_0}}
\\&\lesssim
\left(
\dashint_{B(x,2^{i+1}\alpha t/\sqrt{u})}\left(\dashint_{B(\xi, t)}|F(y,t)|^{2}dw(y)\right)^{\frac{p_0}{2}}dw(\xi)\right)^{\frac{1}{p_0}}
\\&\lesssim
\mathcal{M}_{p_0}^w
\left(
\left(\dashint_{B(\cdot, t)}|F(y,t)|^{2}dw(y)
\right)^{\frac{1}{2}}\right)(x),
\end{align*}
where in the first inequality we have used that for $y\in B(\xi,t)$ we have that $B(\xi,t)\subset B(y,2t)$ and \eqref{doublingcondition}, and in the second one Jensen's inequality, since $2>p_0$.

Consequently, letting
$
\mathfrak{O}F(x,t)$ be equal to $\mathcal{M}_{p_0}^w
\!\left(\!
\left(\dashint_{B(\cdot, t)}|F(y,t)|^{2}dw(y)
\right)^{\!\!\frac{1}{2}}\!\right)\!\!(x)$ or $ 
\mathcal{M}_{p_0}^w
\left(F(\cdot,t)\right)(x)
$,
 \eqref{poissonmaximalop2}, the boundedness of $\mathcal{M}_{p_0}^w$ on $L^2(v_0dw)$ (recall that $v_0\in A_{\frac{2}{p_0}}(w)$),  Fubini's theorem, and \eqref{doublingcondition} imply
\begin{multline*}
\widetilde{I}
\lesssim
u^{n\widehat{r}\left(\frac{1}{4r}-\frac{1}{2q_0}\right)}
\left(\int_0^{\infty}\int_{\R^n}\left|\mathfrak{O}F(x,t)\right|^2v_0(x)dw(x)\frac{dt}{t}\right)^{\frac{1}{2}}
\\
\lesssim
u^{n\widehat{r}\left(\frac{1}{4r}-\frac{1}{2q_0}\right)}\min\left\{\|\widetilde{\Scal}F\|_{L^2(v_0dw)},\|\widetilde{\mathcal{V}}F\|_{L^2(v_0dw)}\right\},
\end{multline*}
which, in view of \eqref{poissonmaximalop1}, completes the proof of \eqref{poissonmaximalop}.
\end{proof}
\begin{remark}\label{remark:possonmaximal}
Given $w\in A_2$,  $\max\{r_w,(q_-(L_w))_{w,*}\}<p_0<2<q_0<q_{+}(L_w)$, and  $v_0$, $\alpha\geq 1$, $u$, and $F$ as  in the statement of Lemma \ref{lemma:poissonmaximal}, we have that
\begin{multline}\label{poissonmaximalopgradiente}
\left(\int_{\R^n}\int_0^{\infty}\int_{B(x,\alpha t)}\big|\nabla e^{-\frac{t^2}{4u}}F(y,t)\big|^2\frac{dw(y)\,dt}{tw(B(y, \alpha t))}v_0(x)dw(x)\right)^{\frac{1}{2}}
\\
\lesssim\left(\int_0^{\infty}\int_{\R^n}\left(
\int_{B(x,\alpha t)}\big|\nabla e^{-\frac{t^2}{4u}}F(y,t)\big|^{q_0}\frac{dw(y)}{w(B(y,\alpha t))}\right)^{\frac{2}{q_0}}v_0(x)dw(x)\frac{dt}{t}\right)^{\frac{1}{2}}
\\
\lesssim
u^{n\widehat{r}\left(\frac{1}{4r}-\frac{1}{2q_0}\right)}
\min\left\{\|\widetilde{\Scal}\nabla F\|_{L^2(v_0dw)},\|\widetilde{\mathcal{V}}\nabla F\|_{L^2(v_0dw)}\right\}.
\end{multline}
The first inequality in \eqref{poissonmaximalopgradiente} follows as \eqref{poissonmaximalop1}. 

As for the second inequality,  note that, for every $t>0$, by Lemma \ref{lem:Gsum} with $q=q_0$, $p=p_0$, and $S_t$ equal to the identity, we have that 
$$
\left(
\dashint_{B(x,\alpha t/2\sqrt{u})}\big|\nabla e^{-\frac{t^2}{4u}}F(y,t)\big|^{q_0}dw(y)\right)^{\frac{1}{q_0}}\lesssim
\sum_{i\geq 1}e^{-c4^i}
\left(\dashint_{B(x,2^{i+1}\alpha t/2\sqrt{u})}|\nabla F(y,t)|^{p_0}dw(y)\right)^{\frac{1}{p_0}}.
$$
Therefore,
\begin{multline*}
\left(\int_0^{\infty}\int_{\R^n}\left(
\dashint_{B(x,\alpha t/2\sqrt{u})}\big|\nabla e^{-\frac{t^2}{4u}}F(y,t)\big|^{q_0}dw(y)\right)^{\frac{2}{q_0}}v_0(x)dw(x)\frac{dt}{t}\right)^{\frac{1}{2}}
\\
\lesssim 
\sum_{i\geq 1}e^{-c4^i}
\left(\int_0^{\infty}\int_{\R^n}\left(
\dashint_{B(x,2^{i+1}\alpha t/2\sqrt{u})}|\nabla F(y,t)|^{p_0}dw(y)\right)^{\frac{2}{p_0}}v_0(x)dw(x)\frac{dt}{t}\right)^{\frac{1}{2}}.
\end{multline*}
Proposition \ref{prop:Q} and the above inequality imply
\begin{multline*}
\left(\int_0^{\infty}\int_{\R^n}\left(
\int_{B(x,\alpha t)}\big|\nabla e^{-\frac{t^2}{4u}}F(y,t)\big|^{q_0}\frac{dw(y)}{w(B(y,\alpha t))}\right)^{\frac{2}{q_0}}v_0(x)dw(x)\frac{dt}{t}\right)^{\frac{1}{2}}
\\\lesssim
u^{n\widehat{r}\left(\frac{1}{4r}-\frac{1}{2q_0}\right)}\sum_{i\geq 1}e^{-c4^i}
\left(\int_0^{\infty}\int_{\R^n}\left(
\dashint_{B(x,2^{i+1}\alpha t/2\sqrt{u})}|\nabla F(y,t)|^{p_0}dw(y)\right)^{\frac{2}{p_0}}v_0(x)dw(x)\frac{dt}{t}\right)^{\frac{1}{2}}.
\end{multline*}
This substitutes \eqref{poissonmaximalop2} in the proof of Lemma \ref{lemma:poissonmaximal}. Hereafter, the proof follows as  the proof of that lemma, but writing $\nabla F$ in place of $F$.
\end{remark}
Note now that proceeding as in the proof of \cite[Theorem 3.5, part ($b$)]{ChMPA16}, we obtain the following comparison result between the conical square functions defined in \eqref{conicalH-1} and \eqref{conicalP-1}, even for odd values of $m$.
\begin{proposition}\label{prop:coniclalpoisson-heat}
Given $w\in A_{2}$, $v\in A_{\infty}(w)$, $m\in \N$, and $f\in L^2(w)$, there holds
$$\|\Scal_{m,\pp}^wf\|_{L^p(vdw)}\lesssim
\|\Scal_{m,\hh}^wf\|_{L^p(vdw)},\quad \textrm{for all}\quad p\in \mathcal{W}_v^w(0,p_+(L_w)^{m+1,*}_w).
$$
\end{proposition}
Besides, with the aim of  obtaining boundedness for the conical square functions defined via the gradient of the heat or the Poisson semigroup,  we  show the following comparison result.
\begin{proposition}\label{prop:gradient-singradientheat}
Given $w\in A_{2}$, $v\in A_{\infty}(w)$, $m\in \N$, $K\in \N_0$, and $f\in L^2(w)$, there hold
\begin{list}{$(\theenumi)$}{\usecounter{enumi}\leftmargin=1cm \labelwidth=1cm\itemsep=0.2cm\topsep=.2cm \renewcommand{\theenumi}{\alph{enumi}}}

\item $\|\Grm_{m,\hh}^wf\|_{L^p(vdw)}\lesssim\|\Scal_{m,\hh}^wf\|_{L^p(vdw)}$,\, for all\, $0<p<\infty$;

\item $\|\Grm_{K,\pp}^wf\|_{L^p(vdw)}\lesssim \|\Scal_{K+2,\hh}^wf\|_{L^p(vdw)}+
\|\Grm_{K,\hh}^wf\|_{L^p(vdw)}$,  $p\in \mathcal{W}_v^w(0,p_+(L_w)^{K+1,*}_w)$.

\end{list}
\end{proposition}
\begin{proof}
The proof of part $(a)$ follows as the proof of \cite[Theorem 3.3]{ChMPA16}.  Indeed, use \eqref{doublingcondition} and apply  the fact that  
 $\sqrt{\tau}\nabla e^{-\tau L_w}\in \mathcal{O}(L^2(w)-L^2(w))$; then, change the variable $t$ into $\sqrt{2}t$ , and apply again \eqref{doublingcondition} and  Proposition \ref{prop:alpha} to obtain
\begin{multline*}
\|\Grm_{m,\hh}^wf\|_{L^p(vdw)}
\lesssim
\left(\int_{\R^n}\left(\int_0^{\infty}\dashint_{B(x,t)}|t\nabla e^{-\frac{t^2}{2}}(t\sqrt{L_w})^me^{-\frac{t^2}{2}L_w}f(y)|^2\frac{dw(y)dt}{t}\right)^{\frac{p}{2}}v(x)dw(x)\right)^{\frac{1}{p}}
\\
\lesssim	\sum_{j\geq 1}e^{-c4^j}
\left(\int_{\R^n}\left(\int_0^{\infty}\int_{B(x,2^{j+2}t)}|(t\sqrt{L_w})^me^{-t^2L_w}f(y)|^2\frac{dw(y)dt}{tw(B(y,t))}\right)^{\frac{p}{2}}v(x)dw(x)\right)^{\frac{1}{p}}
\\
\lesssim	\sum_{j\geq 1}e^{-c4^j}2^{j\theta_{v,w}}
\|\Scal_{m,\hh}^wf\|_{L^p(vdw)}
\lesssim	
\|\Scal_{m,\hh}^wf\|_{L^p(vdw)}.
\end{multline*}

\medskip

In order to prove part $(b)$, first of all note that
\begin{align*}
\Grm_{K,\pp}^wf(x)
&\leq
\left(\int_0^{\infty}\int_{B(x,t)}|t\nabla(t\sqrt{L_w})^K (e^{-t\sqrt{L_w}}-e^{-t^2L_w})f(y)|^2\frac{dw(y)dt}{tw(B(y,t))}\right)^{\frac{1}{2}}
+\Grm_{K,\hh}^wf(y)
\\&
=:{Q}_Kf(y)+\Grm_{K,\hh}^wf(y)
.
\end{align*}
So we just need to prove that
\begin{align}\label{QkSk+2}
\|{Q}_Kf\|_{L^p(vdw)}\lesssim 
\|\mathcal{S}_{K+2,\hh}^wf\|_{L^p(vdw)},\quad \forall p\in \mathcal{W}_v^w(0,p_+(L_w)^{K+1,*}_w).
\end{align}
We show this by  extrapolation.
In particular, by
Theorem \ref{theor:extrapol}, part($b$) (or Theorem \ref{theor:extrapol}, part($c$) if $p_+(L_w)_w^{K+1,*}=\infty$),  \eqref{QkSk+2} follows from the  inequality
\begin{align}\label{extrapolpoissonheat}
\|{Q}_Kf\|_{L^2(v_0dw)}\lesssim 
\|\mathcal{S}_{K+2,\hh}^wf\|_{L^2(v_0dw)},\quad \forall
v_0\in  RH_{\left(\frac{p_+(L_w)^{K+1,*}_w}{2}\right)'}(w).
\end{align}
To this end, fix $w\in A_2$, $f\in L^2(w)$, and $v_0\in RH_{\left(\frac{p_+(L_w)^{K+1,*}_w}{2}\right)'}(w)$, and note that by Remark \ref{remark:choicereverseholder}, with $q_1=2$, $\widetilde{q}=p_+(L_w)$ and $\widetilde{k}=K+1$, we can find $q_0$, $\widehat{r}$, and $r$ such that   $r_w<\widehat{r}<2$, $2<q_0<p_+(L_w)$, $q_0/2\leq r<\infty$, $v_0\in RH_{r'}(w)$, and
\begin{align}\label{positivepoissonpoissson}
K+1+\frac{n\,\widehat{r}}{2r}-\frac{n\,\widehat{r}}{q_0}
> 0.
\end{align}
Keeping this in mind,  by the subordination formula 
 \eqref{subordinationformula} and Minkowski's integral inequality, we have
\begin{align*}
Q_{K}f(x)
\lesssim
\int_{0}^{\frac{1}{4}}u^{\frac{1}{2}}e^{-u}
I(u)\frac{du}{u}
 +
\int_{\frac{1}{4}}^{\infty}u^{\frac{1}{2}}e^{-u}
I(u)\frac{du}{u}=:I+II,
\end{align*}
where
\begin{align*}
I(u):=\left(\int_0^{\infty}\int_{B(x,t)}|t\nabla(t\sqrt{L_w})^K  (e^{-\frac{t^2}{4u}L_w}-e^{-t^2L_w})f(y)|^2\frac{dw(y)dt}{tw(B(y,t))}\right)^{\frac{1}{2}}.
\end{align*}

We first estimate $II$,  for $1/4<u<\infty$, Cauchy-Schwartz's inequality implies
\begin{multline*}
\big|\big(t\nabla(t\sqrt{L_w})^K e^{-\frac{t^2}{4u}L_w}-t\nabla(t\sqrt{L_w})^Ke^{-t^2L_w}\big)f(y)\big|
\leq 
\int_{\frac{t}{2\sqrt{u}}}^{t} \big|\partial_s t\nabla(t\sqrt{L_w})^K e^{-s^2 L_w} f(z)\big| ds
\\
\lesssim \int_{\frac{t}{2\sqrt{u}}}^t
\big|t\nabla(t\sqrt{L_w})^K (s\sqrt{L_w})^2e^{-s^2L_w}f(y)\big|\frac{ds}{s}
\\
\lesssim u^{\frac{1}{4}} \left(\int_{\frac{t}{2\sqrt{u}}}^t
\big|t\nabla(t\sqrt{L_w})^K (s\sqrt{L_w})^2e^{-s^2L_w}f(y)\big|^2\frac{ds}{s}\right)^{\frac{1}{2}}.
\end{multline*}
Then, for $u>1/4$, we estimate $I(u)$ using \eqref{doublingcondition},   Fubini's theorem, the fact that $s<t<2\sqrt{u}s$, and that $\sqrt{\tau} \nabla e^{-\tau L_w}\in \mathcal{F}(L^2(w)-L^2(w))$. Besides, we change the variable $s$ into $\sqrt{2}s$ to obtain
\begin{align*}
I(u)&
\lesssim
u^{\theta_K}
\left(\int_0^{\infty}\int_{\frac{t}{2\sqrt{u}}}^t\int_{B(x,t)}|s\nabla(s\sqrt{L_w})^{K+2}e^{-s^2L_w}f(y)|^2\frac{dw(y)\,ds}{sw(B(y,t))}\frac{dt}{t}\right)^{\frac{1}{2}}
\\
&
\lesssim
u^{\theta_K}
\left(\int_0^{\infty}\int_{\frac{t}{2\sqrt{u}}}^t\int_{B(x,\sqrt{u}s)}\big|\frac{s}{{\scriptstyle\sqrt{2}}}\nabla e^{-\frac{s^2}{2}L_w}(s\sqrt{L_w})^{K+2}e^{-\frac{s^2}{2}L_w}f(y)\big|^2\frac{dw(y)\,ds}{sw(B(y,s))}\frac{dt}{t}\right)^{\frac{1}{2}}
\\
&\lesssim
\sum_{j\geq 1}
e^{-c4^j}
u^{\theta_K}
\left(\int_0^{\infty}\int_s^{2\sqrt{u}s}\frac{dt}{t}\int_{B(x,2^{j+1}\sqrt{u}s)}|(s\sqrt{L_w})^{K+2}e^{-\frac{s^2}{2}L_w}f(y)|^2\frac{dw(y)\,ds}{sw(B(y,s))}\right)^{\frac{1}{2}}
\\
&\lesssim
\sum_{j\geq 1}
e^{-c4^j}
u^{\theta_K}
\left(\int_0^{\infty}\int_{B(x,2^{j+2}\sqrt{u}s)}|(s\sqrt{L_w})^{K+2}e^{-s^2L_w}f(y)|^2\frac{dw(y)\,ds}{sw(B(y,s))}\right)^{\frac{1}{2}}.
\end{align*}
Hence, by Minkowski's integral inequality and change of angle (Proposition \ref{prop:alpha})
\begin{align*}
\|II\|_{L^2(v_0dw)}
\lesssim
\sum_{j\geq 1}
e^{-c4^j}2^{j\theta_{v_0,w}}
\int_{\frac{1}{4}}^{\infty}u^{\theta_{v_0,w,K}}e^{-u}
\frac{du}{u}\|\Scal_{K+2,\hh}^wf\|_{L^2(v_0dw)}
\lesssim
\|\Scal_{K+2,\hh}^wf\|_{L^2(v_0dw)}.
\end{align*}
As for the estimate of $I$, note that for $0<u<1/4$, again by Cauchy- Schwartz's inequality
\begin{multline*}
\abs{\Big(t\nabla(t\sqrt{L_w})^K e^{-\frac{t^2}{4u}L_w}-t\nabla(t\sqrt{L_w})^K e^{-t^2 L_w}\Big)f(y)} 
 \lesssim 
\int_{t}^{\frac{t}{2\sqrt{u}}} \abs{t\nabla(t\sqrt{L_w})^K (s\sqrt{L_w})^2 e^{-s^2 L_w} f(y)} \frac{ds}{s}
\\ \lesssim
\br{\int_{t}^{\infty} \abs{t\nabla(t\sqrt{L_w})^K (s\sqrt{L_w})^2 e^{-s^2 L_w} f(y)}^2 \frac{ds}{s}}^{\frac{1}{2}} \left(\log u^{-\frac{1}{2}}\right)^{\frac{1}{2}}.
\end{multline*}
The above estimate, Fubini's theorem, \eqref{doublingcondition}, and  Jensen's inequality ($q_0>2$) imply
\begin{multline*}
I
\lesssim
\int_{0}^{\frac{1}{4}}u^{\frac{1}{2}}\left(\log u^{-\frac{1}{2}}\right)^{\frac{1}{2}}\frac{du}{u}
\left(\int_0^{\infty}\dashint_{B(x,t)}\int_{t}^{\infty} \abs{t\nabla(t\sqrt{L_w})^K (s\sqrt{L_w})^2 e^{-s^2 L_w} f(y)}^2 \frac{ds}{s}\frac{dw(y)dt}{t}\right)^{\frac{1}{2}}
\\
\lesssim
\left(\int_0^{\infty}\int_{0}^{s} \left(\frac{t}{s}\right)^{2K+2}\left(\dashint_{B(x,t)} \abs{ s\nabla(s\sqrt{L_w})^{K+2}  e^{-s^2 L_w} f(y)}^{q_0}dw(y)\right)^{\frac{2}{q_0}}\frac{dt}{t}\frac{ds}{s}\right)^{\frac{1}{2}}
.
\end{multline*}
Besides, note that by Proposition \ref{prop:Q}, with $\beta=t/s<1$ and $q=q_0/2$,
\begin{multline*}
\int_{\R^n}\left(\int_{B(x,st/s)} \abs{  s\nabla(s\sqrt{L_w})^{K+2}  e^{-s^2 L_w} f(y)}^{q_0}\frac{dw(y)}{w(B(y,st/s))}\right)^{\frac{2}{q_0}}v_0(x)dw(x)
\\
\lesssim
\left(\frac{t}{s}\right)^{n\widehat{r}\left(\frac{1}{r}-\frac{2}{q_0}\right)}\int_{\R^n}\left(\int_{B(x,s)} \abs{   s\nabla(s\sqrt{L_w})^{K+2}   e^{-s^2 L_w} f(y)}^{q_0}\frac{dw(y)}{w(B(y,s))}\right)^{\frac{2}{q_0}}v_0(x)dw(x).
\end{multline*}
Thus, for $C_{1}:=2K+2+n\widehat{r}\left(\frac{1}{r}-\frac{2}{q_0}\right)$, we apply first  Fubini's theorem,  and  the above inequality. Then, by \eqref{doublingcondition}, changing the variable $s$ into $2s$, \eqref{positivepoissonpoissson},  the fact that $\sqrt{\tau}\nabla e^{-\tau L_w}\in \mathcal{O}(L^2(w)-L^{q_0}(w))$, and by Proposition \ref{prop:alpha}, we get
\begin{align*}
&\|I\|_{L^2(v_0dw)}
\\&\lesssim
\left(\int_{\R^n}\int_0^{\infty}\int_{0}^{s} \left(\frac{t}{s}\right)^{C_1}\frac{dt}{t}\left(\int_{B(x,s)} \abs{ s\nabla(s\sqrt{L_w})^{K+2}  e^{-s^2 L_w} f(y)}^{q_0}\frac{dw(y)}{w(B(y,s))}\right)^{\frac{2}{q_0}}\frac{ds}{s}v_0(x)dw(x)\right)^{\frac{1}{2}} 
\\&\lesssim
\left(\int_{\R^n}\int_0^{\infty}\left(\dashint_{B(x,2s)} \abs{s\nabla  e^{-s^2 L_w} (s\sqrt{L_w})^{K+2}  e^{-s^2 L_w} f(y)}^{q_0}dw(y)\right)^{\frac{2}{q_0}}\frac{ds}{s}v_0(x)dw(x)\right)^{\frac{1}{2}} 
\\&\lesssim
\sum_{i\geq 1}e^{-c4^i}
\left(\int_{\R^n}\int_0^{\infty}\int_{B(x,2^{i+2}s)} \abs{ (s\sqrt{L_w})^{K+2}  e^{-s^2 L_w} f(y)}^{2}\frac{dw(y)ds}{sw(B(y,s))}v_0(x)dw(x)\right)^{\frac{1}{2}} 
\\&
\lesssim
\sum_{i\geq 1}e^{-c4^i}2^{i\theta_{v_0,w}}
\|\Scal_{K+2,\hh}^wf\|_{L^2(v_0dw)}
\\&
\lesssim
\|\Scal_{K+2,\hh}^wf\|_{L^2(v_0dw)},
\end{align*}
which finishes the proof.
\end{proof}

\medskip

%
%
%

In the next theorem we obtain a norm comparison result for $\Ncal_{\pp}^w$. This  will be used in the proof of Theorem \ref{thm:boundednessnon-tangential}.
\begin{theorem}\label{thm:improvementnontangentialpoisson}
Given $w\in A_2$, $v\in A_{\infty}(w)$, and $f\in L^2(w)$, there holds
\begin{align*}
\|\mathcal{N}_{\pp}^wf\|_{L^p(vdw)}\lesssim \|\Ncal_{\hh}^wf\|_{L^p(vdw)}+\|\Scal_{2,\hh}^wf\|_{L^p(vdw)},\quad
\forall\,p\in \mathcal{W}_v^w(p_-(L_w),p_+(L_w)^{*}_w).
\end{align*}
\end{theorem}
\begin{proof}
First of all, fix $w$, $v$, $f$, and $p$ as in the statement of the theorem. Then, recalling the definition of $\Ncal_{\hh}^w$ in \eqref{nontangential}, note that
\begin{align}\label{sumpoissonnontangential1}
\Ncal_{\pp}^{w}f(x) \leq \Ncal_{\hh}^{w}f(x)
+
\sup_{t>0} \br{\fint_{B(x, t)} \abs{\br{e^{-t\sqrt {L_w}}-e^{-t^2L_w}}f(z)}^2 dw(z)}^{\frac{1}{2}}=:\Ncal_{\hh}^wf(x)+ \sup_{t>0}I.
\end{align}
Then, by the subordination formula in \eqref{subordinationformula} and Minkowski's integral inequality.
\begin{align}\label{nontangential2}
I&
\leq
\int_{0}^{{\frac{1}{4}}} u^{{\frac{1}{2}}} \br{\fint_{B(x,t)} \abs{\br{e^{-\frac{t^2}{4u} L_w}-e^{-t^2L_w}}f(z)}^2 dw(z)}^{\frac{1}{2}} \frac{du}{u}
\\\nonumber&\quad+
\int_{{\frac{1}{4}}}^{\infty} e^{-u} u^{{\frac{1}{2}}} \br{\fint_{B(x,t)} \abs{\br{e^{-\frac{t^2}{4u} L_w}-e^{-t^2L_w}}f(z)}^2 dw(z)}^{\frac{1}{2}} \frac{du}{u}
:=I_1+I_2.
\end{align}
Similarly as in the proof of \cite[Proposition 7.1, part $(b)$]{MaPAII17} (see also \cite{ChMPA18}), we have that
$$
I_2\lesssim \int_{\frac{1}{4}}^{\infty}e^{-cu}\Scal_{2,\hh}^{2\sqrt{u},w}f(x)\frac{du}{u}.
$$
Hence, by Minkowski's integral inequality and Proposition \ref{prop:alpha},
\begin{align}\label{nontangentialI2}
\|I_2\|_{L^p(vdw)}\lesssim \int_{{\frac{1}{4}}}^{\infty}  e^{-cu}\|\Scal_{2,\hh}^{2\sqrt u,w}f\|_{L^p(vdw)}\frac{du}{u}
\lesssim \|\Scal_{2,\hh}^{w}f\|_{L^p(vdw)}.
\end{align}
In order to estimate $I_1$, take $p_0$ such that $p_-(L_w)<p_0<\min\{2,p\}$ close enough to $p_-(L_w)$  so that $v\in A_{\frac{p}{p_0}}(w)$, and 
 note that the fact that $e^{-\tau L_w}\in \mathcal{O}(L^{p_0}(w)-L^{2}(w))$
implies
\begin{align*}
I_1
&
\lesssim  
\sum_{j\geq 1} e^{-c4^j} \int_{0}^{{\frac{1}{4}}} u^{\frac{1}{2}}  \br{\fint_{B(x,2^{j+1} t)} 
\abs{\br{e^{-(\frac{1}{4u}-\frac{1}{2}) t^2 L_w}-e^{-\frac{t^2}{2} L_w}}f(z)}^{p_0}  dw(z)}^{\frac{1}{p_0}} \frac{du}{u}.
\end{align*}
Besides, for $0<u<\frac{1}{4}$, it holds from  H\"older's inequality that
\begin{multline*}
\abs{\br{e^{-(\frac{1}{4u}-\frac{1}{2}) t^2 L_w}-e^{-\frac{t^2}{2} L_w}}f(z)} 
\leq 
\int_{\frac{t}{\sqrt 2}}^{t\sqrt{\frac{1}{4u}-\frac{1}{2}}} \abs{\partial_s e^{-s^2 L_w} f(z)} ds
\\ 
\lesssim
\int_{\frac{t}{\sqrt 2}}^{t\sqrt{\frac{1}{4u}-\frac{1}{2}}} \abs{s^2 L_w e^{-s^2 L_w} f(z)} \frac{ds}{s}
\lesssim
\br{\int_{\frac{t}{\sqrt 2}}^{\frac{t}{2\sqrt{u}}}\abs{s^2 L_w e^{-s^2 L_w} f(z)}^{p_0} \frac{ds}{s}}^{\frac{1}{p_0}}  (\log u^{-\frac{1}{2}})^{\frac{1}{p_0'}}.
\end{multline*}
Therefore,
\begin{align*}
I_1
&
\lesssim  
\sum_{j\geq 1} e^{-c4^j} \int_{0}^{{\frac{1}{4}}} u^{\frac{1}{2}} (\log u^{-\frac{1}{2}})^{\frac{1}{p_0'}} \br{\fint_{B(x,2^{j+1} t)} 
\int_{\frac{t}{\sqrt 2}}^{\frac{t}{2\sqrt{u}}} \abs{s^2 L_w e^{-s^2 L_w} f(z)}^{p_0} \frac{ds}{s}  dw(z)}^{\frac{1}{p_0}} \frac{du}{u}.
\end{align*}
Note now  that, for $\frac{t}{\sqrt{2}}<s<\frac{t}{2\sqrt{u}}$, 
\begin{multline*}\big\{(z,y)\in \R^{n}\times\R^n\!:z\in B(x,2^{j+1}t),\,y\in B(z,2\sqrt{u}\,s)\big\}\\\subset\big\{(z,y)\in \R^{n}\times\R^n\!:y\in B(x,2^{j+2}t),\,z\in B(y,2\sqrt{u}\,s)\big\},
\end{multline*}
and for $z\in B(y,2\sqrt{u} s)$ it holds that $ B(y,2\sqrt{u} s)\subset  B(z,4\sqrt{u} s)$. Then,  Fubini's theorem and \eqref{doublingcondition} imply
\begin{multline*}
\dashint_{B(x,2^{j+1} t)} 
\int_{\frac{t}{\sqrt 2}}^{\frac{t}{2\sqrt{u}}} \abs{s^2 L_w e^{-s^2 L_w} f(z)}^{p_0} \frac{ds}{s}  dw(z)
=
\int_{\frac{t}{\sqrt 2}}^{\frac{t}{2\sqrt{u}}}\fint_{B(x,2^{j+1} t)} 
 \abs{s^2 L_w e^{-s^2 L_w} f(z)}^{p_0} dw(z) \frac{ds}{s} 
 \\
=
\int_{\frac{t}{\sqrt 2}}^{\frac{t}{2\sqrt{u}}}\fint_{B(x,2^{j+1} t)} \dashint_{B(z,2\sqrt{u}s)}dw(y)
 \abs{s^2 L_w e^{-s^2 L_w} f(z)}^{p_0} dw(z) \frac{ds}{s} 
 \\
 \lesssim
 \int_{\frac{t}{\sqrt 2}}^{\frac{t}{2\sqrt{u}}}\dashint_{B(x,2^{j+2} t)} \dashint_{B(y,2\sqrt{u}s)}
 \abs{s^2 L_w e^{-s^2 L_w} f(z)}^{p_0}dw(z)dw(y)\frac{ds}{s}
 =:II.
  \end{multline*}
  Next, apply Fubini's theorem, \eqref{doublingcondition}, Jensen's inequality in the integral in $z$ ($2>p_0$), and H\"older's inequality in the integral in $s$ with $2/p_0$ and $(2/p_0)'$,  and again \eqref{doublingcondition}. Then,
 \begin{align*}
II
&\lesssim \dashint_{B(x,2^{j+2} t)}\int_{\frac{t}{\sqrt 2}}^{\frac{t}{2\sqrt{u}}} \dashint_{B(y,2\sqrt{u}s)}
 \abs{s^2 L_w e^{-s^2 L_w} f(z)}^{p_0} dw(z)\frac{ds}{s} dw(y)
 \\&
\leq
 \dashint_{B(x,2^{j+2} t)}\int_{\frac{t}{\sqrt 2}}^{\frac{t}{2\sqrt{u}}}\left( \dashint_{B(y,2\sqrt{u}s)}
 \abs{s^2 L_w e^{-s^2 L_w} f(z)}^{2}dw(z) \right)^{\frac{p_0}{2}}\frac{ds}{s} dw(y)
 \\&
\lesssim
 \dashint_{B(x,2^{j+2} t)}\left(\int_{\frac{t}{\sqrt 2}}^{\frac{t}{2\sqrt{u}}} \dashint_{B(y,2\sqrt{u}s)}
 \abs{s^2 L_w e^{-s^2 L_w} f(z)}^{2}dw(z)\frac{ds}{s}\right)^{\frac{p_0}{2}} dw(y)\left(\log u^{-\frac{1}{2}}\right)^{\frac{2-p_0}{2}}
 \\&
\lesssim
 \dashint_{B(x,2^{j+2} t)}\left(\int_{0}^{\infty} \int_{B(y,2\sqrt{u}s)}
 \abs{s^2 L_w e^{-s^2 L_w} f(z)}^{2} \frac{dw(z)\,ds}{sw(B(z,2\sqrt{u}s))}\right)^{\frac{p_0}{2}} dw(y)\left(\log u^{-\frac{1}{2}}\right)^{\frac{2-p_0}{2}}.
\end{align*}
Consequently,
\begin{align*}
I_1\lesssim \int_0^{\frac{1}{4}}u^{\frac{1}{2}}
\left(\log u^{-\frac{1}{2}}\right)^{\frac{1}{2}}\mathcal{M}_{p_0}^w\left(\mathfrak{S}_{\hh}^{2\sqrt{u},w}f\right)(x)\frac{du}{u},
\end{align*}
where 
$$
\mathfrak{S}_{\hh}^{2\sqrt{u},w}f(y):=
\left(\int_{0}^{\infty} \int_{B(y,2\sqrt{u}s)}
 \abs{s^2 L_w e^{-s^2 L_w} f(z)}^{2} \frac{dw(z)\,ds}{sw(B(z,2\sqrt{u}s))}\right)^{\frac{1}{2}}.
$$
Note now that since $v\in A_{\frac{p}{p_0}}(w)$ the maximal operator $\mathcal{M}_{p_0}^w$ is bounded on $L^p(vdw)$. Hence, by Minkowski's integral inequality
\begin{align}\label{LpI}
\|I_1\|_{L^p(vdw)}\lesssim
\int_0^{\frac{1}{4}}u^{\frac{1}{2}}
\left(\log u^{-\frac{1}{2}}\right)^{\frac{1}{2}}
\left\|\mathfrak{S}_{\hh}^{2\sqrt{u},w}f\right\|_{L^p(vdw)}\frac{du}{u}.
\end{align}

In order to estimate the norm on $L^p(vdw)$ of  $\mathfrak{S}_{\hh}^{2\sqrt{u},w}f$ we shall proceed by extrapolation. 
To this end, note that for $2<q_0<p_+(L_w)$, $r_w<\widehat{r}<2$ and $r\geq q_0/2$, for all $v_0\in RH_{r'}(w)$, by Jensen's inequality, Fubini's theorem,  Proposition \ref{prop:Q} with $\beta=2\sqrt{u}<1$  and $q=q_0/2$, and \eqref{doublingcondition}, we obtain
\begin{align*}
\left(\int_{\R^n}\right.&\left.|\mathfrak{S}_{\hh}^{2\sqrt{u},w}f(y)|^2v_0(y)dw(y)\right)^{\frac{1}{2}}
\\&
\leq
\left(\int_{0}^{\infty} \int_{\R^n}\left(\int_{B(y,2\sqrt{u}s)}
 \abs{s^2 L_w e^{-s^2 L_w} f(z)}^{q_0} \frac{dw(z)}{w(B(z,2\sqrt{u}s))}\right)^{\frac{2}{q_0}}v_0(y)dw(y)\frac{ds}{s}
 \right)^{\frac{1}{2}}
 \\&
 \leq u^{n\widehat{r}\left(\frac{1}{4r}-\frac{1}{2q_0}\right)}\left(
\int_{0}^{\infty}  \int_{\R^n}\left(\dashint_{B(y,s)}
 \abs{e^{-\frac{s^2}{2} L_w}s^2 L_w e^{-\frac{s^2}{2}L_w} f(z)}^{q_0} dw(z)\right)^{\frac{2}{q_0}}v_0(y)dw(y)\frac{ds}{s}\right)^{\frac{1}{2}}=:III.
  \end{align*}
 Now, apply the fact that $e^{-\tau L_w}\in \mathcal{O}(L^2(w)-L^{q_0}(w))$, change the variable $s$ into $\sqrt{2}s$, and apply Proposition \ref{prop:alpha} to get
\begin{align*}
 III&
 \lesssim \sum_{j\geq 1}e^{-c4^j} u^{n\widehat{r}\left(\frac{1}{4r}-\frac{1}{2q_0}\right)}
  \left(\int_{\R^n}\int_{0}^{\infty}\dashint_{B(y,2^{j+2}s)}
 \abs{s^2 L_w e^{-s^2 L_w} f(z)}^{2} dw(z)\frac{ds}{s}v_0(y)dw(y)\right)^{\frac{1}{2}}
  \\&
 \lesssim \sum_{j\geq 1}e^{-c4^j}2^{j\theta_{v_0,w}} u^{n\widehat{r}\left(\frac{1}{4r}-\frac{1}{2q_0}\right)}
  \left(\int_{\R^n}|\Scal_{2,\hh}^wf(y)|^2v_0(y)dw(y)\right)^{\frac{1}{2}}
 \\&
 \lesssim 
 \left(\int_{\R^n}|u^{c(r,q_0)}\Scal_{2,\hh}^wf(y)|^2v_0(y)dw(y)\right)^{\frac{1}{2}},
\end{align*}
where $c(r,q_0):=n\widehat{r}\left(\frac{1}{4r}-\frac{1}{2q_0}\right)$.

Hence,   by Theorem \ref{theor:extrapol}, part $(b)$, we have
\begin{align}\label{hipotesisextrapolation}
\int_{\R^n}&|\mathfrak{S}_{\hh}^{2\sqrt{u},w}f(y)|^{\frac{2r}{\widetilde{r}}}\widetilde{v}(y)dw(y)
 \lesssim 
 \int_{\R^n}|u^{c(r,q_0)}\Scal_{2,\hh}^wf(y)|^{\frac{2r}{\widetilde{r}}}\widetilde{v}(y)dw(y),
\end{align}
for all  $r_w<\widehat{r}<2$, $2<q_0<p_+(L_w)$, $r\geq q_0/2$, $1<\widetilde{r}<\infty$,
 and $\widetilde{v}\in RH_{\widetilde{r}'}(w)$.
Moreover,  for $v\in RH_{\left(\frac{p_+(L_w)_w^*}{p}\right)'}(w)$,  Remark \eqref{remark:choicereverseholder} with $q_1=p$, $\widetilde{k}=1$, and $\tilde{q}=p_+(L_w)$, implies that there exist 
$r_w<\widehat{r}<2$, $2<q_0<p_+(L_w)$, $r\geq q_0/2$, and $1<\widetilde{r}=2r/p<\infty$ so that
 ${v}\in RH_{\widetilde{r}'}(w)$, and 
 \begin{align}\label{positive}
1+\frac{n\widehat{r}}{2r}-\frac{n\widehat{r}}{q_0}>0.
\end{align}
Consequently, by \eqref{hipotesisextrapolation} we conclude that given $v\in RH_{\left(\frac{p_+(L_w)_w^*}{p}\right)'}(w)$, there exist $r_w<\widehat{r}<2$, $2<q_0<p_+(L_w)$, $r\geq q_0/2$ such that  \eqref{positive} is satisfied and
\begin{align*}
\int_{\R^n}&|\mathfrak{S}_{\hh}^{2\sqrt{u},w}f(y)|^{p}{v}(y)dw(y)
 \lesssim 
 \int_{\R^n}|u^{c(r,q_0)}\Scal_{2,\hh}^wf(y)|^{p}{v}(y)dw(y).
\end{align*}
Plugging this into \eqref{LpI} allows us to obtain
\begin{align*}
\|I_1\|_{L^p(vdw)}
\lesssim \int_0^{\frac{1}{4}}u^{\frac{1}{2}+n\widehat{r}\left(\frac{1}{4r}-\frac{1}{2q_0}\right)}
\left(\log u^{-\frac{1}{2}}\right)^{\frac{1}{2}}
\frac{du}{u}\left\|\Scal_{2,\hh}^{w}f\right\|_{L^p(vdw)}. 
\end{align*}
Note now that in view of \eqref{positive}, we can take  $M\in \N$ large enough so that 
$$
\frac{1}{2}+n\widehat{r}\left(\frac{1}{4r}-\frac{1}{2q_0}\right)-\frac{1}{4 M}>0.
$$
Therefore, 
\begin{multline*}
\|I_1\|_{L^p(vdw)}
\lesssim
\int_0^{\frac{1}{4}}u^{\frac{1}{2}+n\widehat{r}\left(\frac{1}{4r}-\frac{1}{2q_0}\right)}
\left(M \log u^{-\frac{1}{2M}}\right)^{\frac{1}{2}}
\frac{du}{u}\left\|\Scal_{2,\hh}^{w}f\right\|_{L^p(vdw)}
\\
\lesssim
\int_0^{\frac{1}{4}}u^{\frac{1}{2}+n\widehat{r}\left(\frac{1}{4r}-\frac{1}{2q_0}\right)-\frac{1}{4M}}
\frac{du}{u}\left\|\Scal_{2,\hh}^{w}f\right\|_{L^p(vdw)}
\lesssim
\left\|\Scal_{2,\hh}^{w}f\right\|_{L^p(vdw)}.
\end{multline*}

\medskip

This \eqref{sumpoissonnontangential1}, \eqref{nontangential2}, and \eqref{nontangentialI2} imply 
\begin{align*}
\norm{\Ncal_{\pp}^{w}f}_{L^p(vdw)} 
&
\lesssim 
\norm{\Ncal_{\hh}^{w} f}_{L^p(vdw)}
+
\norm{\Scal_{2,\hh}^{w}f}_{L^p(vdw)}.
\end{align*} 
\end{proof}
The boundedness of the non-tangential maximal function associated with the heat semigroup, $\Ncal_{\hh}^w$, is proved in \cite{ChMPA18} in the weighted degenerate case. The proof is easy and consist in controlling the norm on $L^p(vdw)$ of $\Ncal_{\hh}^w$ by that of an adequate Hardy-Littlewood maximal operator. This is the argument used in  \cite[Proposition 7.1]{MaPAII17} in the weighted non-degenerate case, see also \cite[Section 6]{HMay09}.
\begin{theorem}{\cite{ChMPA18}}\label{thm:boundednessnontangentialheat}
Given $w\in A_2$ and $v\in A_{\infty}(w)$, the non-tangential maximal function $\Ncal_{\hh}^w$ can be extended to a bounded operator on $L^p(vdw)$, for all $p\in \mathcal{W}_v^w(p_-(L_w),\infty)$.
\end{theorem}

\medskip

To conclude this section we observe some trivial norm comparison results between vertical square functions that will make  the proof of Theorem \ref{thm:boundednessverticalpoisson} easier.
%
%
\begin{remark}\label{rmk:moreconvenient}
Given $w\in A_2$ and $v\in A_{\infty}(w)$,
as a consequence of the boundedness of the Riesz transform $\nabla L_w^{-{\frac{1}{2}}}$ (see \cite{CMR15} and also \cite{Au07,AMIII06}), for every $m\in \N$ and $f\in L^2(w)$, we have that, for all $\mathcal{W}_v^w(q_-(L_w),q_+(L_w))$,
\begin{align}\label{verticalverticalnorm}
\|\mathsf{g}_{m-1,\hh}^wf\|_{L^p(vdw)}\lesssim
\|\mathsf{s}_{m,\hh}^wf\|_{L^p(vdw)}\quad \textrm{and}\quad\|\mathsf{g}_{m-1,\pp}^wf\|_{L^p(vdw)}\lesssim
\|\mathsf{s}_{m,\pp}^wf\|_{L^p(vdw)}.
\end{align}
We prove the above inequalities by extrapolation. In particular, by Theorem \ref{theor:extrapol}, part $(e)$, it is enough to prove that
\begin{align*}
\|\mathsf{g}_{m-1,\hh}^wf\|_{L^2(v_0dw)}\lesssim
\|\mathsf{s}_{m,\hh}^wf\|_{L^2(v_0dw)}\quad \textrm{and}\quad\|\mathsf{g}_{m-1,\pp}^wf\|_{L^2(v_0dw)}\lesssim
\|\mathsf{s}_{m,\pp}^wf\|_{L^2(v_0dw)},
\end{align*}
for all $v_0\in A_{\frac{2}{q_-(L_w)}}(w)\cap RH_{\left(\frac{q_+(L_w)}{2}\right)'}(w)$.
This follows  applying Fubini's theorem and the boundedness of the Riesz transform on $L^2(v_0dw)$. Indeed, letting $F(x,t)$ be equal to $e^{-t^2L_w}f(x)$ or $e^{-t\sqrt{L_w}}f(x)$, we get
\begin{multline*}
\int_{\R^n}\int_0^{\infty}|t\nabla (t\sqrt{L_w})^{m-1}F(x,t)|^2\frac{dt}{t}v_0(x)dw(x)
=
\int_0^{\infty}\int_{\R^n}|\nabla L_w^{-\frac{1}{2}} (t\sqrt{L_w})^{m}F(x,t)|^2v_0(x)dw(x)\frac{dt}{t}
\\
\lesssim
\int_{\R^n}
\int_0^{\infty}|(t\sqrt{L_w})^{m}F(x,t)|^2\frac{dt}{t}v_0(x)dw(x).
\end{multline*}

Besides, as it has been observed in several papers before (see, for instance, \cite{AHM12,HMay09}),  by applying \eqref{subordinationformula}, Minkownki's integral inequality, and  the change of  variable $t$ into $2\sqrt{u}t$, we also have that
\begin{align}\label{verticalpoissonverticalheat}
\mathsf{s}_{m,\pp}^wf(x)
\lesssim \int_0^{\infty}u^{\frac{1+m}{2}}e^{-u}\frac{du}{u}\mathsf{s}_{m,\hh}^wf(x)
\lesssim
\mathsf{s}_{m,\hh}^wf(x),\quad x\in \R^n.
\end{align}
\end{remark}
\begin{remark}\label{remark:moreconvenient2}
Note that \eqref{verticalverticalnorm} and Theorem \ref{thm:conical-verticalnon-gradientdegenerate} parts $(a)$ and $(c)$ imply that, for all $K\in \N_0$, $w\in A_2$, and $p\in \mathcal{W}_v^w(q_-(L_w),q_+(L_w))$,
\begin{align*}
\|\mathsf{g}_{K,\hh}^wf\|_{L^p(vdw)}\lesssim
\|\Scal_{K+1,\hh}^wf\|_{L^p(vdw)}\quad \textrm{and}\quad\|\mathsf{g}_{K,\pp}^wf\|_{L^p(vdw)}\lesssim
\|\Scal_{K+1,\pp}^wf\|_{L^p(vdw)}.
\end{align*}
\end{remark}

\section{Proof of the main results}\label{sec:proofs}
\subsection{Norm comparison for vertical and conical square functions}
\subsubsection{Proof of Theorem \ref{thm:conical-verticalnon-gradientdegenerate}, parts $(a)$ and $(b)$.}
Fix $w\in A_2$, $f\in L^2(w)$, and $m\in \N$. 
In order to prove parts $(a)$ and $(b)$ note that for $T_t=e^{-tL_w}$, $F(y,t)=(t\sqrt{L_w})^{m}f(y)$, and any $p_-(L_w)<p_0<2<q_0<p_+(L_w)$, the conditions $(i)-(iii)$ in Proposition \ref{prop:comparison-general} are satisfied. Then, given $0<p<p_+(L_w)$ and $v\in RH_{\left(\frac{p_+(L_w)}{p}\right)'}(w)$, since we can always find $\max\{2,p\}<q_0<p_+(L_w)$ close enough to $p_+(L_w)$ so that $v\in RH_{\left(\frac{q_0}{p}\right)'}(w)$, in view of \eqref{intervalrsw},  Proposition \ref{prop:comparison-general}, part $(b)$, implies
$$
\|\mathsf{s}_{m,\hh}^wf\|_{L^p(vdw)}\lesssim 
\|\Scal_{m,\hh}^wf\|_{L^p(vdw)}.
$$
On the other hand,  given $p_-(L_w)<p<\infty$ and $v\in A_{\frac{p}{p_-(L_w)}}(w)$, since we can always find $p_-(L_w)<p_0<\min\{2,p\}$ close enough to $p_-(L_w)$ so that $v\in A_{\frac{p}{p_0}}(w)$,  in view of \eqref{intervalrsw},  Proposition \ref{prop:comparison-general},  part $(a)$, implies
$$
\|\Scal_{m,\hh}^wf\|_{L^p(vdw)}\lesssim 
\|\mathsf{s}_{m,\hh}^wf\|_{L^p(vdw)}.
$$\qed
\subsubsection{Proof of Theorem \ref{thm:conical-verticalnon-gradientdegenerate}, part $(c)$}
Fix $w\in A_2$, $f\in L^2(w)$, and $m\in \N$.
We shall proceed by extrapolation. In particular, note that by Theorem \ref{theor:extrapol}, part $(d)$ (or part $(a)$ if $p_+(L_w)=\infty$) it is enough to show that
 \begin{align}\label{extrapolationpoissonverticalconnical}
 \|\mathsf{s}_{m,\pp}^wf\|_{L^2(v_0dw)}\lesssim\|\Scal_{m,\pp}^wf\|_{L^2(v_0dw)}, \quad \forall v_0\in A_{\frac{2}{p_-(L_w)}}(w)\cap RH_{\left(\frac{p_+(L_w)}{2}\right)'}(w).
 \end{align}
In order to prove this inequality, first of all note that since $v_0\in A_{\frac{2}{p_-(L_w)}}(w)\cap  RH_{\left(\frac{p_+(L_w)}{2}\right)'}(w)$ we can take $p_-(L_w)<p_0<2<q_0<p_+(L_w)$ so close to $p_-(L_w)$ and $p_+(L_w)$  so that $v_0\in  A_{\frac{2}{p_0}}(w)\cap RH_{\left(\frac{q_0}{2}\right)'}(w)$.
Keeping this choice of $p_0$ and $q_0$,
change the variable $t$ into $2t$ and apply the subordination formula \eqref{subordinationformula} to get
\begin{multline*}
\|\mathsf{s}_{m,\pp}^wf\|_{L^2(v_0dw)}
\lesssim
\left(\int_{\R^n}\int_0^{\infty}|(t\sqrt{L_w})^me^{-t\sqrt{L_w}}e^{-t\sqrt{L_w}}f(y)|^2\frac{dt}{t}v_0(y)dw(y)\right)^{\frac{1}{2}}
\\
\lesssim
\int_{0}^{\frac{1}{4}}u^{\frac{1}{2}}I(u)\frac{du}{u}
 +
\int_{\frac{1}{4}}^{\infty}u^{\frac{1}{2}}e^{-u}
I(u)\frac{du}{u}=:I+II,
\end{multline*}
where  in the second inequality we have  used  Minkowski's integral inequality and 
$$
I(u):=
\left(\int_{\R^n}\int_0^{\infty}|e^{-\frac{t^2}{4u}L_w}(t\sqrt{L_w})^me^{-t\sqrt{L_w}}f(y)|^2\frac{dt}{t}v_0(y)dw(y)\right)^{\frac{1}{2}}.
$$
Moreover, for each $u>0$, recalling our choice of $q_0$, the inequality \eqref{verticalsinRH}, with $T_{t^2}=e^{-\frac{t^2}{4u}L_w}$ and $F(y,t)=(t\sqrt{L_w})^me^{-t\sqrt{L_w}}f(y)$, and \eqref{doublingcondition} yield
\begin{align}\label{generalsubordinationpoison}
I(u)\lesssim
\left(\int_{\R^n}\int_0^{\infty}\left(\int_{B(x,t)}|e^{-\frac{t^2}{4u}L_w}(t\sqrt{L_w})^me^{-t\sqrt{L_w}}f(y)|^{q_0}\frac{dw(y)\,dt}{tw(B(y,t))}\right)^{\frac{2}{q_0}}
v_0(x)dw(x)\right)^{\frac{1}{2}}.
\end{align}
Hence, by Lemma \ref{lemma:poissonmaximal} with $\mathcal{T}_{\tau}=e^{-\tau L_w}$, $F(y,t)=(t\sqrt{L_w})^me^{-t\sqrt{L_w}}f(y)$, $\alpha=1$, and $r=q_0/2$
\begin{align}\label{Iverticalconicalpoisson}
I\lesssim \int_{0}^{\frac{1}{4}}u^{\frac{1}{2}}I(u)\frac{du}{u}\lesssim
\int_{0}^{\frac{1}{4}}u^{\frac{1}{2}}\frac{du}{u}\|\Scal_{m,\pp}^wf\|_{L^2(v_0dw)}\lesssim\|\Scal_{m,\pp}^wf\|_{L^2(v_0dw)}.
\end{align}

Finally,  \eqref{generalsubordinationpoison}, Fubini's theorem, the fact that $e^{-\tau L_w}\in \mathcal{O}(L^2(w)-L^{q_0}(w))$,  \eqref{doublingcondition}, and Proposition \ref{prop:alpha} imply
\begin{multline*}
II
\lesssim\sum_{j\geq 1}e^{-c4^j}
\int_{\frac{1}{4}}^{\infty}u^{\theta}e^{-u}\frac{du}{u}\left(\int_{\R^n}\int_0^{\infty}\int_{B(x,2^{j+1}t)}|(t\sqrt{L})^me^{-t\sqrt{L}}f(y)|^{2}\frac{dw(y)dt}{tw(B(y,t))}
v_0(x)dw(x)\right)^{\frac{1}{2}}
\\
\lesssim\sum_{j\geq 1}e^{-c4^j}2^{j\theta_{v_0,w}}\|\Scal_{m,\pp}^wf\|_{L^2(v_0dw)}
\lesssim\|\Scal_{m,\pp}^wf\|_{L^2(v_0dw)}.
\end{multline*}
This and \eqref{Iverticalconicalpoisson} imply \eqref{extrapolationpoissonverticalconnical} and the proof is complete.
\qed
\subsubsection{Proof of Theorem \ref{thm:conical-verticalnon-gradientdegenerate}, part $(d)$}
Fix $w\in A_2$, $f\in L^2(w)$, and $m\in \N$.
We shall use extrapolation to prove this result. In particular, by Theorem \ref{theor:extrapol}, part $(d)$ (or part $(a)$ if $p_+(L_w)^*_w=\infty$) it is enough to show that 
 \begin{align}\label{extrapolationpoissonconnicalvertical}
\|\Scal_{m,\pp}^wf\|_{L^2(v_0dw)}\lesssim\|\mathsf{s}_{m,\pp}^wf\|_{L^2(v_0dw)},\quad \forall v_0\in A_{\frac{2}{p_-(L_w)}}(w)\cap RH_{\left(\frac{p_+(L_w)_w^*}{2}\right)'}(w).
 \end{align}
To prove this inequality,
first of all note that since $v_0\in A_{\frac{2}{p_-(L_w)}}(w)\cap RH_{\left(\frac{p_+(L_w)^*_w}{2}\right)'}(w)$ by Remark \ref{remark:choicereverseholder} (with $\widetilde{p}=p_-(L_w)$, $\widetilde{q}=p_+(L_w)$, $q_1=2$, and $\widetilde{k}=1$), we can find $p_0, \widehat{r}, r, q_0$ such that $r_w<\widehat{r}<2$, $1<\frac{q_0}{2}\leq r<\infty$, $v_0\in A_{\frac{2}{p_0}}(w)\cap RH_{r'}(w)$, and
\begin{align}\label{positivecomparisonpoisson}
1+\widehat{r}n\left(\frac{1}{2r}-\frac{1}{q_0}\right)>0.
\end{align}

 Changing the variable $t$ into $2t$ and applying the subordination formula \eqref{subordinationformula}, and Minkowski's integral inequality, we get
\begin{align*}
\|\Scal_{m,\pp}^wf\|_{L^2(v_0dw)}
\lesssim \int_0^{\frac{1}{4}}u^{\frac{1}{2}}
I(u)\frac{du}{u}
+
\int_{\frac{1}{4}}^{\infty}u^{\frac{1}{2}}e^{-u}
I(u)\frac{du}{u}=:I+II,
\end{align*}
where 
$$
I(u):=\left(\int_{\R^n}\int_{0}^{\infty}\int_{B(x,2t)}|e^{-\frac{t^2}{4u}L_w}(t\sqrt{L_w})^me^{-t\sqrt{L_w}}f(y)|^2\frac{dw(y)\,dt}{tw(B(y,2t))}v_0(x)dw(x)\right)^{\frac{1}{2}}.
$$
For $0<u<1/4$, since $v_0\in A_{\frac{2}{p_0}}(w)\cap RH_{r'}(w)$ with $r\geq q_0/2$, Lemma \ref{lemma:poissonmaximal} with  $\mathcal{T}_{\tau}=e^{-\tau L_w}$, $F(y,t)=(t\sqrt{L_w})^me^{-t\sqrt{L_w}}f(y)$, and $\alpha=2$ implies
\begin{align*}
I(u)
\lesssim u^{n\widehat{r}\left(\frac{1}{4r}-\frac{1}{2q_0}\right)}
\left(\int_{\R^n}\int_0^{\infty}|(t\sqrt{L_w})^me^{-t\sqrt{L_w}}f(y)|^{2}\frac{dt}{t}v_0(y)dw(y)\right)^{\frac{1}{2}}=
u^{n\widehat{r}\left(\frac{1}{4r}-\frac{1}{2q_0}\right)}
\|\mathsf{s}_{m,\pp}^wf\|_{L^2(v_0dw)}.
\end{align*}
Therefore, \eqref{positivecomparisonpoisson} yields
\begin{align*}
I\lesssim \int_{0}^{\frac{1}{4}}u^{\frac{1}{2}+{n\widehat{r}}\left(\frac{1}{4r}-\frac{1}{2q_0}\right)}\frac{du}{u}
\|\mathsf{s}_{m,\pp}^wf\|_{L^2(v_0dw)}\lesssim 
\|\mathsf{s}_{m,\pp}^wf\|_{L^2(v_0dw)}.
\end{align*}

To estimate $II$ apply \eqref{doublingcondition}, the fact that $e^{-\tau L_w}\in \mathcal{O}(L^{p_0}(w)-L^2(w))$, \eqref{Asinpesoconpeso} (recall that $v_0\in A_{\frac{2}{p_0}}$), and Fubini's theorem 
to conclude
\begin{align*}
II&\lesssim\sum_{j\geq 1}e^{-c4^j}\!\!
\int_{\frac{1}{4}}^{\infty}\!\!u^{\theta}e^{-u}\frac{du}{u}\left(\int_{\R^n}\int_0^{\infty}\!\!\left(\dashint_{B(x,2^{j+2}t)}\!\!|(t\sqrt{L_w})^me^{-t\sqrt{L_w}}f(y)|^{p_0}dw(y)\right)^{\!\!\frac{2}{p_0}}\frac{dt}{t}v_0(x)dw(x)\right)^{\!\!\frac{1}{2}}
\\
&\lesssim\sum_{j\geq 1}e^{-c4^j}
\left(\int_{\R^n}\int_0^{\infty}\dashint_{B(x,2^{j+2}t)}|(t\sqrt{L_w})^me^{-t\sqrt{L_w}}f(y)|^{2}d(v_0w)(y)\frac{dt}{t}v_0(x)dw(x)\right)^{\frac{1}{2}}
\\
&\lesssim\sum_{j\geq 1}e^{-c4^j}
\left(\int_{\R^n}\int_0^{\infty}|(t\sqrt{L_w})^me^{-t\sqrt{L_w}}f(y)|^{2}\dashint_{B(y,2^{j+2}t)}d(v_0w)(x)\frac{dt}{t}v_0(y)dw(y)\right)^{\frac{1}{2}}
\\&
\lesssim 
\|\mathsf{s}_{m,\pp}^wf\|_{L^2(v_0dw)}.
\end{align*}
\qed
\subsubsection{Proof of Theorem \ref{thm:verticalconicalgegenerategradient}, parts $(a)$ and $(b)$}Fix $w\in A_2$, $f\in L^2(w)$, and $K\in \N_0$. Following the notation in Proposition \ref{prop:comparison-general} consider ${T}_t=\nabla e^{-t^2L_w}$ and $F(y,t)=t(t\sqrt{L_w})^mf(y)$. Note that $F$ satisfies condition $(i)$ in that proposition. Besides, for $\max\{r_w, (q_-(L_w))_{w,*}\}<p_0<2<q_0<q_+(L_w)$, by Lemma \ref{lem:Gsum} with $\alpha=1$ and $S_t=e^{-\frac{t^2}{2}L_w}$, we have that $T_t$ satisfies conditions $(ii)$ and $(iii)$ in Proposition \ref{prop:comparison-general}. 

Moreover, note that for all $0<p<q_{+}(L_w)$ and $v\in RH_{\left(\frac{q_+(L_w)}{p}\right)'}(w)$, we can find $q_0$ satisfying
$\max\{2,p\}<q_0<q_+(L_w)$ so that $v\in RH_{\left(\frac{q_0}{p}\right)'}(w)$. Hence,  Proposition \ref{prop:comparison-general}, part $(b)$,  implies
$$
\|\mathsf{g}_{K,\hh}^wf\|_{L^p(vdw)}
\lesssim
\|\Grm_{K,\hh}^wf\|_{L^p(vdw)},
$$
which proves part $(a)$.

If we now take $\max\{r_w,(q_-(L_w))_{w,*}\}<p<\infty$ and $v\in A_{\frac{p}{\max\{r_w,(q_-(L_w))_{w,*}\}}}(w)$, we can find $p_0$ satisfying 
$\max\{r_w,(q_-(L_w))_{w,*}\}<p_0<\min\{2,p\}$ (recall that $r_w<2$ and $(q_-(L_w))_{w,*}<q_-(L_w)<2$) so that $v\in A_{\frac{p}{p_0}}(w)$. Hence,  Proposition \ref{prop:comparison-general}, part $(a)$,  implies
$$
\|\Grm_{K,\hh}^wf\|_{L^p(vdw)}
\lesssim
\|\mathsf{g}_{K,\hh}^wf\|_{L^p(vdw)},
$$
which proves part $(b)$.
\qed
\subsubsection{Proof of Theorem \ref{thm:verticalconicalgegenerategradient}, part $(c)$}
We proceed as in the proof of Theorem \ref{thm:conical-verticalnon-gradientdegenerate}, part $(c)$, but in this case we need to prove, for every $K\in \N_0$,
 \begin{align*}
 \|\mathsf{g}_{K,\pp}^wf\|_{L^2(v_0dw)}\lesssim\|\Grm_{K,\pp}^wf\|_{L^2(v_0dw)}, \quad \forall v_0\in A_{\frac{2}{\max\{r_w,(q_-(L_w))_{w,*}\}}}(w)\cap RH_{\left(\frac{q_+(L_w)}{2}\right)'}(w), 
 \end{align*}
instead of \eqref{extrapolationpoissonverticalconnical}. 
To this end, fix $p_0$ and $q_0$ satisfying 
$\max\{r_w,(q_-(L_w))_{w,*}\}<p_0<2<q_0<q_+(L_w)$ so that $v_0\in A_{\frac{2}{p_0}}(w)\cap RH_{\left(\frac{q_0}{2}\right)'}(w)$.
From now on, the proof follows the lines of that of Theorem \ref{thm:conical-verticalnon-gradientdegenerate}, part $(c)$, replacing $p_+(L_w)$ with $q_+(L_w)$, using \eqref{verticalsinRH} with 
$T_{t^2}=\nabla e^{-\frac{t^2}{4u}L_w}$ and
$F(y,t)=t(t\sqrt{L_w})^Ke^{-t\sqrt{L_w}}f(y)$, and 
 Remark \ref{remark:possonmaximal} with  the previous $F(y,t)$, instead of Lemma \ref{lemma:poissonmaximal}, and   Lemma \ref{lem:Gsum} with $q=q_0$, $p=2$, $\alpha=2\sqrt{u}$, and  $S_t$ equal to the identity, instead of the fact that $e^{-\tau L_w}\in \mathcal{O}(L^2(w)-L^{q_0}(w))$. 
\qed
\subsubsection{Proof of Theorem \ref{thm:verticalconicalgegenerategradient}, part $(d)$} 
We proceed as in the proof of Theorem \ref{thm:conical-verticalnon-gradientdegenerate}, part $(d)$, but now we prove, for every $K\in  \N_0$,
 \begin{align*}
\|\Grm_{K,\pp}^wf\|_{L^2(v_0dw)} \lesssim\|\mathsf{g}_{K,\pp}^wf\|_{L^2(v_0dw)},\quad \forall v_0\in A_{\frac{2}{\max\{r_w,(q_-(L_w))_{w,*}\}}}(w)\cap RH_{\left(\frac{q_+(L_w)_w^*}{2}\right)'}(w),
 \end{align*}
instead of \eqref{extrapolationpoissonconnicalvertical}.
To this end, note that, by Remark \ref{remark:choicereverseholder} (with $\widetilde{p}=\max\{r_w,(q_-(L_w))_{w,*}\}$, $\widetilde{q}=q_+(L_w)$, $q_1=2$, and $\widetilde{k}=1$), we can find $p_0, \widehat{r}, r, q_0$ such that $r_w<\widehat{r}<2$, $1<\frac{q_0}{2}\leq r<\infty$, $v_0\in A_{\frac{2}{p_0}}(w)\cap RH_{r'}(w)$, and
\begin{align*}
1+\widehat{r}n\left(\frac{1}{2r}-\frac{1}{q_0}\right)>0.
\end{align*}
Hereafter, we can follow the proof of Theorem \ref{thm:conical-verticalnon-gradientdegenerate}, part $(d)$, but using Remark \ref{remark:possonmaximal} with $F(y,t)=t(t\sqrt{L_w})^Ke^{-t\sqrt{L_w}}f(y)$, instead of Lemma \ref{lemma:poissonmaximal}, and   Lemma \ref{lem:Gsum} with $p=p_0$, $q=2$, $\alpha=2\sqrt{u}$, and $S_t$ equal to the identity, instead of the fact that $e^{-\tau L_w}\in \mathcal{O}(L^{p_0}(w)-L^{2}(w))$. 
\qed

\subsection{Boundedness results}
 \subsubsection*{Proof of Theorem \ref{thm:boundednessconicalodd}}
The proof follows at once by Theorem \ref{thm:boundednessconicaleven} and Propositions \ref{prop:widetildeS-heatS}, part $(b)$, \ref{prop:coniclalpoisson-heat},  and \ref{prop:gradient-singradientheat}. \qed
 
 \medskip
 
 \subsubsection*{Proof of Theorem \ref{thm:boundednessverticalheat} } 
The boundedness of $\mathsf{s}^w_{m,\hh}$ follows by Theorem \ref{thm:conical-verticalnon-gradientdegenerate}, part ($a$), and Theorems \ref{thm:boundednessconicalodd} and \ref{thm:boundednessconicaleven}, since $p_+(L_w)<p_+(L_w)^{m,*}_w$, for all $m\in \N$.

The boundedness of $\mathsf{s}^w_{m,\pp}$ follows by \eqref{verticalpoissonverticalheat} and the boundedness of $\mathsf{s}^w_{m,\hh}$.
\qed

\medskip
 
 \subsubsection*{Proof of Theorem \ref{thm:boundednessverticalpoisson} }
The boundedness of $\mathsf{g}^w_{K,\hh}$ and $\mathsf{g}^w_{K,\pp}$ follow by \eqref{verticalverticalnorm} and Theorem \ref{thm:boundednessverticalheat}, since $\mathcal{W}_v^w(q_-(L_w),q_+(L_w))\subset
\mathcal{W}_v^w(p_-(L_w),p_+(L_w))$.\qed

\medskip

%
 \subsubsection*{Proof of Theorem \ref{thm:boundednessnon-tangential} }
The proof follows by Theorem \ref{thm:improvementnontangentialpoisson} and the fact that $\Ncal_{\hh}^w$ and $\Scal_{2,\hh}^w$ are bounded operators on $L^p(vdw)$ for all $p\in \mathcal{W}_v^w(p_-(L_w),\infty)$ (see Theorem \ref{thm:boundednessconicaleven}, and \cite{ChMPA18} or \cite[Proposition 7.1]{MaPAII17}).\qed

\medskip

Finally, we prove Theorem \ref{thm:boundednesssaquareroot}.
To this end, we need the following Calder\'on-Zygmund decomposition for functions on weighted Sobolev spaces
\begin{lemma}{\cite[Lemma 6.6]{AMIII06}.}\label{lem:CZweighted}
Let $n\ge 1$, $\alpha>0$, $\varpi\in A_{\infty}$, and $1\le p<\infty$ such that $\varpi\in A_p$. Assume that $f\in \Scal $ such that $\|\nabla f\|_{L^p(\varpi)}<\infty$.  Then, there exist a collection of balls $\{B_i\}_{i}$ with radii $r_{B_i}$, smooth functions $\{b_i\}_i$, and  a function $g\in L_{\loc}^1(\varpi)$ such that
\begin{equation}\label{CZ:decomp}
f=g+\sum_{i=1}^{\infty}b_i
\end{equation}
and the following properties hold:
\begin{equation}\label{CZ:g}
|\nabla g(x)| \leq C\alpha,\quad \text{for }\mu-\text{a.e. }x,
\end{equation}
\begin{equation}\label{CZ:b}
\supp b_i\subset B_i\quad \text{and}\quad \int_{B_i}|\nabla b_i(x)|^p d\varpi(x)\leq C\alpha^p \varpi(B_i),
\end{equation}
\begin{equation}\label{CZ:sum}
\sum_{i=1}^{\infty}\varpi(B_i)\leq \frac{C}{\alpha^p}\int_{\R^n}|\nabla f(x)|^p d\varpi(x),
\end{equation}
\begin{equation}\label{CZ:overlap}
\sum_{i=1}^{\infty}\chi_{4B_i}\leq N,
\end{equation}
where $C$ and $N$ depend only on $n$, $p$, and $\varpi$. In addition,  for $1\le q<p_{\varpi}^*$, where $p_{\varpi}^*$ is defined in \eqref{p_w^*},
\begin{equation}\label{CZ:PS}
\br{\fint_{B_i} |b_i|^q d\varpi}^{{\frac{1}{q}}} \lesssim \alpha r_{B_i}.
\end{equation}
\end{lemma}

\subsubsection*{Proof of Theorem \ref{thm:boundednesssaquareroot}}
In \cite[Proposition 6.1]{CMR15} the authors proved that, for all $f\in \mathcal{S}$ and
 $$
 p\in \mathcal{W}_v^w(\max\{r_w,p_-(L_w)\},p_+(L_w))=
 \left(\mathfrak{r}_v(w)\max\{r_w,p_-(L_w)\},\frac{p_+(L_w)}{\mathfrak{s}_v(w)}\right),$$
it holds
\begin{align}\label{boundednesssquareroot}
\|\sqrt{L_w}f\|_{L^p(vdw)}\lesssim \|\nabla f\|_{L^{p}(vdw)}.
\end{align}
Here we extend this boundedness for  all
$$
p\in \mathcal{W}_v^w(\max\{r_w,(p_-(L_w))_{w,*}\},p_+(L_w))=
 \left(\mathfrak{r}_v(w)\max\{r_w,(p_-(L_w))_{w,*}\},\frac{p_+(L_w)}{\mathfrak{s}_v(w)}\right),
$$
(we recall that $(p_-(L_w))_{w,*}<p_-(L_w)$, see \eqref{p_{w,*}}).

First of all, note that we may assume that $r_w<p_{-}(L_w)$. Otherwise, 
$$\max\{r_w,p_-(L_w)\}=\max\{r_w,(p_-(L_w))_{w,*}\}$$ and by  \cite[Proposition 6.1]{CMR15} (see \eqref{boundednesssquareroot}) the proof would be complete. Hence, assuming that $r_w<p_-(L_w)$, let us extend
\eqref{boundednesssquareroot}  
for all $p\in\mathcal{W}_v^w(\max\{r_w,(p_-(L_w))_{w,*}\},p_+(L_w))$.

To lighten the proof we denote $p_-:=p_-(L_w)$ and $p_+:=p_+(L_w)$. Besides, we fix  $p$ satisfying:
\begin{align}\label{fixp}
\mathfrak{r}_v(w)\max\{r_w,(p_-)_{w,*}\}<p<\mathfrak{r}_v(w)p_-=\mathfrak{r}_v(w)\max\{r_w,p_-\}.
 \end{align}
We shall show that
 \begin{align}\label{weaksquarroot}
 vw\left(\left\{x\in \R^n:\sqrt{L_w}f(x)>\alpha\right\}\right)\lesssim \frac{1}{\alpha^p}\int_{\R^n}|\nabla f(x)|^p v(x)dw(x), \quad \forall \,\alpha>0.
 \end{align}
This, together with \eqref{boundednesssquareroot}, will allow us to conclude  the proof Theorem \ref{thm:boundednesssaquareroot} by interpolation  (see \cite{Ba09} and recall that by Remark \ref{remark:product-weight} $vw\in A_{\infty}$).
Thus, let us prove \eqref{weaksquarroot}.

To this end, note that the interval
\begin{align}\label{emptyinterval}
\left(\mathfrak{r}_v(w)p_-,\min\left\{\frac{p_+}{\mathfrak{s}_v(w)},p_{vw}^*\right\}\right)\neq \emptyset.
\end{align}
Indeed, by hypothesis we have that $\mathcal{W}_v^w(\max\{r_w,p_-\},p_+)\neq \emptyset$. Then, in view of \eqref{intervalrsw} and recalling that we are assuming that  $r_w<p_-$, our hypothesis implies that
$$
\mathfrak{r}_v(w)p_-<\frac{p_+}{\mathfrak{s}_v(w)}.
$$
Therefore, we just need to show that 
$$
\mathfrak{r}_v(w)p_-<p_{vw}^*.
$$
In order to prove this, notice that we can assume that $nr_{vw}>p$ (otherwise $p_{vw}^*=\infty$ and the inequality is trivial). Hence, by \eqref{fixp}, \eqref{p_w^*}, and Remark \ref{remark:product-weight},
\begin{multline*}
\frac{1}{p_{vw}^*}=\frac{1}{p}-\frac{1}{nr_{vw}}
<\frac{1}{\mathfrak{r}_v(w)(p_-)_{w,*}}-\frac{1}{nr_{vw}}
=
\frac{nr_w+p_-}{\mathfrak{r}_v(w)p_-nr_w}-\frac{1}{nr_{vw}}
\\
=
\frac{1}{\mathfrak{r}_v(w)p_-}-\frac{1}{nr_{vw}}\left(1-\frac{r_{vw}}{r_w\mathfrak{r}_v(w)}\right)
\leq 
\frac{1}{\mathfrak{r}_v(w)p_-}.
\end{multline*}

Therefore, in view of \eqref{emptyinterval}, we can take
\begin{align}\label{fixp_1}
\mathfrak{r}_v(w)p_-<p_1<\min\left\{\frac{p_+}{\mathfrak{s}_v(w)},p_{vw}^*\right\}.
\end{align}
In particular,
\begin{align}\label{p_1weight}
v\in A_{\frac{p_1}{p_-}}(w)\cap RH_{\left(\frac{p_+}{p_1}\right)'}(w).
\end{align}
Next, fix $\alpha >0$ and take a Calder\'on-Zygmund decomposition of $f$ at height $\alpha$ as in Lemma \ref{lem:CZweighted}, for $\varpi=vw$, and $p$ as in \eqref{fixp}.  Note that, by Remark \ref{remark:product-weight} $vw\in A_{\infty}$ and $r_{vw}\leq r_{w}\mathfrak{r}_v(w)<p$. Moreover, let $b_i$, $g$, and $\{B_i\}$ be the functions and the collection of balls given by Lemma \ref{lem:CZweighted}, and let $M\in \N$. We define $B_{r_{B_i}}:=(I-e^{-r_{B_i}^2L_w})^M$ and $A_{r_{B_i}}:=I-B_{r_{B_i}}=\sum_{k=1}^MC_{k,M}e^{-kr_{B_i}^2}$.
Hence, we can write $f=g+\sum_{i\in \N}A_{r_{B_i}}b_i+\sum_{i\in \N}B_{r_{B_i}}b_i=:g+\widetilde{b}+\widehat{b}$, and then
\begin{multline}\label{weakdecomposition}
vw\left(\left\{x\in \R^n:\sqrt{L_w}f(x)>\alpha\right\}\right)\leq
vw\left(\left\{x\in \R^n:\sqrt{L_w}g(x)>\frac{\alpha}{3}\right\}\right)
\\
+
vw\left(\left\{x\in \R^n:\sqrt{L_w}\,\widetilde{b}(x)>\frac{\alpha}{3}\right\}\right)
+
vw\left(\left\{x\in \R^n:\sqrt{L_w}\,\widehat{b}(x)>\frac{\alpha}{3}\right\}\right)=:I+II+III.
\end{multline}
In order to estimate $I$, first recall that $p<p_1$ (see \eqref{fixp}
 and \eqref{fixp_1}). Then,
apply Chebyshev's inequality, \eqref{boundednesssquareroot}, and properties \eqref{CZ:g}-\eqref{CZ:overlap} to obtain
\begin{multline}\label{termIsquareroot}
I\lesssim 
\frac{1}{\alpha^{p_1}}\int_{\R^n}|\sqrt{L_w}g(x)|^{p_1}v(x)dw(x)
\lesssim
\frac{1}{\alpha^{p_1}}\int_{\R^n}|\nabla g(x)|^{p_1}v(x)dw(x)
\lesssim
\frac{1}{\alpha^{p}}\int_{\R^n}|\nabla g(x)|^{p}v(x)dw(x)
\\
\lesssim
\frac{1}{\alpha^{p}}\left(\int_{\R^n}|\nabla f(x)|^{p}v(x)dw(x)
+\sum_{i\in \N}
\int_{\R^n}|\nabla b_i(x)|^{p}v(x)dw(x)\right)
\lesssim
\frac{1}{\alpha^p}\int_{\R^n}|\nabla f(x)|^pv(x)dw(x).
\end{multline}
In order to estimate  $II$ and $III$ we shall use the following inequality:
\begin{multline}\label{maximal-u}
\left(\sum_{i\in \N} \int_{B_i}\left(\mathcal{M}^{vw}(|u|^{p_1'})(x)\right)^{\frac{1}{p_1'}}v(x)dw(x)\right)^{p_1}
\lesssim
\left(\int_{\cup_{i\in \N}B_i}\left(\mathcal{M}^{vw}(|u|^{p_1'})(x)\right)^{\frac{1}{p_1'}}v(x)dw(x)\right)^{p_1}
\\
\lesssim
vw(\cup_{i\in \N}B_i)\|u\|_{L^{p_1'}(vdw)}^{p_1}\lesssim \frac{1}{\alpha^{p}}\int_{\R^n}|\nabla f(x)|^{p}v(x)dw(x),
\end{multline}
where  $u\in L^{p_1'}(vdw)$ such that $\|u\|_{L^{p_1'}(vdw)}=1$. The inequality follows by 
Kolmogorov's inequality (see \cite[Exercise 2.1.5]{GrafakosI}, and follow the proof suggested there replacing the Lebesgue measure with the measure given by the weight $vw$). Besides, in the last inequality we have applied  \eqref{CZ:sum}.

After this observation let us estimate $II$.
By Chebyshev's inequality, \eqref{boundednesssquareroot} and the definition of $\widetilde{b}$, we have
\begin{multline}\label{plugtermII}
II\lesssim 	\frac{1}{\alpha^{p_1}}\int_{\R^n}
\big|\sqrt{L_w}\,\widetilde{b}(x)\big|^{p_1}v(x)dw(x)
\lesssim
\frac{1}{\alpha^{p_1}}\int_{\R^n}
\big|\nabla\widetilde{b}(x)\big|^{p_1}v(x)dw(x)
\\
\lesssim
\frac{1}{\alpha^{p_1}}\left(\sup_{\|u\|_{L^{p_1'}(vdw)}=1}\sum_{i\in \N}\sum_{k=1}^{M}C_{k,M}\sum_{j\geq 1}\int_{C_j(B_i)}
|\nabla e^{-kr_{B_i}^2L_w}b_i(x)|\,|u(x)|v(x)dw(x)\right)^{p_1}.
\end{multline}
By H\"older's inequality and the fact that $\sqrt{\tau}\nabla e^{-\tau L_w}\in \mathcal{F}(L^{p_1}(vw)-L^{p_1}(vw))$ (see \eqref{fixp_1} and Lemma \ref{lem:ODweighted}), we estimate the integral in $x$ as follows
\begin{align*}
\int_{C_j(B_i)}&
|\nabla e^{-kr_{B_i}^2L_w}b_i(x)|\,|u(x)|v(x)dw(x)
\\&
\lesssim\frac{1}{r_{B_i}}
\left(\int_{C_j(B_i)}
\big|\sqrt{k}r_{B_i}\nabla e^{-kr_{B_i}^2L_w}b_i(x)\big|^{p_1}d(vw)(x)\right)^{\frac{1}{p_1}}
\left(\int_{C_j(B_i)}
|u(x)|^{p_1'}d(vw)(x)\right)^{\frac{1}{p_1'}}
\\&
\lesssim\frac{e^{-c4^j}}{r_{B_i}}vw(2^{j+1}B_i)
\left(\dashint_{B_i}
|b_i(x)|^{p_1}d(vw)(x)\right)^{\frac{1}{p_1}}
\left(\dashint_{2^{j+1}B_i}
|u(x)|^{p_1'}d(vw)(x)\right)^{\frac{1}{p_1'}}
\\&
\lesssim \alpha e^{-c4^j}vw(B_i)
\inf_{x\in B_i}\left(\mathcal{M}^{vw}(|u|^{p_1'})(x)\right)^{\frac{1}{p_1'}},
\end{align*}
where in the last inequality we have used \eqref{CZ:PS} (see \eqref{fixp_1}), and \eqref{doublingcondition}. Plugging this into \eqref{plugtermII} and applying \eqref{maximal-u},  we have
\begin{multline}\label{termIIsquareroot}
II
\lesssim 
\left(\sup_{\|u\|_{L^{p_1'}(vdw)}=1}\sum_{i\in \N}\sum_{k=1}^{M}C_{k,M}\sum_{j\geq 1}e^{-c4^{j}}\int_{B_i}\left(\mathcal{M}^{vw}(|u|^{p_1'})(x)\right)^{\frac{1}{p_1'}}v(x)dw(x)\right)^{p_1}
\\
\lesssim 
\left(\sup_{\|u\|_{L^{p_1'}(vdw)}=1}\sum_{i\in \N}\int_{B_i}\left(\mathcal{M}^{vw}(|u|^{p_1'})(x)\right)^{\frac{1}{p_1'}}v(x)dw(x)\right)^{p_1}
\lesssim \frac{1}{\alpha^p}\int_{\R^n}|\nabla f(x)|^pv(x)dw(x).
\end{multline}
Finally, we estimate $III$. By \eqref{CZ:sum} and \eqref{doublingcondition}, we have
\begin{multline}\label{termIIIsquareroot}
III\lesssim 
vw\Bigg(\bigcup_{i\in \N}4B_i\Bigg)+
vw\left(\left\{x\in \R^n\setminus \bigcup_{i\in \N}4B_i:\sqrt{L_w}\,\widehat{b}(x)>\frac{\alpha}{3}\right\}\right)
\\
\lesssim
\frac{1}{\alpha^p}
\int_{\R^n}|\nabla f(x)|^pv(x)dw(x)+\mathcal{III},
\end{multline}
where 
$$
\mathcal{III}:=
vw\left(\left\{x\in \R^n\setminus \bigcup_{i\in \N}4B_i:\sqrt{L_w}\,\widehat{b}(x)>\frac{\alpha}{3}\right\}\right).
$$
Again by Chebyshev's inequality and duality, proceeding as in the estimate of term $II$, we have
\begin{multline}\label{termIII}
\mathcal{III}
\lesssim
\frac{1}{\alpha^{p_1}}
\left(
\sup_{\|u\|_{L^{p_1'}(vdw)=1}}
\sum_{i\in \N}
\sum_{j\geq 2}
\left(\int_{C_j(B_i)}
\left|\sqrt{L_w}B_{r_{B_i}}b_i(x)\right|^{p_1}v(x)dw(x)
\right)^{\frac{1}{p_1}}
\|u\chi_{C_j(B_i)}\|_{L^{p_1'}(vdw)}
\right)^{p_1}
\\
=:
\frac{1}{\alpha^{p_1}}
\left(
\sup_{\|u\|_{L^{p_1'}(vdw)=1}}
\sum_{i\in \N}
\sum_{j\geq 2}\mathcal{III}_{ij}
\|u\chi_{C_j(B_i)}\|_{L^{p_1'}(vdw)}
\right)^{p_1}.
\end{multline}
We estimate $\mathcal{III}_{ij}$ by using \eqref{squareroot} and Minkowski's integral inequality:
 \begin{multline}\label{plugtermIIIij}
\mathcal{III}_{ij}
\lesssim
\left(\int_{C_j(B_i)}\left(\int_{0}^{\infty}
\left|tL_we^{-t^2L_w}B_{r_{B_i}}b_i(x)\right|\frac{dt}{t}\right)^{p_1}v(x)dw(x)
\right)^{\frac{1}{p_1}}
\\
\lesssim
\int_{0}^{\infty}
\left(\int_{C_j(B_i)}
\left|tL_we^{-t^2L_w}B_{r_{B_i}}b_i(x)\right|^{p_1}v(x)dw(x)
\right)^{\frac{1}{p_1}}\frac{dt}{t}.
\end{multline}
We compute the above integral in $x$ by using functional calculus.   The notation is taken from  \cite{Au07}, \cite[Section 7]{AMIII06}, and \cite{CMR15}.
We write $\vartheta\in[0,\pi/2)$ for the supremum of $|{\rm arg}(\langle L_wf,f\rangle_{L^2(w)})|$ over all $f$ in the domain of $L_w$.
Let $0<\vartheta <\theta<\nu<\mu<\pi /2$ and note that, for a fixed $t>0$, $\phi(z,t):=e^{-t^2 z}(1-e^{-r_{B_i}^2 z})^M$ is holomorphic in the open sector $\Sigma_\mu=\{z\in\mathbb{C}\setminus\{0\}:|{\rm arg} (z)|<\mu\}$ and satisfies $|\phi(z,t)|\lesssim |z|^M\,(1+|z|)^{-2M}$ (with implicit constant depending on $\mu$, $t>0$, $r_{B_i}$, and $M$) for every $z\in\Sigma_\mu$. 
Hence, we can write
$$
\phi(L_w,t)=\int_{\Gamma } e^{-zL_w }\eta (z,t)dz,
\qquad \text{where} \quad 
\eta(z,t) = \int_{\gamma} e^{\zeta z} \phi(\zeta,t) d\zeta.
$$
Here $\Gamma=\partial\Sigma_{\frac\pi2-\theta}$ with positive orientation (although orientation is irrelevant for our computations) and 
$\gamma=\R_+e^{i\,{\rm sign}({\rm Im} (z))\,\nu}$. It is not difficult to see that for every $z\in \Gamma$,
$$
|\eta(z,t)| \lesssim \frac{r_{B_i}^{2M}}{(|z|+t^2)^{M+1}}.
$$
By these observations, the fact that $zL_we^{-zL_w}\in \mathcal{O}(L^{p_1}(vw)-L^{p_1}(vw))$ (see \eqref{p_1weight}), \eqref{CZ:PS} (recall \eqref{fixp_1}) and since $j\geq 2$, we have
\begin{align}\label{plugtermIIIij2}
\Bigg(\int_{C_j(B_i)}&
\left|tL_we^{-t^2L_w}B_{r_{B_i}}b_i(x)\right|^{p_1}v(x)dw(x)
\Bigg)^{\frac{1}{p_1}}
\\\nonumber&
\lesssim vw(2^{j+1}B_i)
\int_{\Gamma}
\left(\dashint_{C_j(B_i)}
\left|zL_we^{-zL_w}b_i(x)\right|^{p_1}d(vw)(x)
\right)^{\frac{1}{p_1}}\frac{r_{B_i}^{2M}t}{(|z|+t^2)^{M+1}}\frac{|dz|}{|z|}
\\\nonumber&
\lesssim vw(2^{j+1}B_i)^{\frac{1}{p_1}}2^{j\theta_1}
\int_{\Gamma}\Upsilon\left(\frac{2^jr_{B_i}}{|z|^{\frac{1}{2}}}\right)^{\theta_2}
e^{-c\frac{4^jr_{B_i}^2}{|z|}}
\frac{r_{B_i}^{2M}t}{(|z|+t^2)^{M+1}}\frac{|dz|}{|z|}\left(\dashint_{B_i}
|b_i(x)|^{p_1}d(vw)(x)
\right)^{\frac{1}{p_1}}
\\\nonumber&
\lesssim \alpha r_{B_i}vw(2^{j+1}B_i)^{\frac{1}{p_1}}2^{j\theta_1}
\int_{0}^{\infty}\Upsilon\left(\frac{2^jr_{B_i}}{s^{\frac{1}{2}}}\right)^{\theta_2}
e^{-c\frac{4^jr_{B_i}^2}{s}}
\frac{r_{B_i}^{2M}t}{(s+t^2)^{M+1}}\frac{ds}{s}.
\\\nonumber&
\lesssim \alpha r_{B_i}vw(2^{j+1}B_i)^{\frac{1}{p_1}}2^{j\theta_1}
\int_{0}^{\infty}\Upsilon\left(s\right)^{\theta_2}
e^{-cs^2}
\frac{r_{B_i}^{2M}t}{(4^jr_{B_i}^2/s^2+t^2)^{M+1}}\frac{ds}{s},
\end{align}
where in the last inequality we have change the variable $s$ into $4^jr_{B_i}^2/s^2$.
Plugging this and \eqref{plugtermIIIij2} into \eqref{plugtermIIIij}, and changing the variable  $t$ into $2^{j}r_{B_i}t$, we obtain, for $M\in \N$ such that $2M>\theta_2$,
\begin{align*}
\mathcal{III}_{ij}&
\lesssim \alpha  vw(2^{j+1}B_i)^{\frac{1}{p_1}}2^{-j(2M+1-\theta_1)}
\int_{0}^{\infty}t\int_{0}^{\infty}\Upsilon\left(s\right)^{\theta_2}
e^{-cs^2}
\frac{1}{(1/s^2+t^2)^{M+1}}\frac{ds}{s}
\frac{dt}{t}
\\
&
\lesssim \alpha  vw(2^{j+1}B_i)^{\frac{1}{p_1}}2^{-j(2M+1-\theta_1)}
\Bigg(
\int_{0}^{1}t\frac{dt}{t}\int_{0}^{\infty}\Upsilon\left(s\right)^{\theta_2}
e^{-cs^2}s^{2M+2}
\frac{ds}{s}
\\&\hspace*{5cm}
+
\int_{1}^{\infty}t^{-1}\frac{dt}{t}\int_{0}^{\infty}\Upsilon\left(s\right)^{\theta_2}
e^{-cs^2}s^{2M}
\frac{ds}{s}\Bigg)
\\
&
\lesssim \alpha vw(2^{j+1}B_i)^{\frac{1}{p_1}}2^{-j(2M+1-\theta_1)}.
\end{align*}
Consequently, in view of \eqref{termIII},
by \eqref{pesosineq:Ap} and taking  $M\in \N$ large enough satisfying
$2M>\max\left\{\theta_2,\theta_1+r_w\mathfrak{r}_v(w)n-1\right\},$ we get
\begin{multline*}
\mathcal{III}
\lesssim 
\left(
\sup_{\|u\|_{L^{p_1'}(vdw)=1}}
\sum_{i\in \N}vw(B_i)
\inf_{x\in B_i}\left(\mathcal{M}^{vw}(|u|^{p_1'})(x)\right)^{\frac{1}{p_1'}}
\right)^{p_1}
\\
\lesssim 
\left(
\sup_{\|u\|_{L^{p_1'}(vdw)=1}}
\int_{\cup_{i\in \N}B_i}
\left(\mathcal{M}^{vw}(|u|^{p_1'})(x)\right)^{\frac{1}{p_1'}}v(x)dw(x)
\right)^{p_1}
\lesssim
\frac{1}{\alpha^p}\int_{\R^n}|\nabla f(x)|^pv(x)dw(x),
\end{multline*}
where in the last inequality we have used \eqref{maximal-u}.
This and \eqref{termIIIsquareroot} imply that 
$$III\lesssim \frac{1}{\alpha^p}\int_{\R^n}|\nabla f(x)|^pv(x)dw(x),$$ which, in turn, together with \eqref{weakdecomposition}, \eqref{termIsquareroot}, and  \eqref{termIIsquareroot},   implies \eqref{weaksquarroot}, and the proof is complete. 
\qed

%
\section{Unweighted boundedness for degenerate operators}\label{sec:unweighted}
In this section we prove unweighted results for degenerate operators. That is, we show boundedness on $L^p(\R^n)$ for the degenerate operators defined in \eqref{vertical-H}-\eqref{squareroot}. We obtain this from our results on $L^p(vdw)$, by taking $v=w^{-1}$.  The statements of our results are written so that the ranges where we obtain those boundedness results depend  on the weight $w$, but do not depend on the operator $L_w$.

Before being more precise, we observe that
since  we assume that $n\geq 2$, we have that $2^{*}_w=\infty$ if and only if $n=2$ and $r_w=1$. Otherwise,
$$
2^{*}_w=\frac{nr_w2}{nr_w-2}.
$$
In particular,
\begin{align}\label{2w*}
2_w^*=\begin{cases}
\frac{2nr_w}{nr_w-2}&\textrm{if}\,\,nr_w>2,
\\
\infty &\textrm{if}\,\, nr_w=2;
\end{cases}
\quad\textrm{and}\quad
(2_w^*)'=\begin{cases}
\frac{2nr_w}{nr_w+2}&\textrm{if}\,\,nr_w>2,
\\
1 &\textrm{if}\,\, nr_w=2.
\end{cases}
\end{align}
Besides, for all $M\in \N$
\begin{align}\label{2w*M}
(2^*_w)^{M,*}_w=2^{M+1,*}_w.
\end{align}
Finally, from \cite{CMR15} we know that
\begin{align}\label{chainpq2pq}
p_-(L_w)=q_-(L_w)<(2_w^*)'<2<\min\{2_w^*,q_+(L_w)\}\leq\max\{2_w^*,q_+(L_w)\}<q_+(L_w)_w^*<p_+(L_w).
\end{align}

\begin{corollary} 
Given $w \in A_2$, for $m\in \N$ and $(2_w^*)'s_w<p<\frac{2_w^*}{r_w}$, the vertical square functions $\mathsf{s}_{m,\hh}^w$ and $\mathsf{s}_{m,\pp}^w$ can be extended to   bounded operators on $L^p(\R^n)$. In particular, we have this in the following cases:
 \begin{list}{$(\theenumi)$}{\usecounter{enumi}\leftmargin=1cm \labelwidth=1cm\itemsep=0.2cm\topsep=.0cm \renewcommand{\theenumi}{\alph{enumi}}}
 
 \item If $r_w=1$, for $n=2$,  $w\in A_1\cap RH_{p'}$, $1<p<\infty$; and for $n>2$, $w\in A_1\cap RH_{\left(\frac{(n+2)p}{2n}\right)'}$, $\frac{2n}{n+2}<p<\frac{2n}{n-2}$.
 
  \item If $r_w>1$, for $1< r\leq 2$, $w\in A_r\cap 
  RH_{\left(\frac{(nr+2)p}{2nr}\right)'}$, 
   $\frac{2nr}{nr+2}<p<\frac{2n}{nr-2}$.
 
 \end{list}
\end{corollary}
\begin{proof}
Fix $w\in A_2$, $m\in \N$, and $(2_w^*)'s_w<p<\frac{2_w^*}{r_w}$. Then, by \eqref{chainpq2pq}, we have that
$p_-(L_w)s_w<p<\frac{p_+(L_w)}{r_w}$. Consequently, by \eqref{eq:defi:rw}, 
$$
w\in RH_{\left(\frac{p}{p_-(L_w)}\right)'}\cap A_{\frac{p_+(L_w)}{p}},
$$
which in view of \eqref{dualityapclassesRHclasses} yields
$$
w^{-1}\in A_{\frac{p}{p_-(L_w)}}(w)\cap RH_{\left(\frac{p_+(L_w)}{p}\right)'}(w).
$$
Hence taking $v=w^{-1}$ in Theorem \ref{thm:boundednessverticalheat}, we conclude that  $\mathsf{s}_{m,\hh}^w$ and $\mathsf{s}_{m,\pp}^w$ can be extended to  bounded operators on $L^p(\R^n)$.

Assume now that $r_w=1$, in particular $w\in A_2$, and note that, since we are assuming that $n\geq 2$ by \eqref{2w*}, 
\begin{align*}
2_w^*=\begin{cases}
\frac{2n}{n-2}&\textrm{if}\,\,n>2,
\\
\infty &\textrm{if}\,\, n=2;
\end{cases}
\quad\textrm{and}\quad
(2_w^*)'=\begin{cases}
\frac{2n}{n+2}&\textrm{if}\,\,n>2,
\\
1 &\textrm{if}\,\, n=2.
\end{cases}
\end{align*}
Thus, for $n=2$ since $2_w^*=\infty$, the conditions $1<p<\infty$ and
$
w\in A_1\cap RH_{p'},
$
can be written as
$(2_w^*)'s_w<p<\frac{2_w^*}{r_w}$. If now $n>2$, note that again the conditions $\frac{2n}{n+2}<p<\frac{2n}{n-2}$ and
$
w\in A_1 \cap RH_{\left(\frac{(n+2)p}{2n}\right)'},
$
can be written as
$(2_w^*)'s_w<p<\frac{2_w^*}{r_w}$.

Assume next that $r_w>1$ and $w\in A_r\cap 
  RH_{\left(\frac{(nr+2)p}{2nr}\right)'}$, for $1<r\leq 2$ and
  $\frac{2nr}{nr+2}<p<\frac{2n}{nr-2}$. This can be written as $w\in A_r\subseteq A_2$ and, since $r_w<r$,
  $$
  (2_w^*)'s_w=\frac{2nr_w}{nr_w+2}s_w< \frac{2nr}{nr+2}s_w<p<\frac{2nr}{r_w(nr-2)}<\frac{2nr_w}{r_w(nr_w-2)}=
 \frac{2_w^*}{r_w}.$$

\end{proof}
\begin{corollary} 
Given $w \in A_2$, for $K\in \N_0$ and $(2_w^*)'s_w<p\leq\frac{2}{r_w}$, the vertical square functions $\mathsf{g}_{K,\hh}^w$ and $\mathsf{g}_{K,\pp}^w$ can be extended to bounded operators on $L^p(\R^n)$. In particular, we have this in the following cases:
 \begin{list}{$(\theenumi)$}{\usecounter{enumi}\leftmargin=1cm \labelwidth=1cm\itemsep=0.2cm\topsep=.0cm \renewcommand{\theenumi}{\alph{enumi}}}
 
 \item If $r_w=1$, for $n=2$, $w\in A_1\cap RH_{p'}$, and $1<p\leq 2$; and, for $n>2$, $w\in A_1\cap RH_{\left(\frac{(n+2)p}{2n}\right)'}$, and $\frac{2n}{n+2}<p\leq 2$.
 
  \item If $r_w>1$, for $1< r\leq 2$, $w\in A_r\cap 
  RH_{\left(\frac{(nr+2)p}{2nr}\right)'}$, $\frac{2nr}{nr+2}<p\leq \frac{2}{r}$.
 
 \end{list}
\end{corollary}
\begin{proof}
Fix $w\in A_2$, $K\in \N_0$, and $(2_w^*)'s_w<p\leq\frac{2}{r_w}$. Then, by \eqref{chainpq2pq}, we have that
$q_-(L_w)s_w<p<\frac{q_+(L_w)}{r_w}$. Consequently, by \eqref{eq:defi:rw},
$$
w\in RH_{\left(\frac{p}{q_-(L_w)}\right)'}\cap A_{\frac{q_+(L_w)}{p}},
$$
which in view of \eqref{dualityapclassesRHclasses} yields
$$
w^{-1}\in A_{\frac{p}{q_-(L_w)}}(w)\cap RH_{\left(\frac{q_+(L_w)}{p}\right)'}(w).
$$
Hence taking $v=w^{-1}$ in Theorem \ref{thm:boundednessverticalpoisson}, we conclude that $\mathsf{g}_{K,\hh}^w$ and $\mathsf{g}_{K,\pp}^w$ can be extended to bounded operators on $L^p(\R^n)$.

In particular if $r_w=1$,  for $n=2$ since $2_w^*=\infty$ (see \eqref{2w*}), the conditions $1<p\leq 2$ and
$
w\in A_1\cap RH_{p'},
$
can be written as
$(2_w^*)'s_w<p\leq \frac{2}{r_w}$. If now $n>2$, we have that $(2_w^*)'=\frac{2n}{n+2}$. Then,  the conditions $\frac{2n}{n+2}<p\leq 2$ and
$
w\in A_1 \cap RH_{\left(\frac{(n+2)p}{2n}\right)'},
$
can be written as
$(2_w^*)'s_w<p\leq \frac{2}{r_w}$ (see \eqref{2w*}).

If now $r_w>1$, $1<r\leq 2$, and $w\in A_r\cap 
  RH_{\left(\frac{(nr+2)p}{2nr}\right)'}$ with
  $\frac{2nr}{nr+2}<p\leq \frac{2}{r}$. This can be written as $w\in A_r\subseteq A_2$ and, since $r_w<r$,
  $$
  (2_w^*)'s_w=\frac{2nr_w}{nr_w+2}s_w< \frac{2nr}{nr+2}s_w<p\leq\frac{2}{r}<\frac{2}{r_w}.$$

\end{proof}
\begin{corollary}\label{cor:previous}
Given $w \in A_2$, for $m\in \N$ and $(2_w^*)'s_w<p<\frac{2^{2m,*}_{w}}{r_w}$, the conical square functions $\Scal_{2m-1,\hh}^w$, $\Scal_{2m-1,\pp}^w$,  $\Grm_{2m-1,\hh}^w$ and $\Grm_{2m-1,\pp}^w$  can be extended to bounded operators on $L^p(\R^n)$. In particular, this holds in the following cases:
 \begin{list}{$(\theenumi)$}{\usecounter{enumi}\leftmargin=1cm \labelwidth=1cm\itemsep=0.2cm\topsep=.0cm \renewcommand{\theenumi}{\alph{enumi}}}
 
 \item If $r_w=1$, for $n=2$, $w\in A_1\cap RH_{p'}$, and $1<p<\infty$; and, for $n>2$, $w\in A_1\cap RH_{\left(\frac{(n+2)p}{2n}\right)'}$,  and $\frac{2n}{n+2}<p<2_w^{2m,*}$.
 
  \item If $r_w>1$, for $1< r\leq 2$, $w\in A_r\cap RH_{\left(\frac{(nr+2)p}{2nr}\right)'}$, and $\frac{2n}{n+2}<p<\infty$,  if $nr\leq 4m$, and 
   $\frac{2nr}{nr+2}<p<\frac{2n}{nr-4m}$, if $nr>4m$.
  \end{list}  
    Note that $\big(\frac{2nr}{nr+2},\frac{2n}{nr-4m}\big)$ is a not empty interval for $n<4m+1$ and $1<r\leq 2$.
\end{corollary}
\begin{proof}
Fix $w\in A_2$, $m\in \N$, and $(2_w^*)'s_w<p<\frac{2_w^{2m,*}}{r_w}$. Then, by \eqref{chainpq2pq} and \eqref{2w*M}, we have that
$p_-(L_w)s_w<p<\frac{p_+(L_w)^{2m-1,*}_w}{r_w}$. Consequently, by \eqref{eq:defi:rw},
$$
w\in RH_{\left(\frac{p}{p_-(L_w)}\right)'}\cap A_{\frac{p_+(L_w)^{2m-1,*}_w}{p}},
$$
which in view of \eqref{dualityapclassesRHclasses} yields
$$
w^{-1}\in A_{\frac{p}{p_-(L_w)}}(w)\cap RH_{\left(\frac{p_+(L_w)_w^{2m-1,*}}{p}\right)'}(w).
$$
Hence taking $v=w^{-1}$ in Theorem \ref{thm:boundednessconicalodd}, we conclude that $\Scal_{2m-1,\hh}^w$, $\Scal_{2m-1,\pp}^w$,  $\Grm_{2m-1,\hh}^w$, and $\Grm_{2m-1,\pp}^w$  can be extended to bounded operators on $L^p(\R^n)$.

If now we assume that $r_w=1$ (in particular $w\in A_2$). Then,  for $n=2$ since $2_w^*=\infty$ (see \eqref{2w*}) we have that $2_w^{2m,*}=\infty$ (see \eqref{2w*M}). Thus, the conditions $1<p<\infty$ and
$
w\in A_1\cap RH_{p'},
$
can be written as
$(2_w^*)'s_w<p<\frac{2_w^{2m,*}}{r_w}$. Similarly, in view of \eqref{2w*},  if $n>2$, the conditions $\frac{2n}{n+2}<p<2^{2m,*}_w$ and
$
w\in A_1 \cap RH_{\left(\frac{(n+2)p}{2n}\right)'},
$
can be written as
$$
(2_w^*)'s_w=\frac{2n}{n+2}s_w<
p<2_w^{2m,*}=\frac{2_w^{2m,*}}{r_w}.
$$

Assume next that $r_w>1$, $1<r\leq 2$, and $w\in A_r\cap 
  RH_{\left(\frac{(nr+2)p}{2nr}\right)'}$, in particular, notice  that  $w\in A_2$.
  In the case that $nr\leq 4m$ we take  
  $\frac{2nr}{nr+2}<p<\infty$. Then, since $2_w^{2m,*}=\infty$,  we in fact have
  $$
  (2_w^*)'s_w=\frac{2nr_w}{nr_w+2}s_w< \frac{2nr}{nr+2}s_w<p<\infty=
 \frac{2_w^{2m,*}}{r_w}.$$
As for the case that $nr> 4m$, we take  
  $\frac{2nr}{nr+2}<p<\frac{2n}{nr-4m}$. Thus,  since $r_w<r$, we get
  $$
  (2_w^*)'s_w=\frac{2nr_w}{nr_w+2}s_w< \frac{2nr}{nr+2}s_w<p<\frac{2nr}{r_w(nr-4m)}<
\frac{2nr_w}{r_w(nr_w-4m)}= \frac{2_w^{2m,*}}{r_w}.$$
\end{proof}
\begin{corollary} 
Given $w \in A_2$, for $(2_w^*)'s_w<p<\infty$ we have that $\mathcal{N}_{\hh}^w$  can be extended to a bounded operator on $L^p(\R^n)$. In particular, this holds in the following cases:
 \begin{list}{$(\theenumi)$}{\usecounter{enumi}\leftmargin=1cm \labelwidth=1cm\itemsep=0.2cm\topsep=.0cm \renewcommand{\theenumi}{\alph{enumi}}}
 
 \item If $r_w=1$, for $n=2$, $w\in A_1\cap RH_{p'}$, and  $1<p<\infty$; and for $n>2$, $w\in A_1\cap RH_{\left(\frac{(n+2)p}{2n}\right)'}$, and  $\frac{2n}{n+2}<p<\infty$.
 
  \item If $r_w>1$, for $1\leq r<2$, $w\in A_r\cap RH_{\left(\frac{(nr+2)p}{2nr}\right)'}$, and $\frac{2nr}{nr+2}<p<\infty$.

 \end{list}
\end{corollary}
\begin{proof}
Fix $w\in A_2$ and $(2_w^*)'s_w<p<\infty$. Then, by \eqref{chainpq2pq} and \eqref{2w*M}, we have that
$p_-(L_w)s_w<p<\infty$. Consequently, by \eqref{eq:defi:rw},
$$
w\in RH_{\left(\frac{p}{p_-(L_w)}\right)'},
$$
which in view of \eqref{dualityapclassesRHclasses} yields
$$
w^{-1}\in A_{\frac{p}{p_-(L_w)}}(w).
$$
Hence taking $v=w^{-1}$ in Theorem \ref{thm:boundednessnontangentialheat}, we conclude that $\Ncal_{\hh}^w$  can be extended to a bounded operator on $L^p(\R^n)$.

Assume now that $r_w=1$ (in particular, $w\in A_2$). For $n=2$, since $2_w^*=\infty$ (see \eqref{2w*}), the conditions $1<p<\infty$ and
$
w\in A_1\cap RH_{p'},
$
can be written as
$(2_w^*)'s_w<p<\infty$. If now $n>2$, in view of \eqref{2w*}, the conditions $\frac{2n}{n+2}<p<\infty$ and
$
w\in A_1 \cap RH_{\left(\frac{(n+2)p}{2n}\right)'},
$
can be written as
$(2_w^*)'s_w<p<\infty$.

Assume next that $r_w>1$ and $w\in A_r\cap 
  RH_{\left(\frac{(nr+2)p}{2nr}\right)'}$, for $1<r\leq 2$ and
  $\frac{2nr}{nr+2}<p<\infty$. In particular, $w\in A_2$. Besides, since $r_w<r$ in view on \eqref{2w*}, we can write the previous conditions as
  $$
  (2_w^*)'s_w=\frac{2nr_w}{nr_w+2}s_w< \frac{2nr}{nr+2}s_w<p<\infty.
  $$
 \end{proof}
\begin{corollary} 
Given $w \in A_2$, for $(2_w^*)'s_w<p<\frac{2^{2,*}_w}{r_w}$ we have that $\mathcal{N}_{\pp}^w$  can be extended to a bounded operator on $L^p(\R^n)$. In particular, this holds in the following cases:
 \begin{list}{$(\theenumi)$}{\usecounter{enumi}\leftmargin=1cm \labelwidth=1cm\itemsep=0.2cm\topsep=.0cm \renewcommand{\theenumi}{\alph{enumi}}}
 
 \item If $r_w=1$, for $n=2$, $w\in A_1\cap RH_{p'}$, and  $1<p<\infty$; and for $n>2$, $w\in A_1\cap RH_{\left(\frac{(n+2)p}{2n}\right)'}$, and  $\frac{2n}{n+2}<p<2_w^{2,*}$.
 
  \item If $r_w>1$, for $1\leq r<2$, $w\in A_r\cap RH_{\left(\frac{(nr+2)p}{2nr}\right)'}$, and $\frac{2nr}{nr+2}<p<\infty$, if $nr\leq 4$, or 
   $\frac{2nr}{nr+2}<p<\frac{2n}{nr-4}$, if $nr>4$.
   \end{list}
   Note that the interval $\big(\frac{2nr}{nr+2},\frac{2n}{nr-4}\big)$ is not empty for $n<5$ and $1<r\leq 2$.
\end{corollary}
\begin{proof}
Fix $w\in A_2$ and $(2_w^*)'s_w<p<\frac{2_w^{2,*}}{r_w}$. Then, by \eqref{chainpq2pq} and \eqref{2w*M}, we have that
$p_-(L_w)s_w<p<\frac{p_+(L_w)^{*}_w}{r_w}$. Consequently, by \eqref{eq:defi:rw},
$$
w\in RH_{\left(\frac{p}{p_-(L_w)}\right)'}\cap A_{\frac{p_+(L_w)^{*}_w}{p}},
$$
which in view of \eqref{dualityapclassesRHclasses} yields
$$
w^{-1}\in A_{\frac{p}{p_-(L_w)}}(w)\cap RH_{\left(\frac{p_+(L_w)_w^{*}}{p}\right)'}(w).
$$
Hence taking $v=w^{-1}$ in Theorem \ref{thm:boundednessnon-tangential}, we conclude that $\Ncal_{\pp}^w$  can be extended to a bounded operator on $L^p(\R^n)$.

To obtain parts $(a)$ and $(b)$ we proceed as in the proof Corollary \ref{cor:previous} taking $m=1$.
 \end{proof}
\begin{corollary} \label{cor:previous2}
Given $w \in A_2$, for  $\max\left\{r_w,((2_w^*)')_{w,*}\right\} s_w<p<\frac{2^{*}_w}{r_w}$, we have that $\sqrt{L_w}$
can be extended to a bounded operator on $L^p(\R^n)$.
In particular, this is the case in the following situations:
 \begin{list}{$(\theenumi)$}{\usecounter{enumi}\leftmargin=1cm \labelwidth=1cm\itemsep=0.2cm\topsep=.0cm \renewcommand{\theenumi}{\alph{enumi}}}
  \item For $p=2$,   $1\leq r_w<\frac{2}{n}+1$ and 
  $s_w<\min\left\{\frac{2}{r_w},\frac{nr_w+4}{nr_w}\right\}$ ($(s_w)'>\max\left\{\left(\frac{2}{r_w}\right)',\frac{nr_w}{4}+1\right\}$).
  
  \item If $r_w=1$ and $n=2$ for $w\in A_1\cap RH_{p'}$ and
  $1<p<\infty$; if 
  $r_w=1$ and
  $2< n\leq 4$, for $w\in A_1\cap RH_{p'}$ and
  $1<p<\frac{2n}{n-2}$.
  
  \item  If $r_w=1$ and $n> 4$, for $w\in A_1\cap RH_{\left(\frac{p(n+4)}{2n}\right)'}$ and
  $\frac{2n}{n+4}<p<\frac{2n}{n-2}$.
  
  \item If $r_w>1$, for $1< r\leq 2$, $w\in A_r\cap RH_{\left(\frac{p}{\max\left\{r,2rn/(nr+4)\right\}}\right)'}$ and $\max\left\{r,\frac{2rn}{nr+4}\right\}<p<\frac{2n}{nr-2}$.
  
  \end{list}

\end{corollary}
\begin{proof}
Fix $w\in A_2$ and $\max\left\{r_w,((2_w^*)')_{w,*}\right\}s_w<p<\frac{2_w^{*}}{r_w}$. Then, by \eqref{chainpq2pq} and \eqref{2w*M}, we have that
$\max\left\{r_w,(p_-(L_w))_{w,*}\right\}s_w<p<\frac{p_+(L_w)}{r_w}$. Consequently, by \eqref{eq:defi:rw},
$$
w\in RH_{\left(\frac{p}{\max\left\{r_w,(p_-(L_w))_{w,*}\right\}}\right)'}\cap A_{\frac{p_+(L_w)}{p}},
$$
which in view of \eqref{dualityapclassesRHclasses} yields
$$
w^{-1}\in A_{\frac{p}{\max\left\{r_w,(p_-(L_w))_{w,*}\right\}}}(w)\cap RH_{\left(\frac{p_+(L_w)}{p}\right)'}(w).
$$
Hence taking $v=w^{-1}$ in Theorem \ref{thm:boundednesssaquareroot}, we conclude that $\sqrt{L_w}$  can be extended to a bounded operator on $L^p(\R^n)$.

Now, note that, 
for $p=2$, 
$1\leq r_w<\frac{2}{n}+1$ and 
  $s_w<\min\left\{\frac{2}{r_w},\frac{nr_w+4}{nr_w}\right\}$ imply that
  $$
  \max\left\{r_w, ((2^*_w)')_{w,*}\right\}s_w= \max\left\{r_w, \frac{2nr_w}{nr_w+4}\right\}s_w<2<\frac{2n}{nr_w-2}=\frac{2nr_w}{r_w(nr_w-2)}\leq \frac{2^*_w}{r_w}.
  $$

 To see that $(b)$ and $(c)$ imply that $w\in A_2$ and $\max\left\{r_w,((2_w^*)')_{w,*}\right\} s_w<p<\frac{2^{*}_w}{r_w}$, it is enough to notice that for  $r_w=1$, if $2\leq n\leq 4$  we have that $\max\{r_w,((2_w^*)')_{w,*}\}=r_w=1$. Moreover, by \eqref{2w*}, $2_w^*=\infty$, if $n=2$, and $2_w^*=\frac{2n}{n-2}$, if $n>2$.

  On the other hand, if $r_w=1$ and $n>4$, we have that $\max\{r_w,((2_w^*)')_{w,*}\}=((2_w^*)')_{w,*}=\frac{2n}{n+4}$. Besides, as we have observed above,
 $2n/(n-2)=2^*_w=2^*_w/r_w$.

Finally, part $(d)$ follows  from the following observation: for $1<r_w<r\leq 2$,
$$
\max\left\{r,\frac{2rn}{nr+4}\right\}>
\max\left\{r_w,\frac{2r_wn}{nr_w+4}\right\}=\max\left\{r_w,((2_w^*)')_{w,*}\right\},
$$
and 
$$
\frac{2n}{nr-2}<\frac{2nr_w}{r_w(nr_w-2)}=\frac{2_w^*}{r_w}.
$$
\end{proof}
\begin{remark}
Note that in Corollary \ref{cor:previous2}, part $(a)$, we improve the range obtained in  \cite[Theorem 11.4]{CMR15}. To see this, note that we define $s_w$ as the conjugate exponent of the one defined in  \cite{CMR15} using the same notation.
\end{remark}


\section*{Acknowledgements}
I thank J.M. Martell for some valuable comments.

\bibliographystyle{acm}

\end{document}